\documentclass[pdflatex,sn-mathphys-num]{sn-jnl}

\usepackage{graphicx}%
\usepackage{multirow}%
\usepackage{amsmath,amssymb,amsfonts}%
\usepackage{amsthm}%
\usepackage{mathrsfs}%
\usepackage[title]{appendix}%
\usepackage{xcolor}%
\usepackage{textcomp}%
\usepackage{manyfoot}%
\usepackage{booktabs}%
\usepackage{algorithm}%
\usepackage{algorithmicx}%
\usepackage{algpseudocode}%
\usepackage{listings}
\usepackage{appendix}
\theoremstyle{thmstyleone}%
\newtheorem{theorem}{Theorem}
\newtheorem{proposition}[theorem]{Proposition}%

\theoremstyle{thmstyletwo}%
\newtheorem{example}{Example}%
\newtheorem{remark}{Remark}
\newtheorem{lemma}{Lemma}
\newtheorem{corollary}{Corollary}

\theoremstyle{thmstylethree}%
\newtheorem{definition}{Definition}%

\newcommand{\N}{\mathbb{N}}
\newcommand{\R}{\mathbb{R}}

\newcommand{\mP}{\mathbb{P}}

\newcommand{\cH}{\mathcal{H}}

\newcommand{\cF}{\mathcal{F}}
\newcommand{\cE}{\mathcal{E}}

\newcommand{\cK}{\mathcal{K}}

\newcommand{\cC}{\mathcal{C}}
\newcommand{\cO}{\mathcal{O}}
\newcommand{\cS}{\mathcal{S}}

\newcommand{\cL}{\mathcal{L}}

\newcommand{\id}{\hbox{id}}

\newcommand{\eps}{\varepsilon}
\renewcommand{\phi}{\varphi}

\newcommand{\rT}{\mathrm {T}}

\numberwithin{equation}{section}

\begin{document}

\title[Genericity of Lyapunov spectrum of  bounded random compact operators on  infinite-dimensional Hilbert spaces]{Genericity of Lyapunov spectrum of  bounded random compact operators on  infinite-dimensional Hilbert spaces}


\author*{\fnm{Thai Son} \sur{Doan}}\email{dtson@math.ac.vn}

\affil*{\orgdiv{Institute of Mathematics}, \orgname{Vietnam Academy of Science and Technology}, \orgaddress{\street{18 Hoang Quoc Viet}, \city{Ha Noi}, \country{Viet Nam}}}


\abstract{This paper is devoted to studying the stability of Lyapunov exponents and the simplicity of Lyapunov spectrum for bounded random compact operators on a separable infinite-dimensional Hilbert space from a generic point of view generated by the essential supremum norm. Firstly, we show the density of both the set of bounded random compact operators having a finite number of Lyapunov exponents and the set of bounded random compact operators having a countably infinite number of Lyapunov exponents. Meanwhile, the set of bounded random compact operators having no Lyapunov exponent is nowhere dense. Finally, for any $k\in\N$, we show that the set of bounded random compact operators satisfying the condition that the Lyapunov spectra corresponding to their first $k$ Lyapunov exponents are simple and continuous contains an open and dense set.}




\maketitle

\tableofcontents

\section{Introduction}
The fundamental results on Lyapunov exponents for random dynamical systems on finite-dimensional spaces were first obtained in \cite{Oseledets}, which is now called the Oseledets Multiplicative Ergodic Theorem, see also  \cite{Arnold, Filip}. An extension of this theorem to random dynamical systems on infinite-dimensional Hilbert spaces was established in \cite{Ruelle_1982}. After that, there have been many attempts to develop the theory of Lyapunov exponents for random dynamical systems on a Banach space. Remarkable results in \cite{LianLu2010,Cecilia,Cecilia15} provided a complete result on multiplicative ergodic theorem of random dynamical systems in infinite-dimensional Banach spaces.

Since the appearance of the multiplicative ergodic theorem in \cite{Oseledets}, exploring properties of Lyapunov exponents of random dynamical systems has become one of the central tasks in the theory of random dynamical systems. In this task, understanding the stability of Lyapunov exponents under a small perturbation and the simplicity of Lyapunov spectrum from a generic point of view have been received a lot of interests. These interests come from the role of Lyapunov spectrum in many branches of the qualitative theory of random dynamical systems, such as invariant manifolds theory \cite{Arnold,LianLu2010,LiLuBates14}, linearization theory \cite{LuZhangZhang_2020,Wanner,Zhao}, normal form theory \cite{LiLu05a,LiLu2016,LiLu2008,LiLu05} and bifurcation theory \cite{Doan18,Crauel1998,Engel2019}. 

In this paper,  the stability of Lyapunov exponents and the genericity of Lyapunov spectrum for linear random dynamical systems equipped with the essential supremum norm (  $L^{\infty}$-topology) are investigated. Concerning some literature on other types of stability of Lyapunov exponents, we refer the readers to \cite{Arbieto,BessaVilarinho} for $L^p$ stability, to \cite{Ledrappier,Bogenschutz,Froyland,Froyland19} for stochastic stability. A comprehensive survey on the problem of stability of Lyapunov exponent can be found in \cite{Viana2014,Viana2020}.

\emph{Brief literature on stability of Lyapunov exponent and genericity of Lyapunov spectrum of linear random dynamical systems on finite-dimensional spaces}:  

In general, Lyapunov exponents do not depend continuously on the coefficients, see \cite{Knill}. Later, it was proved in \cite[Theorem C]{Bochi2002} that there is a residual set of $\mbox{SL}(2,\mathbb R)$-cocycles equipped with $C^0$-topology which either are uniformly hyperbolic (having a dominated splitting) or have zero exponents. An extension of this result to arbitrary dimensions was established in \cite{BochiViana2005}.   Nowadays, this result is known as the Ma\~{n}e-Bochi Theorem, see \cite[Section 9.2]{Viana2014}. Several extensions of Ma\~{n}e-Bochi Theorem to continuous-time systems were established in \cite{Bessa2006,Bessa2008a}.  

By generalizing  Millionshchikov's rotation technique, see e.g. \cite{Millionshchikov}, the author in \cite{Cong05} showed the openness and denseness of integral separation of bounded linear random dynamical systems equipped with the $L^{\infty}$-topology. Hence, in the   $L^{\infty}$-topology, the alternative of having zero Lyapunov exponents in Ma\~{n}e-Bochi's Theorem can be dropped. Consequently, simplicity of Lyapunov spectrum is a generic property on the space of bounded linear random dynamical systems. A version of this result for continuous-time systems was established in \cite{Cong_Doan16}.

By showing that an integrally separated linear random dynamical system preserves a suitable
cone under a long enough iteration, the author in \cite{Doan17} combined the work in  \cite{Cong05} and a Perron-Frobenius theorem for linear random dynamical systems  \cite{Ruelle1979,Lian15,MierczynskiShen_01,MierczynskiShen_02,Rugh} to show that the analyticity of Lyapunov exponent is still a generic property.

\emph{Existing literature and the contribution on stability  of Lyapunov exponent and simplicity of Lyapunov spectrum of random compact operators on infinite-dimensional spaces}: 

In contrast to various known results about simplicity and continuity of Lyapunov exponents for random dynamical systems on a finite-dimensional space, only a partial result was developed for the same problems but on an infinite-dimensional Hilbert space. We mentioned here the result obtained in \cite{Bessa2008,Bessa2019}, which can be considered as a generalization of Ma\~{n}e-Bochi's Theorem. Concretely, the authors in  \cite{Bessa2008,Bessa2019} showed that generically either product of random compact operators converges to the null operators or the corresponding Oseledets-Ruelle decomposition is dominated. Our first contribution in this paper is to show that the set of products of random compact operators converging to the null operators is nowhere dense. The main ingredient in the proof of this result is a development of Millionshchikov’s rotation technique on infinite-dimensional Hilbert spaces.  Concerning the remaining possibility of the number of Lyapunov exponents of random compact operators, we show that the set of bounded random compact operators having finite Lyapunov exponents and all of them are simple is dense, and the set of bounded random compact operators having infinite countable numbers of Lyapunov exponents is also dense.

Regarding the aspect of continuity and higher regularity of Lyapunov exponents, our second contribution of the paper is to show the existence of an open and dense set of the space of products of bounded random compact operators such that their first finite number of Lyapunov exponents are simple and depend analytically on the generators. The first ingredient in the proof of this result is a generalization of the Perron–Frobenius theorem and the Krein-Rutman theorem for random compact operators established in \cite{Dubois,Ruelle1979}. The second ingredient of the proof is the density of integral separation for random compact operators. 

The paper is organized as follows: In Section \ref{Section2}, we recall the multiplicative ergodic theorem for non-invertible random compact operators and state the main results of the paper. Section \ref{Prepartoryresults} is devoted to presenting two preparatory results for the proof of the main theorems. The first result is about a generalization of the Krein-Rutman theorem for positive random compact operators (Proposition \ref{GeneralizedKrein-Rutman} in Subsection \ref{Subsection_3.1}). The second result is to present the way to perturb a random compact operator possessing a random unit vector with good growth rate to one satisfying the assumption of the Krein-Rutman theorem (Proposition \ref{TechnicalProposition_C} in Subsection \ref{Subsection_3.1}). The third result is to present a way to perturb a random compact operator converging to the null operator to the new one possessing an invariant unit random vector (Proposition \ref{TechnicalProposition_1} in Subsection \ref{Subsection3.2}). Proof of the main results is given in Section \ref{Section_ProofA} and Section \ref{Section_ProofB}. Several fundamental materials on exterior power, measurable selection theorem of random subspaces, and Halmos-Rokhlin's lemma are collected in the Appendix.  
\section{Preliminaries and the statement of the main results}\label{Section2}
Let $\cH$ be an infinite-dimensional separable Hilbert space. Denote by $\cK(\cH)$ the space of compact operators from $\cH$ into itself. Let $(\Omega,\cF,\mP)$ be a non-atomic Lebesgue probability space and $\theta:\Omega\rightarrow \Omega$ an invertible ergodic transformation preserving the probability $\mP$.  Denote by $\mathcal K_{\infty}(\Omega;\mathcal H)$ the space of strongly measurable and essentially bounded maps\footnote{The map $T:\Omega\rightarrow \cK(\cH)$ is called \emph{strongly measurable} if for any $x\in\cH$ the map from $\Omega$ to $\cH$ defined by  $\omega\mapsto  T(\omega) x$ is measurable, see e.g. \cite{LianLu2010}}  $T:\Omega\rightarrow \mathcal K(\mathcal H)$ endowed with the $L^{\infty}$-norm
\[
\|T\|_{\infty}:=\mbox{ess}\sup_{\omega\in\Omega} \|T(\omega)\|.
\]
Then, the space $(\mathcal K_{\infty}(\Omega;\mathcal H), \|\cdot\|_{\infty})$ is a Banach space and each $T\in \mathcal K_{\infty}(\Omega;\mathcal H)$ gives rise to an one-sided linear random dynamical system $T:\N\times \Omega\rightarrow  \mathcal K(\mathcal H)$ via
\begin{equation}\label{GeneratedLinearRDS}
T(n,\omega)
:=
\left\{
  \begin{array}{ll}
    \id , & \hbox{if } n=0, \\[1.5ex]
    T(\theta^{n-1}\omega)\dots T(\omega), & \hbox{else.}
  \end{array}
\right.
\end{equation}
For short, we use the notation $T_{\omega}^n$ instead of $T(n,\omega)$ and we also identify each random compact operator $T$ with its generated linear random dynamical system $T^n_{\omega}$. Thanks to work in \cite{Ruelle_1982,Cecilia,Cecilia15}, the following multiplicative ergodic theorem provides the asymptotic behavior of random compact operators over an invertible ergodic system. Note that these random compact operators do not need to be assumed to be invertible. 
\begin{theorem}[Multiplicative ergodic theorem for non-invertible random compact  operators]\label{MET_HilbertSpace} Let $T\in \mathcal K_{\infty}(\Omega;\mathcal H)$. Then, there exists a measurable set $	\widehat \Omega$ being of full probability and invariant under $\theta$ such that the following statements hold:

\noindent 
(i) \textbf{Lyapunov exponents}: For every $\omega\in\widehat \Omega$, the limit $\lim_{n\to\infty} \big((T^n_{\omega})^{\rT} T^n_{\omega}\big)^{\frac{1}{2n}}=\Lambda_{\omega}(T)$ exists and is a compact operator. Let $e^{\lambda_1(T)}>e^{\lambda_2(T)}>\dots$ denote the nonzero eigenvalues of $\Lambda_{\omega}(T)$. Then, 
$\lambda_k(T)$ are real and independent of $\omega$ and might terminate at $\lambda_{s(T)}(T)$; otherwise we write $s(T)=\infty$  and in this case $\lim_{k\to\infty}\lambda_k(T)=-\infty$. 

\noindent (ii) \textbf{Oseledets-Ruelle decomposition}: Let $F^{\infty}_{\omega}(T)$ be the null space of $\Lambda_{\omega}(T)$, if it exists. Then, $
	\lim_{n\to\infty}\frac{1}{n}\log\|T^n_{\omega}|_{F_{\omega}^{\infty}(T)}\|=
		-\infty$. There exists an equivariant decomposition 
	\[
	\mathcal H=  \bigoplus_{i=1}^{s(T)} \mathcal O^i_{\omega}(T) \oplus F^{\infty}_{\omega}(T),
	\]
where the Oseledets-Ruelle subspaces $\mathcal O^i_{\omega}(T), i=1,\dots,s,$ are of finite dimension and equivariant, i.e. $T(\omega) \mathcal O^i_{\omega}(T)\subset \mathcal O^i_{\theta\omega}(T)$, and
	\[
	\lim_{n\to\infty}\frac{1}{n}\log\|T^n_{\omega}v\|=\lambda_i\qquad\hbox{for } v\in\mathcal O^i_{\omega}(T)\setminus\{0\}.
	\]
\end{theorem}
According to the above theorem, three possibilities for the number of Lyapunov exponents of a random compact operator are finite numbers, infinitely countable numbers, or no number, see also \cite{LianLu2010}. Our first result is to show the density of the set of random compact operators having a finite number of Lyapunov exponents and all of them are simple, i.e., the Oseledets-Ruelle subspaces corresponding to all finite Lyapunov exponents are of dimension one.  Next, we also show the density of the set of random compact operators having infinitely many Lyapunov exponents. Concerning the remaining possibility, we show the nowhere density of random compact operators having no finite Lyapunov exponents. These results are summarized in the following theorem.\\

\noindent
\textbf{Theorem A (Genericity on the number of Lyapunov exponents).}  Define the following subsets of $ \mathcal K_{\infty}(\Omega;\mathcal H)$ 
\begin{align}
\mathcal S
& :=\Big\{T\in \mathcal K_{\infty}(\Omega;\mathcal H): s(T)<\infty\; \& \dim \mathcal O^k_{\omega}(T)=1\hbox{ for } k=1,\dots, s(T) \Big\}, \label{SimpleLE}\\[1ex]
\mathcal I
& :=\Big\{T\in \mathcal K_{\infty}(\Omega;\mathcal H): s(T)=\infty\Big\}\label{InfinteLE},\\[1ex]
\mathcal N
&:=\big\{ T\in \mathcal K_{\infty}(\Omega;\mathcal H): \Lambda_{\omega}(T)=0\quad \hbox{for } \mP-\hbox{a.e. } \omega\in\Omega\big\}.\label{Nullset}
\end{align}
Then, the following statements hold
\begin{itemize}
    \item [(i)] The set $\mathcal S$ is dense in $(\mathcal K_{\infty}(\Omega;\mathcal H),\|\cdot\|_{\infty})$.
    \item [(ii)] The set $\mathcal I$ is dense in $(\mathcal K_{\infty}(\Omega;\mathcal H),\|\cdot\|_{\infty})$.
    \item [(iii)] The set $\mathcal N$ is nowhere dense in $\mathcal K_{\infty}(\Omega;\mathcal H)$.
\end{itemize}
Next, we are interested in the simplicity and the continuity of the Lyapunov exponents of bounded random compact operators. More concretely, for each $k\in\N$  we show the genericity of the simplicity and analyticity of the first $k$ Lyapunov exponents.\\ 

\noindent
\textbf{Theorem B (Genericity on simplicity and analyticity of Lyapunov exponents).} For each $k\in\N$, there exists an open and dense set $\mathcal G_k\subset \mathcal K_{\infty}(\Omega;\mathcal H)$ such that for all $T\in \mathcal G_k$ we have $s(T)\geq k$ and the following statements hold
\begin{itemize}
\item [(i)] The Oseledets-Ruelle subspaces corresponding to the first $k$ Lyapunov exponent of $T$ are simple, i.e.
\[
\dim \mathcal O^1_{\omega}(T)=\dots=\dim \mathcal O^k_{\omega}(T)=1.
\]
\item [(ii)] For each $i=1,\dots,k$, the map $\lambda_i:\mathcal G_k\rightarrow \R, T\mapsto \lambda_i(T)$ is analytic.
\end{itemize}

\section{Preparatory results}\label{Prepartoryresults}
\subsection{Krein-Rutman theorem for random compact operators}\label{Subsection_3.1}
A subset $\mathcal C\subset \cH$ is called a \emph{proper cone} if $tv\in\mathcal C$ for all $t\geq 0, v\in\mathcal C$ and $\mathcal C\cap (-\mathcal C)=\{0\}$. For a closed proper convex cone $\mathcal C$, the dual cone $\mathcal C^{'}$ is defined by
\begin{equation}\label{Dualcone}
\mathcal C^{'}:=\{u\in\cH: \langle u,v\rangle\geq 0\quad\hbox{for all } v\in\mathcal C \}.
\end{equation}
A family of closed proper convex cones  $(\mathcal C_{\omega})_{\omega\in \Omega}$ is assumed to satisfies an additional condition about the existence of unit interior points of these cones and their dual cones:
\begin{itemize}
  \item [(C)] There are measurable mappings $c,c':\Omega\rightarrow \cH$ such that $\|c(\omega)\|=\|c'(\omega)\|=1$ and $B(c(\omega),r) \subset \mathcal C_{\omega}$, $B(c'(\omega),r) \subset \mathcal C'_{\omega}$
 with $r$ is independent of $\omega$.
\end{itemize}
In the sequel, we collect examples of a family of closed proper convex cones $(\cC_{\omega})_{\omega\in\Omega}$ satisfying condition (C) used throughout the paper.
\begin{example}\label{ConeExample}
(i) Let $(\cC_{\omega})_{\omega\in\Omega}$ be a family of closed proper convex cones satisfying condition (C).  Denote by $\cL(\cH)$ the space of bounded operators from $\cH$ into itself. Let $P:\Omega\rightarrow \cL(\cH)$ be a strongly measurable map such that $P(\cdot)$ is essentially bounded. Define 
\[
\widetilde{\cC}_{\omega}=P(\omega)\cC_{\omega}\quad \hbox{for } \omega\in\Omega.
\]
Then, $(\widetilde{\cC}_{\omega})_{\omega\in\Omega}$ is also a family of closed proper convex cones satisfying condition (C).

\noindent
(ii) Let $e:\omega\rightarrow \cH$ be a random unit variable. For each $\omega\in\Omega$, let 
\begin{equation}\label{Randomcone_Example}
\cC_{\omega}:=\left\{v\in \cH:\langle e(\omega),v\rangle \geq\frac{\|v\|}{2}\right\}.
\end{equation}
Obviously,  $(\cC_{\omega})_{\omega\in\Omega}$ is a family of proper closed convex cones. We show that $(\cC_{\omega})_{\omega\in\Omega}$ also satisfies condition (C) with the interior unit vector $e(\omega)$ by verifying that $B(e(\omega),\frac{1}{6})\subset \cC_{\omega}$ and $B(e(\omega),\frac{1}{6})\subset \cC'_{\omega}$. Let $v\in B(e(\omega),\frac{1}{6})$ be arbitrary. Then, $v=e(\omega)+x$ for some $x\in \cH$ with $\|x\|\leq \frac{1}{6}$. Thus, $\|v\|\leq \frac{7}{6}$ and 
\[
\langle e(\omega),v\rangle=1+ \langle x,v\rangle\geq 1-\|x\|\|v\|\geq \frac{29}{36},
\]
which implies that $\langle e(\omega),v\rangle \geq\frac{\|v\|}{2}$ and therefore $B(e(\omega),\frac{1}{6})\subset \cC_{\omega}$. To show  $B(e(\omega),\frac{1}{6})\subset \cC'_{\omega}$, pick an arbitrary $u\in \cC_{\omega}$ with $\|u\|=1$. Then, 
\[
\langle v,u\rangle=\langle e(\omega),u\rangle+ \langle x,u\rangle 
\geq \frac{\|u\|}{2}-\|x\|\|u\|>0,
\]
which together with the definition of the dual cone $\cC'_{\omega}$ as in \eqref{Dualcone} implies that $v\in \cC'_{\omega}$. Consequently, $B(e(\omega),\frac{1}{6})\subset \cC'_{\omega}$.

\noindent 
(iii) Let $\cH=\cO(\omega) \oplus F(\omega)$ be a random decomposition of $\cH$ with $\dim \cO(\omega)=1$. Let $e(\omega)$ be the random unit vector of $\cO(\omega)$. Suppose that there exists $\eta\in (0,1)$ such that 
\begin{equation}\label{AngleSeparation}
\frac{|\langle u,e(\omega)\rangle|}{\|u\|}\leq \eta \qquad\hbox{ for all } u\in F(\omega)\setminus\{0\}. \end{equation}
The above condition means a uniform separation from zero of the angle between two vectors in $\cO(\omega)$  and $F(\omega)$. This condition also implies the following estimate on the decomposition of an arbitrary unit vector in $\cH$.  Let $\xi\in\cH$ be an arbitrary unit vector. Write $\xi=v+u$, where $v\in\cO(\omega)$ and $u\in F(\omega)$. Then, using \eqref{AngleSeparation} we obtain that 
\begin{equation}\label{AngleSeparation_Estimate01}
\|v\|, \|u\|\leq \frac{1}{\sqrt{1-\eta^2}}.     
\end{equation}
For each $\omega\in\Omega$, let 
\[
\cC_{\omega}:=\left\{ u+\gamma e(\omega): u\in F(\omega)\quad \hbox{ with } \quad  \|u\|\leq \frac{\gamma}{2} \right\}.
 \]
Obviously, $(\mathcal C_{\omega})_{\omega\in\Omega}$ is a family of closed proper convex cones. We are proving that  $(\mathcal C_{\omega})_{\omega\in \Omega}$ satisfies condition (C) with the unit interior point $c(\omega)=c'(\omega):=e(\omega)$. To see this, let
\begin{equation}\label{Definitionofinteriorvector}
r:=\frac{\sqrt{1-\eta^2}}{3 }.   
\end{equation}
We verify that $B(e(\omega),r)\subset \cC_{\omega}$ and $B(e(\omega),r)\subset \cC'_{\omega}$. For this purpose, choose and fix an arbitrary $\beta\in (0,r)$ and $\xi\in \cH$ with $\|\xi\|=1$. Write $\xi=v+u$ with $v\in\cO(\omega)$ and $u\in F(\omega)$. Thus, 
\[
e(\omega)+\beta \xi=
\beta u+ (e(\omega)+\beta v).
\]
Note that $e(\omega)+\beta v$ is either $(1+\beta\|v\|) e(\omega)$ or $(1-\beta\|v\|)e(\omega)$. In both cases, by \eqref{AngleSeparation_Estimate01}, \eqref{Definitionofinteriorvector} and $\beta \in (0,r)$ we have \[
\|\beta u\| \leq \frac{r}{\sqrt{1-\eta^2}}
\leq \frac{1-\frac{r}{\sqrt{1-\eta^2}}}{2}
\leq \frac{1-\beta\|v\|}{2},
\]
which implies that  $e(\omega)+\beta \xi\in \cC_{\omega}$ and therefore $B(e(\omega),r)\subset \cC_{\omega}$. To see $e(\omega)+\beta \xi\in \cC'_{\omega}$, let $u+\gamma e(\omega)\in \cC_{\omega}$ with $u\in F(\omega)$ and $\|u\|\leq \frac{\gamma}{2}$. Then, by \eqref{Definitionofinteriorvector}
\begin{align*}
\langle e(\omega)+\beta \xi, u+\gamma e(\omega)\rangle
& \geq 
\gamma (1+ \beta \langle \xi,e(\omega)\rangle)
-\|u\|\|e(\omega)+\beta \xi\|\\[1ex]
&\geq 
\gamma (1-r)-\frac{\gamma}{2}(1+r)\geq 0,
\end{align*}
which indicates that $e(\omega)+\beta \xi\in \cC'_{\omega}$. Thus, $B(e(\omega),r)\subset \cC'_{\omega}$.
\end{example}
In what follows, we recall the result from \cite{Dubois} and its slight improvement, see \cite[Theorem 8]{Doan17}, about simplicity and analyticity of the top exponent for a random compact operator having an invariant random cone. To state this result, recall that the \emph{top Lyapunov exponent function} $\kappa: \cK_{\infty}(\Omega;\mathcal H)\rightarrow \R\cup\{-\infty\}$ is defined as 
\begin{equation}\label{TopLyapunovExponent}
\kappa(T):=\lim_{n\to\infty}\frac{1}{n}\log\|T_{\omega}^n\|\qquad\hbox{for}\quad  \mP-\hbox{a.e. } \omega\in\Omega.
\end{equation}
The \emph{Hilbert metric} $d_{\mathcal C}$ on $\left(\mathcal C\cup -\mathcal C\right)\setminus\{0\}$ is defined by
\begin{equation}\label{HilbertMetric}
d_{\mathcal C}(u,v):=\log (\beta_{\mathcal C}(u,v))+\log (\beta_{\mathcal C}(v,u))\qquad\hbox{for } u,v\in \left(\mathcal C\cup -\mathcal C\right)\setminus\{0\},
\end{equation}
with
\begin{equation}\label{Distancebeta}
\beta_{\mathcal C}(u,v):=\inf\{t>0: t|u|-|v|\in\mathcal C\},
\end{equation}
where for any $w\in (\cC\cup -\cC)$ we let $|w|:=
\left\{
\begin{array}{ll}
    w & \hbox{ if } w\in \cC\setminus\{0\}, \\[1ex]
    -w &  \hbox{ if } w\in -\cC \setminus\{0\},
\end{array}
\right.
$.
\begin{theorem}[Krein-Rutman theorem for random compact  operators]\label{Krein-Rutman}
Consider a random compact operator $T\in \mathcal K_{\infty}(\Omega;\mathcal H)$. Let  $(\mathcal C_{\omega})_{\omega\in \Omega}$ be a family of closed proper convex cones satisfying the condition (C). Suppose that $T(\omega)\mathcal C_{\omega}\subset \left(\mathcal C_{\theta \omega}\cup -\mathcal C_{\theta \omega}\right)$ and there exists $R\in (0,\infty)$ such that
\begin{equation}\label{Interiorcone_condition}
d_{\mathcal C_{\theta\omega}}(T(\omega)v, c(\theta\omega))\leq R\qquad\hbox{for } \omega\in \Omega, v\in \mathcal C_{\omega}\setminus\{0\}.
\end{equation}
Then, the following statements hold:
\begin{itemize}
\item [(i)] The Oseledets-Ruelle subspace of $T$ corresponding to the top Lyapunov exponent $\kappa(T)$ is of dimension one,
\item [(ii)] The top Lyapunov exponent function $\kappa(\cdot)$ is analytic at $T$.
\end{itemize}
\end{theorem}
\begin{remark}\label{Explaination_interiorcone}
 Condition \eqref{Interiorcone_condition} implies that $T(\omega)$ maps the cone $\cC_{\omega}$ into the interior of $\cC_{\theta\omega}$. For an easier task in checking this condition later, condition \eqref{Interiorcone_condition} is replaced by a slightly stronger condition
\begin{equation}\label{Strongercondition}
\beta_{\mathcal C_{\theta\omega}}(T(\omega)v, c(\theta\omega))< R \qquad\hbox{for } \omega\in \Omega, v\in \mathcal C_{\omega} \hbox{with } \|v\|=1,   
\end{equation}
where $R>0$ is a positive constant. To see that condition \eqref{Strongercondition} implies \eqref{Interiorcone_condition}, choose and fix  arbitrary $\omega\in \Omega, v\in \mathcal C_{\omega}\setminus\{0\}$ with  $\|v\|=1$. Using property (C) of cone $\cC_{\theta\omega}$ and the fact that $\left\|\frac{T(\omega)v}{\|T\|_{\infty}}\right\|\leq 1$, we obtain that \[
c(\theta\omega)-\frac{r}{2}\frac{T(\omega)v}{\|T\|_{\infty}}\in \cC_{\theta\omega},
\]
which implies that $\frac{2\|T\|_{\infty}}{r}c(\theta\omega)-T(\omega)v\in \cC_{\theta\omega}$. Thus, 
\begin{equation}\label{InteriorCone_02}
\log\beta_{\mathcal C_{\theta\omega}}(c(\theta\omega),T(\omega)v)
\leq 
\log \frac{2\|T\|_{\infty}}{r},
\end{equation}
which together with \eqref{Strongercondition} implies that 
\begin{align*}
d_{\cC_{\theta\omega}}
(T(\omega)v, c(\theta\omega))
&= \log \beta_{\mathcal C_{\theta\omega}}(c(\theta\omega),T(\omega)v) + \log (\beta_{\mathcal C_{\theta\omega}}(T(\omega)v), c(\theta\omega))\\
&< 
\log R+ \log  \frac{2\|T\|_{\infty}}{r}.
\end{align*}
It means that \eqref{Interiorcone_condition} holds with the positive constant $\log \frac{2R\|T\|_{\infty}}{r}$.
\end{remark}
In the following result, Theorem \ref{Krein-Rutman} is strengthened in two aspects. The first improved point is to replace the condition \eqref{Strongercondition} on invariance of the cone under one iteration of random maps by the condition on invariance of the cone under a finite iteration of random maps, see also \cite[Theorem 4]{Doan17} for a similar argument. The second improved point is that the conclusion of Theorem \ref{Krein-Rutman} is not only true for a random compact operator having an invariance cone but also true for nearby random compact operators.

\begin{proposition}[Generalized Krein-Rutman theorem for random compact operators]\label{GeneralizedKrein-Rutman}
 Consider a random compact operator $T\in \mathcal K_{\infty}(\Omega;\mathcal H)$. Let $(\mathcal C_{\omega})_{\omega\in\Omega}$ be a family of closed proper convex cones satisfying the condition (C). Suppose that $T$ satisfies the following assumptions 
 \begin{itemize}
     \item [(A1)] There exists a prime number $N$ such that 
     \[
     T^N_{\omega}\mathcal C_{\omega}\subset \left(\mathcal C_{\theta^N \omega} \cup - \mathcal C_{\theta^N \omega}\right)\qquad\hbox{ for } \omega\in\Omega.
     \]
     \item [(A2)] There exists a positive constant $R>0$ such that \begin{equation}\label{InvarianceCone_Condition}
\beta_{\mathcal C_{\theta^N\omega}}(T^N_{\omega}v, c(\theta^N\omega))< R\quad \hbox{ for } \omega\in\Omega, v\in \mathcal C_{\omega}\setminus\{0\} \hbox{ with } \|v\|=1.
\end{equation}
 \end{itemize}
 Then, there exists $\eps>0$ such that for all $S\in \mathcal K_{\infty}(\Omega;\mathcal H)$ with $\|S-T\|_{\infty}\leq \eps$ the following statements hold
\begin{itemize}
\item [(i)] The Oseledets-Ruelle subspace of $S$ corresponding to the top Lyapunov exponent $\kappa(S)$ is of dimensional one,
\item [(ii)] The top Lyapunov exponent function $\kappa(\cdot)$ is analytic at $S$.
\end{itemize}
\end{proposition}
\begin{proof}
Firstly, we show the assertions for the original random compact operator $T$ under an additional assumption that  $\theta^N$ is an ergodic transformation from $\Omega$ into itself. Define a map $\pi: \mathcal K_{\infty}(\Omega,\cH)\rightarrow \mathcal K_{\infty}(\Omega,\cH)$ by
\begin{equation}\label{DefinitionPi}
\pi(S)(\omega):=S(\theta^{N-1}\omega)\dots S(\omega).
\end{equation}
Note that $\pi$ is a product of bounded linear operators $\pi_i:\mathcal K_{\infty}(\Omega,\cH)\rightarrow \mathcal K_{\infty}(\Omega,\cH), S\mapsto \pi_i(S)$ with $\pi_i(S)(\omega):=S(\theta^i\omega)$ and $i=1,\dots, N-1$ and thus the function $\pi$ is analytic. For each $S\in \mathcal K_{\infty}(\Omega,\cH)$, by ergodicity of $\theta^N$ and Theorem \ref{MET_HilbertSpace}  the top Lyapunov exponent $\kappa^{(N)}(\pi(S))$ of  the compact operator  $\pi(S)\in \mathcal K_{\infty}(\Omega,\mathcal B)$ over $(\Omega,\cF,\mP,\theta^N)$ is given by 
\[
\kappa^{(N)}(\pi(S))
=
\lim_{n\to\infty}\frac{1}{n}\log \|\pi(S)(\theta^{(n-1)N}\omega)\dots\pi(S)(\theta^N\omega)\pi(S)(\omega)\|,
\]
this together with  \eqref{DefinitionPi} implies that
\begin{equation}\label{Relation_TopEx}
\lambda_1(S)=\frac{1}{N}\lambda_1^{(N)}(\pi(S)).
\end{equation}
By \eqref{InvarianceCone_Condition} the random compact operator $\pi(T)$ on $(\Omega,\cF,\mP,\theta^N)$ preserves the family of closed proper convex cones $(\cC_{\omega})_{\omega\in\Omega}$. Using Theorem \ref{Krein-Rutman} and Remark \ref{Explaination_interiorcone}, the function $\kappa^{(N)}(\cdot)$ is analytic at $\pi(T)$ and this together with \eqref{Relation_TopEx} and analyticity of the function $\pi$ implies that $\kappa(\cdot)$ is analytic at $T$. Additionally, by Theorem \ref{Krein-Rutman} and Remark \ref{Explaination_interiorcone} the dimension of the Oseledets-Ruelle subspace corresponding to the top Lyapunov exponent of $\pi(T)$ is $1$. Therefore, the  Oseledets-Ruelle subspace corresponding to the top Lyapunov exponent of $T$ is also  $1$. Hence, the conclusion of the proposition holds for $T$ in the case that $\theta^N$ is ergodic.

Now we deal with the case that $\theta^N$ is not ergodic. Let $W$ be a measurable set satisfying that
$\mP(W)\in (0,1)$ and  $W=\theta^N W$. Since $N$ is a prime number, there exists a measurable set  $\widehat W$ such that $\theta^N \widehat W=\widehat W$ and the restriction of the map $\theta^N$ on
 $\widehat W$ is ergodic, see \cite[Proof of Theorem 4]{Doan17}. Thus, using a similar statement as in the case that $\theta^N$ is ergodic on $\Omega$ to the induced random compact operator $T$ on $(\widehat W,\mathcal F,\mP,\theta^N)$ yields the assertions for $T$ in this case.

To conclude the proof, it is sufficient to show that there exists $\eps>0$ such that all $S\in \mathcal K_{\infty}(\Omega,\cH)$ with $\|S-T\|_{\infty}\leq \eps$ fulfill \eqref{InvarianceCone_Condition}, i.e. $S^N_{\omega}\mathcal C_{\omega}\subset \mathcal C_{\theta^N \omega}$ and there exists $K>0$ such that 
\begin{equation}\label{InvarianceCone_Condition_New}
\beta_{\mathcal C_{\theta^N\omega}}(S^N_{\omega}v, c(\theta^N\omega))\leq K\quad \hbox{ for } \omega\in\Omega, v\in \mathcal C_{\omega}\setminus\{0\} \hbox{ with } \|v\|=1.
\end{equation}
For this purpose, we show that   for all $v\in\cC_{\omega}$ with $\|v\|=1$ and $u\in\cC'_{\theta^N\omega}$ with $\|u\|=1$
\begin{equation}\label{InteriorCone_Aim}
\langle u, T^N_{\omega}v\rangle \geq \frac{r}{2R}\qquad\hbox{ for } \omega\in\Omega.   
\end{equation}
Choose and fix $\omega\in \Omega$ and $v\in \mathcal C_{\omega}$ with $\|v\|=1$ and $u\in\cC'_{\theta^N\omega}$ with $\|u\|=1$. By \eqref{InvarianceCone_Condition} we derive that $R\,T^N_{\omega}v-c(\theta^N\omega)\in \cC_{\theta^N\omega}$ and thus $\langle u,R\,T^N_{\omega}v-c(\theta^N\omega)\geq 0$. So,
\begin{equation}\label{Firstestimate_New}
 \langle u, T^N_{\omega}v\rangle\geq \frac{\langle u, c(\theta^N\omega)\rangle}{R}.
\end{equation}
By (C) and the fact that $\|\frac{r}{2}u\|=\frac{r}{2}$, we have $c(\theta^N\omega)-\frac{r}{2} u\in \cC_{\theta^N\omega}$. Thus, $\langle u,c(\theta^N\omega)-\frac{r}{2}u\rangle\geq 0$. This together with \eqref{Firstestimate_New} implies that $\langle u, T^N_{\omega}v\rangle\geq\frac{r}{2R}$, i.e. \eqref{InteriorCone_Aim} is proved. By continuity argument, there exists $\eps>0$ such that all $S\in \mathcal K_{\infty}(\Omega,\cH)$ with $\|S-T\|_{\infty}\leq \eps$ satisfy that 
\begin{equation}\label{ChoiceofEpsilon}
\|T^N_{\omega}-S^N_{\omega}\|\leq \frac{r}{4R}\quad \hbox{for all }\quad  \omega\in\Omega
\end{equation}
and we show that this choice of $\eps$ satisfies \eqref{InvarianceCone_Condition_New}. For this purpose, let $S\in \mathcal K_{\infty}(\Omega,\cH)$ be arbitrary but fixed with $\|S-T\|_{\infty}\leq \eps$. Let $\omega\in \Omega$ and $v\in \mathcal C_{\omega}$ with $\|v\|=1$. Then, for all $u\in \cC'_{\theta^N\omega}$ 
\[
\langle u,S^N_{\omega}v\rangle \geq 
\langle u,T^N_{\omega}v\rangle-\left|\langle u,(S^N_\omega-T^N_{\omega})v\rangle \right|\geq \frac{r}{2R}-\frac{r}{4R}>0.
\]
It means that $S^N_{\omega}v$ lies on the dual cone of $ \cC'_{\theta^N\omega}$ and therefore $S^N_{\omega}v\in \cC_{\theta^N\omega}$. Furthermore, for any $t\geq \frac{4R}{r}$ and 
$u\in \cC'_{\theta^N\omega}$ with $\|u\|=1$ we have 
\[
\langle u, t S^N_{\omega}v-c(\theta^N\omega)\rangle
=
\langle u, t\; T^N_{\omega}v\rangle + \langle u, t\; (S^N_{\omega}-T^N_{\omega})v\rangle - \langle u,c(\theta^N\omega)\rangle.
\]
This together with \eqref{InteriorCone_Aim} and \eqref{ChoiceofEpsilon} implies that 
\begin{align*}
\langle u, t S^N_{\omega}v-c(\theta^N\omega)\rangle
& \geq 
t \langle u, T^N_{\omega}v\rangle
-t \left|\langle u,  (S^N_{\omega}-T^N_{\omega})v\rangle\right|-\left| \langle u,c(\theta^N\omega)\rangle \right|\\
& \geq 
t\frac{r}{2R}- t \frac{r}{4R}-1\\
& \geq 0.
\end{align*}
Consequently, for all $t\geq \frac{4R}{r}$ we have 
$tS^N_{\omega}v-c(\theta^N\omega)\in \cC_{\theta^N\omega}$ and thus 
\[
\beta_{\cC_{\theta^N\omega}}(S^N_{\omega}v,c(\theta^N\omega)) \leq \frac{4R}{r}.
\]
Hence, \eqref{InvarianceCone_Condition_New} holds for $K=\frac{4R}{r}$ and the proof is complete.
\end{proof}
We conclude this preparatory subsection by establishing a result on how to perturb a  random compact operator possessing a random unit vector with good growth rate to one possessing an invariant family of closed proper convex cones satisfying all assumptions of Proposition \ref{GeneralizedKrein-Rutman}. Random compact operators possessing a random unit vector with good growth rate will appear later as arbitrarily small perturbation of random compact operators converging to the null operator, see Proposition \ref{Millionschikov}. It is also noted that linear random dynamical systems having one invariant direction with a good growth rate play an important role in proving the density of integrally separated bounded linear random dynamical systems, see \cite[Theorem 6]{Cong05}.
\begin{proposition}\label{TechnicalProposition_C}
 Let $T\in \mathcal K_{\infty}(\Omega;\mathcal H)$ and $e(\omega)$ a random unit vector of $\cH$. Suppose that $T(\omega)e(\omega)$ is collinear to $e(\theta\omega)$ and there exist  $N\in \N$ and $\rho,\alpha>0$ such that  \begin{equation}\label{dominateddirection}
\|T(\omega)e(\omega)\|\geq \rho,\; \|T_{\omega}^N e(\omega)\|\geq e^{-\alpha N}\|T_{\omega}^N\|\qquad\hbox{ for  }  \omega\in\Omega.   
\end{equation}
Define $S\in \mathcal K_{\infty}(\Omega;\mathcal H)$ by for each $\omega\in\Omega$ 
\begin{equation}\label{Perturbed_Operator}
S(\omega) v
=
\left\{
\begin{array}{ll}
   e^{2\alpha }\, T(\omega)e(\omega)  & \hbox{ if } v=e(\omega),  \\[1ex]
   T(\omega)v     &  \hbox{ if } v\in e(\omega)^{\perp}.
\end{array}
\right.
\end{equation}
Then, the following statements hold
\begin{itemize}
    \item [(i)] $\|S-T\|_{\infty}
    \leq (e^{2\alpha }-1)\|T\|_{\infty}$,
    \item [(ii)] The random compact operator $S$ satisfies all assumptions of Proposition \ref{GeneralizedKrein-Rutman}, i.e. there exist  a family of closed proper convex cones $(\mathcal C_{\omega})_{\omega\in \Omega}$ satisfying the condition (C),  a prime number $N\in\mathbb N$  such that $S^N_{\omega}\mathcal C_{\omega}\subset \mathcal C_{\theta^N \omega}$ and  a positive constant $R>0$ such that 
\[
\beta_{\mathcal C_{\theta^N\omega}}(S^N_{\omega}v, c(\theta^N\omega))< R\quad \hbox{ for } \omega\in\Omega, v\in \mathcal C_{\omega}\setminus\{0\} \hbox{ with } \|v\|=1.
\]
Consequently, the top Lyapunov exponent function $\kappa(\cdot)$ is analytic at $S$.
\end{itemize}
\end{proposition}
In the above proposition, there is no assumption that $e(\omega)^{\perp}$ is invariant under $S(\omega)$. This issue makes it difficult to estimate the growth of the vectors starting from $e(\omega)^{\perp}$ under iteration of $S(\omega)$. In order to overcome this difficulty, $S(\omega)$ is transformed to the new random linear compact $S_{\eps}(\omega)$ for which $e(\omega)^{\perp}$ is almost invariant and the problem of existence of invariant cone for $S(\omega)$ is shifted to the existence of invariant cone for the new one. To define $S_{\eps}(\omega)$, let $\Pi_e(\omega)$ denote the projection onto $e(\omega)$ along $e(\omega)^{\perp}$. For any $\eps\in [0,1)$, let $S_{\eps}\in \mathcal K_{\infty}(\Omega;\mathcal H)$ be given as  
\begin{equation}\label{Perturbed_Operator_variation}
S_{\eps}(\omega) v
=
\left\{
\begin{array}{ll}
   S(\omega)e(\omega)  & \hbox{ if } v=e(\omega),  \\[1ex]
  \eps\; \Pi_e(\theta\omega) S(\omega)v +(\id-\Pi_e(\theta\omega)) S(\omega)v    &  \hbox{ if } v\in e(\omega)^{\perp}.
\end{array}
\right.
\end{equation}
To see a relation between $S_{\eps}(\omega)$ and $S(\omega)$, we define a linear operator $P_\eps(\omega):\cH\rightarrow \cH$ by 
\[
P_{\eps}(\omega) v=\left\{
\begin{array}{ll}
    \eps e(\omega) & \hbox{ if } v=e(\omega), \\[1ex]
      v& \hbox{ if } v\in e(\omega)^{\perp}. 
\end{array}
\right.
\]
Then, $P_{\eps}(\omega)$ and its inverse $P_{\eps}^{-1}(\omega)$ are bounded linear operators. In fact, $\|P_{\eps}(\omega)\|=1$ and $\|P_{\eps}^{-1}(\omega)\|=\frac{1}{\eps}$. By \eqref{Perturbed_Operator_variation} and the fact that $S(\omega)e(\omega)$ is collinear to $e(\theta\omega)$, we have
\begin{equation}\label{Relation_01}
P_{\eps}(\theta\omega)^{-1} S_{\eps}(\omega) P_{\eps}(\omega)
=
S(\omega).
\end{equation}
Before going to the proof of Proposition \ref{TechnicalProposition_C}, we establish some preparatory results on $S_{\eps}(\omega)$.
\begin{lemma}\label{TechnicalLemma3} For any $n\in\N$ and $\eta>0$, there exists $\eps\in (0,1)$ such that 
\[
\|S_{\eps,\omega}^n v\|\leq (1+\eta)\|T_{\omega}^n\|\| v\|\qquad \hbox{ for all } v\in e(\omega)^{\perp}, \omega\in \Omega.
\]
\end{lemma}
\begin{proof}
Choose and fix an arbitrary $n\in\N, \eta>0$. By \eqref{Perturbed_Operator_variation}, we have 
\[
\|S_{\eps}(\omega)-S_{0}(\omega)\|
\leq \eps \|S(\omega)\|\leq \eps \|S\|_{\infty}\quad \hbox{ for } \omega\in\Omega,
\]
which implies that 
\begin{equation}\label{Estimate1}
\lim_{\eps\to 0} \left(\sup_{\omega\in\Omega}\|S_{\eps,\omega}^n- S_{0,\omega}^n\|\right)=0.
\end{equation}
Note that by \eqref{dominateddirection} and invariance of $e(\omega)$ we have 
\[
\|T_{\omega}^n\|\geq \|T_{\omega}^n e(\omega)\|\geq\rho^n\quad\hbox{ for  } \omega\in\Omega.
\]
This together with \eqref{Estimate1} gives that there exists $\eps>0$ such that 
\begin{equation}\label{Estmiate2}
\|S_{\eps,\omega}^n- S_{0,\omega}^n\|
\leq \eta \|T_{\omega}^n\|\quad\hbox{ for } \omega\in\Omega.
\end{equation}
By \eqref{Perturbed_Operator_variation} and the fact that $T(\omega)$ preserves the random unit vector $e(\omega)$, we 
deduce that 
\[
S_{0,\omega}^n v= (\id-\Pi_{e}(\theta^n\omega))T_{\omega}^n v\qquad\hbox{ for all } v\in e(\omega)^{\perp}.
\]
Thus, $\|S_{0,\omega}^nv\|\leq \|T_{\omega}^n\|$ for all $v\in e(\omega)^{\perp}$ with $\|v\|=1$. Combining this fact with \eqref{Estmiate2}, we obtain that for all $v\in e(\omega)^{\perp}$ with $\|v\|=1$
\[
\|S_{\eps,\omega}^nv\|\leq \|S_{0,\omega}^n v\|+\eta\|T_{\omega}^n\|
\leq (1+\eta)\|T_{\omega}^n\|.
\]
The proof is complete.
\end{proof}
We are now in a position to prove Proposition \ref{TechnicalProposition_C}.
\begin{proof}[Proof of Proposition \ref{TechnicalProposition_C}]
(i) Let $\xi\in \cH$ be arbitrary with $\|\xi\|=1$. Then, for each $\omega\in\Omega$ there exist $\gamma\in\R$ and $u\in e(\omega)^{\perp}$ such that $\xi=\gamma e(\omega)+u$. Thus $\|\xi\|^2=\gamma^2+\|u\|^2$ and hence $|\gamma|\leq 1$. Consequently, by \eqref{Perturbed_Operator} we have 
\[
\|T(\omega) \xi-S(\omega)\xi\|=\|T(\omega) u-S(\omega)u\|=|(e^{2\alpha}-1)\gamma|\leq (e^{2\alpha}-1),
\]
which completes the proof of this part.

(ii) Thanks to \eqref{Relation_01} and Example \ref{ConeExample}(i), it is sufficient to prove that for some $\eps>0$ the random compact operator $S_{\eps}$ defined as in \eqref{Perturbed_Operator_variation} satisfies the assertion (ii), i.e. a long enough iteration of  $S_{\eps}$ preserves a random cone $(\cC_{\omega})_{\omega\in\Omega}$ satisfying the condition (C). Here, the family of closed convex cones $(\cC_{\omega})_{\omega\in\Omega}$ is given by 
\begin{equation}\label{Randomcone}
\cC_{\omega}:=\left\{v\in \cH:\langle e(\omega),v\rangle \geq\frac{\|v\|}{2}\right\}\qquad \hbox{for } \omega\in\Omega.
\end{equation}
By Example \ref{ConeExample}(ii),  $(\cC_{\omega})_{\omega\in\Omega}$ satisfies the condition (C). Choose and fix  $k\in\N$ such that $ e^{\alpha k N}\geq 8$. Using  Lemma \ref{TechnicalLemma3} with  $\eta:=\frac{1}{3}$, there exists $\eps>0$ such that  
\begin{equation}\label{Twoconstants}
\|S_{\eps,\omega}^{kN} v\|\leq \frac{4}{3}\|T_{\omega}^{kN}\|\| v\|\qquad \hbox{ for all } v\in e(\omega)^{\perp}, \omega\in \Omega.
\end{equation}
To conclude the proof, we show that 
\begin{equation}\label{InvarianceCone}
S_{\eps,\omega}^{kN} \cC_{\omega}
\subset \mathcal C_{\theta^{kN} \omega}\cup (-\mathcal C_{\theta^{kN}\omega})\qquad\hbox{ for } \omega\in\Omega,
\end{equation}
and there exists $R>0$ such that 
\begin{equation}\label{Step3_Eq1}
\beta_{\cC_{\theta^{kN}\omega}}(S_{\eps,\omega}^{kN} v,e(\theta^{kN}\omega))\leq R\qquad\hbox{for } v\in \cC_{\omega}\setminus\{0\} \hbox{ with } \|v\|=1.    
\end{equation}
For this purpose, choose and fix an arbitrary unit vector $v=\gamma e(\omega)+ u\in \cC_{\omega}$, where $u\in e(\omega)^{\perp}$. Thus, $\|v\|=\sqrt{\gamma^2+\|u\|^2}$ and by \eqref{Randomcone}, we have 
\begin{equation}\label{Initialvector}
\gamma= \langle e(\omega),v\rangle \geq 
\frac{\sqrt{\gamma^2+\|u\|^2}}{2}.
\end{equation}
From $S_{\eps,\omega}^{kN} v= \gamma S_{\eps,\omega}^{kN} e(\omega)+ S_{\eps,\omega}^{kN} u$, we derive that 
\begin{equation}\label{Norm_Estimate01}
\|S_{\eps,\omega}^{kN} v\|
\leq |\gamma|\|S_{\eps,\omega}^{kN} e(\omega)\|+\|S_{\eps,\omega}^{kN}u\|.
\end{equation}
Since $T(\omega)e(\omega)$ and hence $S_{\eps}(\omega)e(\omega)$ are collinear to $e(\theta\omega)$ it follows that 
\[
|\langle e(\theta^n\omega), S_{\eps,\omega}^n e(\omega)\rangle|=\|S_{\eps,\omega}^ne(\omega)\|
\quad \hbox{ for } n\in\N.
\] 
Consequently,  
\[
\langle e(\theta^{kN}\omega),S_{\eps,\omega}^{kN} v\rangle
=
\gamma \|S_{\eps,\omega}^{kN}e(\omega)\|
+ 
\langle e(\theta^{kN}\omega),S_{\eps,\omega}^{kN} u\rangle.
\]
Thus,
\begin{equation}\label{Innerproduct_Esitmate}
\left|\langle e(\theta^{kN}\omega),S_{\eps,\omega}^{kN} v\rangle \right|
\geq |\gamma| \|S_{\eps,\omega}^{kN}e(\omega)\| -\|S_{\eps,\omega}^{kN} u\|.
\end{equation}

\noindent
\underline{\emph{Verification of \eqref{InvarianceCone}}}: By the  definition of $\cC_{\omega}$ as in \eqref{Randomcone}, to show \eqref{InvarianceCone} it is sufficient to
verify that  $\left|\langle e(\theta^{kN}\omega),S_{\eps,\omega}^{kN} v\rangle \right|\geq \frac{\|S_{\eps,\omega}^{kN} v\|}{2}$. By \eqref{Norm_Estimate01} and \eqref{Innerproduct_Esitmate}, it is sufficient to show that $|\gamma| \|S_{\eps,\omega}^{kN}e(\omega)\|
  \geq 
 3 \|S_{\eps,\omega}^{kN} u\|$. In fact, we show a slightly stronger inequality
\begin{equation}\label{Newaim}
  |\gamma| \|S_{\eps,\omega}^{kN}e(\omega)\|
  \geq 
 2\sqrt{3} \|S_{\eps,\omega}^{kN} u\|,
\end{equation}
which is also useful for other estimates later. To gain this, firstly by \eqref{Perturbed_Operator}, \eqref{Perturbed_Operator_variation} and  \eqref{dominateddirection} we have 
\[
|\gamma|\|S_{\eps,\omega}^{kN}e(\omega)\|= |\gamma|e^{2\alpha kN } \|T_{\omega}^{kN}e(\omega)\|
\geq |\gamma|e^{\alpha kN} \|T_{\omega}^{kN}\|.
\]
Since $e^{\alpha kN}\geq 8$ it follows that $|\gamma|\|S_{\eps,\omega}^{kN}e(\omega)\|\geq 8|\gamma| \|T_{\omega}^{kN}\|$. Secondly, by \eqref{Twoconstants} we have $3\|S_{\eps, \omega}^{kN}u\|\leq 4\|T_{\omega}^{kN}\|\|u\|$. Finally, by \eqref{Initialvector} we have $\sqrt{3}|\gamma|\geq \|u\|$ and thus \eqref{Newaim} has been proved.

\noindent 
\underline{\emph{Verification of \eqref{Step3_Eq1}}}:  We are now estimating  $\beta_{\cC}(S_{\eps,\omega}^{kN}v, e(\theta^{kN}\omega))$. For this purpose, for any $t>0$ we have 
\[
\left|\left\langle
e(\theta^{kN}\omega), tS_{\eps,\omega}^{kN}v-e(\theta^{kN}\omega)
\right\rangle\right|    
\geq 
|t|\left|\left\langle
e(\theta^{kN}\omega), S_{\eps,\omega}^{kN}v
\right\rangle\right|-1,
\]
which together with \eqref{Innerproduct_Esitmate} implies that 
\begin{equation*}\label{Step3_Eq2}
 \left|\left\langle
e(\theta^{kN}\omega), tS_{\eps,\omega}^{kN}v-e(\theta^{kN}\omega)
\right\rangle\right|    
\geq 
|t||\gamma| \|S_{\eps,\omega}^{kN}e(\omega)\| -|t|\|S_{\eps,\omega}^{kN} u\|-1.
\end{equation*}
Meanwhile, using \eqref{Norm_Estimate01}, we obtain 
\[
\frac{\|tS_{\eps,\omega}^{kN}v-e(\theta^{kN}\omega)\|}{2}
\leq 
\frac{t\|S_{\eps,\omega}^{kN}v\|+1}{2} \leq 
\frac{t|\gamma|\|S_{\eps,\omega}^{kN}e(\omega)\|+\|S_{\eps,\omega}^{kN}u\|+1}{2}
\]
Hence, in order to have the inclusion  $tS_{\eps,\omega}^{kN}v-e(\theta^{kN}\omega)\in \cC_{\theta^{kN}\omega}$, equivalently \[\left|\left\langle
e(\theta^{kN}\omega), tS_{\eps,\omega}^{kN}v-e(\theta^{kN}\omega)
\right\rangle\right| \geq \frac{\|tS_{\eps,\omega}^{kN}v-e(\theta^{kN}\omega)\|}{2},
\]  
the positive number $t$ just only satisfies the following inequality 
\begin{equation}\label{Required_Inequality}
t|\gamma |\|S_{\eps,\omega}^{kN}e(\omega)\| \geq 3 t|\gamma|  \|S_{\eps,\omega}^{kN}u\|+3.
\end{equation}
From the fact that $\|v\|=1$ and \eqref{Initialvector}, we have $\gamma\geq \frac{1}{2}$. Thus, by virtue of \eqref{Newaim} and the fact that $\|S_{\eps,\omega}^{kN}e(\omega)\|\geq \|T_{\omega}^{kN}e(\omega)\|\geq \rho^{kN}$, any $t$ satisfying 
\[
t\geq \frac{12}{(2-\sqrt{3})\rho^{kN}}
\geq 
 \frac{3}{\left(1-\frac{\sqrt 3}{2}\right)|\gamma|\rho^{kN}}
\]
also satisfy \eqref{Required_Inequality}. Consequently, $\beta_{\cC}(S_{\eps,\omega}^{kN}v, e(\theta^{kN}\omega))\leq \frac{12}{(2-\sqrt{3})\rho^{kN}}$, which proves \eqref{Step3_Eq1}. The proof is complete.
\end{proof}
\subsection{Perturbation of random compact operators converging to null operator}\label{Subsection3.2}
In the following result, we show that corresponding to a random compact operator $T$ converging to the null operator there exist a finite number of disjoint sets $V,\theta V,\dots,\theta^N V$  such that there exists a perturbed random compact linear operator which is arbitrarily close to $T$ sending any given random vector space of dimension $1$  starting from $V$ to any given random vector space of dimension $1$  starting from $\theta^N V$. As a corollary, the random compact operator $T$ can be perturbed to a random compact operator having an invariant random subspace of dimension $1$.
\begin{proposition}\label{TechnicalProposition_1} Let $T\in\mathcal N$ and $\eps>0$ be arbitrary. Then  there exists an integer $N^*(T,\eps)\in\N$ (depending only on $T,\eps$) such that for all $N\geq N^*(T,\eps)$ there exists  a set $V\in\mathcal F$ of positive $\mP$-probability such that the sets $V,\theta V,\dots, \theta^N V$ are disjoint and the following statements hold

\noindent
(i) For any random unit vector $e:V\cup \theta^N V\rightarrow \mathcal H$, there exists $S\in\mathcal K_{\infty}(\Omega,\mathcal H)$ and a scalar random variable $\beta:\bigcup_{i=0}^{N-1}\theta^i V\rightarrow\R$ with $|\beta(\omega)|\geq \frac{\eps}{3}$ for $\omega\in \bigcup_{i=0}^{N-1}\theta^i V$ such that 
\begin{itemize}
\item [(p1)] \emph{Small perturbation}:
\[
S(\omega)=T(\omega)\hbox{ for } \omega\not\in \bigcup_{i=0}^{N-1}\theta^i V
\hbox{ and } \|S(\omega)- T(\omega)\|\leq \eps\hbox{ for } \omega\in \bigcup_{i=0}^{N-1}\theta^i V.	
\]
\item [(p2)] \emph{Existence of invariant one-dimensional random subspace along the sets $V,\theta V,\dots, \theta^N V$}: There exists an extension of $e$ to a random unit vector $e:\bigcup_{i=0}^N\theta^i V\rightarrow \cH$ such that
\[
S(\theta^i\omega)e(\theta^i\omega)=\beta(\theta^i\omega)e(\theta^{i+1}\omega)\quad \hbox{ for } \omega\in V, i=0,1,\dots,N-1.
\]
\end{itemize}

\noindent 
(ii) Let  $h:V\cup\theta^N V\rightarrow \cH$ be a random unit vector. Then, there exist $R \in \cK_{\infty}(\Omega,\cH)$, an extension of $h$ to a random unit vector $h:\Omega\rightarrow \cH$ and a scalar random variable $\gamma:\Omega\rightarrow \mathbb R$ with $|\gamma(\omega)|\geq \frac{\eps}{3}$ for $\omega\in\Omega$ such that $\|R-T\|_{\infty}\leq \eps$ and $R(\omega) h(\omega)=\gamma(\omega)h(\theta\omega)$.

\end{proposition}
\begin{proof}
Choose and fix an arbitrary $\eps>0$. Since $T\in\mathcal N$ it follows that 
\begin{equation}\label{ThmA_Eq1}
\lim_{n\to\infty}\frac{1}{n}\log\|T_{\omega}^n\|=-\infty\qquad\hbox{ for } \omega \in \Omega.
\end{equation}
For each $n\in\N$, we define 
\begin{equation}\label{ThmA_Eq2}
\Omega_n:=\left\{\omega\in\Omega: \|T_{\omega}^k\|< \left(\frac{\eps}{3}\right)^k\quad \hbox{for all } k\geq n\right\}.
\end{equation}
Then, $\Omega_n\subset \Omega_{n+1}$ and by \eqref{ThmA_Eq1} we have $\Omega=\bigcup_{n\geq 1}\Omega_n$. This implies that there exists $N^*(T,\eps)\in\N$ (depending on $T$ and $\eps$) such that  $\mP(\Omega_{N-1})>0$ for all $N\geq N^*(T,\eps)$. Choose and fix an arbitrary $N\geq N^*(T,\eps)$. By Theorem \ref{HR_Lemma}, there exists a measurable set $V\subset \Omega_{N-1}$ of positive probability such that $V,\theta V,\dots, \theta^{N} V$ are pairwise disjoint.  We are now proving that the measurable set $V$ satisfies the desired statements of the proposition:

(i) By Proposition \ref{Completeness} (iii), the set of random bijective compact operators is dense in the space of random compact operators. Then,  $T(\omega)$, where $\omega\in \bigcup_{i=0}^{N-1}\theta^i V$, can be perturbed to an arbitrarily closed random bijective compact linear operator that still satisfies \eqref{ThmA_Eq2}.  Then, w.l.o.g. we  assume that $T(\omega)$ is bijective for $\omega\in \bigcup_{i=0}^{N-1}\theta^i V$. Since $V\subset \Omega_N$ it follows that 
 \[
\|T^N_{\omega} e(\omega)\|=
\prod_{i=0}^{N-1}\left\|T(\theta^i\omega)\frac{T^{i}_{\omega}e(\omega)}{\|T^{i}_{\omega}e(\omega)\|}\right\|<
\left(\frac{\eps}{3}\right)^N,
\]
which implies that $V$ can be partitioned as the  union of the following disjoint sets $V_0,\dots V_{N-1}$, where for $j=0,1,\dots,N-1$
\begin{align*}
V_{j}
&:=
\Big\{\omega\in V: \|T(\theta^{j}\omega )\frac{T^{j}_{\omega}e(\omega)}{\|T^{j}_{\omega}e(\omega)\|}\|\leq \frac{\eps}{3}\; \&\; \|T(\theta^i\omega)\frac{T^{i}_{\omega}e(\omega)}{\|T^{i}_{\omega}e(\omega)\|}\|>\frac{\eps}{3}\notag\\
&\hspace{6.8cm} \hbox{ for } i=0,\dots,j-1\Big\}.   
\end{align*}
Now for each $\omega\in V_{j}$, we let $S(\theta^i\omega)=T(\theta^i\omega)$ for $i=0,\dots,j-1$ and $i=j+1,\dots, N-1$. It remains to construct $S (\theta^j\omega)$ by letting 
\begin{equation}\label{Rotate}
S(\theta^j\omega)
\frac{S^{j}_{\omega}e(\omega)}{\|S^{j}_{\omega}e(\omega)\|}
=
\frac{\eps }{3} 
\frac{(S^{N-j-1}_{\theta^{j+1}\omega})^{-1}e(\theta^N\omega)}{\|(S^{N-j-1}_{\theta^{j+1}\omega})^{-1}e(\theta^N\omega)\|}
\end{equation}
and $S(\theta^j\omega) v
=T(\theta^j\omega) v$ for any $v\in \big(S^{j}_{\omega}e(\omega)\big)^{\perp}$. Then, for any $\omega\in V_{j}$ we have $\|S(\theta^i\omega)-T(\theta^i\omega)\|\leq \frac{\eps}{3}$ for $i=0,1,\dots,N-1$ and by \eqref{Rotate} we arrive that $S^N_{\omega}e(\omega)$ is colinear to $e(\theta^N\omega)$. For $\omega\in V$ and $i=1,\dots,N-1$ we let $e(\theta^i\omega)=\frac{S_{\omega}^i e(\omega)}{\|S_{\omega}^i e(\omega)\|}$. Thus, there exists a non-zero random variable $\beta:\bigcup_{i=0}^{N-1}\theta^i V\rightarrow\R$ such that 
\begin{equation*}\label{Colinearity_01}
S(\omega) e(\omega)=\beta(\omega)e(\theta\omega)\qquad\hbox{ for } \omega\in \bigcup_{i=0}^{N-1}\theta^i V.
\end{equation*}
The remaining property is that $|\beta(\omega)|\geq \frac{\eps}{3}$ for all $\omega\in \bigcup_{i=0}^{N-1}\theta^i V$.  This property can be obtained by perturbing $S(\omega)$ on $\bigcup_{i=0}^{N-1}\theta^i V$, if necessary, as follows 
\begin{equation*}
S_{\mathrm{pert}}(\theta^i\omega) v:=
\left\{
\begin{array}{ll}
     \left(\beta(\theta^i\omega)+ \mbox{sign}( \beta(\theta^i\omega))\frac{\eps}{3}\right)S(\theta^i\omega) v &\quad 
     \hbox{if } v=e(\theta^i\omega),\\[1ex]
     S(\theta^i\omega) v &\quad \hbox{if } v\in (e(\theta^i\omega))^{\perp}.
     \end{array}   
\right.
\end{equation*}

(ii) In order to obtain the desired random compact linear operator $R$, we follow an analogous construction in (i) with a minor adjustment. More precisely,  the whole space $\Omega$ is decomposed as the union of the following disjoint sets. For each $k\in\N$ with $k\geq N+1$,  let $W_k$ denote the set of points in $V$ of which the first return time is $k$, i.e. \begin{equation*}
W_k:=\left\{
\omega\in V: \theta\omega,\dots,\theta^{k-1}\omega\not\in V, \theta^k\omega\in V
\right\}.
\end{equation*}
Then, $\theta^j V_k, k=N+1,\dots, j=1,\dots,k-1$, are pairwise disjoint and by Poincar\'{e}'s recurrence theorem, 
\begin{equation*}
V=\bigcup_{k=N+1}^{\infty}W_k,\qquad \Omega=\bigcup_{k=N+1}^{\infty}\bigcup_{j=0}^{k-1}\theta^j W_k.
\end{equation*}
For each $k\geq N+1$, the construction of the random compact linear operator $R(\omega)$, where $\omega\in \bigcup_{j=0}^{k-1}\theta^j V_k$, can be proceeded similarly as the construction of the random compact linear operator $S(\omega)$, where  $\omega\in \bigcup_{j=0}^{N-1}\theta^j V$ as in (i). 
\end{proof}
As a corollary of the preceding result, any random compact operator converging to null operator can be perturbed to one having a finite number of Lyapunov exponents. 
\begin{corollary}\label{Corollary_1}
Let $T\in \mathcal N$ be arbitrary. Then for any $\eps>0$ there exists 
$S\in \cK_{\infty}(\Omega,\cH)$ such that $\|T-S\|_{\infty}\leq \eps$ and $S$ has a finite number of Lyapunov exponents.
\end{corollary}
\begin{proof}
  Let $\eps>0$ be arbitrary.  By Proposition \ref{TechnicalProposition_1}, there exist $N\in\N$ and a set  $V\in\mathcal F$ of positive $\mP$-probability  such that the sets $V,\theta V,\dots, \theta^N V$ are disjoint, an unit random vector $h:\Omega\rightarrow \mathcal H$, a random compact  operator  $\widetilde T\in\mathcal K_{\infty}(\Omega,\mathcal H)$ and a non-zero scalar random variable $\gamma:\Omega\rightarrow\R$ such that $\|\widetilde T- T\|_{\infty}\leq \frac{\eps}{3}$, $|\gamma(\omega)|\geq \frac{\eps}{3}$ and 
\begin{equation}\label{Colinearity}
\widetilde T(\omega) h(\omega)= \gamma(\omega) h(\theta\omega)\qquad \hbox{for  } \omega \in \Omega.
\end{equation}
The desired random compact operator $S$ can be constructed as follows. For each $\omega\in\Omega$, we consider the restriction of the compact operator $\widetilde T$ on the Hilbert space $h(\omega)^{\perp}$ denoted by $\widetilde T|_{h(\omega)^{\perp}}$. By Proposition \ref{Completeness}(ii), the space of random compact operators of finite rank is dense in the space of random compact operators. Then,  there exists a linear operator $S(\omega)|_{h(\omega)^{\perp}}$ is of finite-rank and $\sup_{\omega\in\Omega}\|\widetilde T(\omega)-S(\omega)\|_{h(\omega)^{\perp}}\leq \frac{\eps}{3}$. The linear operator $S(\omega)|_{h(\omega)^{\perp}}$ is extended to the whole space $\mathcal H$ by letting $S(\omega)h(\omega)=\widetilde T(\omega)h(\omega)$. Then, $\|S(\omega)-\widetilde T(\omega)\|\leq \frac{\eps}{3}$ and therefore $\|S-T\|\leq \frac{2\eps}{3}$. Furthermore, the colinearity property \eqref{Colinearity} is still true for $S$, i.e. 
\begin{equation*}
S(\omega) h(\omega)= \gamma(\omega) h(\theta\omega)\qquad \hbox{for all } \omega \in \Omega,
\end{equation*}
 which implies that $\| S^n_{\omega}\|\geq \left(\frac{\eps}{3}\right)^n$. Hence,
 \[
 \kappa(S)=\lim_{n\to\infty}\frac{1}{n}\log\|S^n_{\omega}\|\geq \log\left(\frac{\eps}{3}\right).
 \]
 This, together with Theorem \ref{MET_HilbertSpace}, implies that $S$ has either a finite or an infinite and countable number of Lyapunov exponents. Note that  $S$ is of finite rank and therefore $S$ cannot have an infinite and countable number of Lyapunov exponents. The proof is complete.
 \end{proof}

\section{Proof of Theorem A}\label{Section_ProofA}

\subsection{Density of bounded random compact operators having finite  number of Lyapunov exponents}

The set $\cS$ defined as in \eqref{SimpleLE} consists of random compact operators satisfying that both the number of Lyapunov exponents is finite and the Lyapunov spectrum corresponding to the finite Lyapunov exponents is simple. The main ingredient in our proof of the density of   $\cS$ is the density of simple Lyapunov spectrum of linear random dynamical systems in a finite-dimensional space given in \cite{Knill92,Arnold_Cong_1999}, see also Theorem \ref{Density} for a concrete form of this result.  

\begin{proof}[Proof of Theorem A (i)]
Note that any $T\in \mathcal K_{\infty}(\Omega,\mathcal H)$  can be perturbed by an arbitrary small perturbation to a random compact operator of finite rank. Since a random compact operator of finite rank has either a finite number of Lyapunov exponents or converges to the null operator. This, together with Corollary \ref{Corollary_1}, implies that the set of compact random operators having a finite number of Lyapunov exponents is dense in $ \mathcal K_{\infty}(\Omega,\mathcal H)$. Then, to prove that $\cS$ is dense in $ \mathcal K_{\infty}(\Omega,\mathcal H)$,  it is sufficient to start with $T\in \mathcal K_{\infty}(\Omega,\mathcal H)$ having a finite number of Lyapunov exponents. Given an arbitrary $\eps>0$, we need to construct $\widetilde T\in \mathcal K_{\infty}(\Omega,\mathcal H)$ such that $\|\widetilde T-T\|_{\infty}\leq \eps$ and $\widetilde T$ has a finite number of Lyapunov exponents and the Lyapunov spectrum corresponding to its finite Lyapunov exponents is simple. Since $T$ has a finite number of Lyapunov exponents, there exist $\lambda_1,\dots,\lambda_k$, where $k\in\N$, and an equivariant decomposition of $T$ of the form 
	\[
	\mathcal H= F^{\infty}_{\omega}(T)\oplus \bigoplus_{i=1}^{k} \mathcal O^i_{\omega}(T),
	\]
where the Oseledets-Ruelle subspaces $\mathcal O^i_{\omega}(T), i=1,\dots,s,$ are of finite dimension, equivariant and
	\begin{align}
	\lim_{n\to\infty}\frac{1}{n}\log\|T^n_{\omega}|_{F_{\omega}^{\infty}(T)}\|=&
		-\infty,\label{Minusinfinity}\\
	\lim_{n\to\infty}\frac{1}{n}\log\|T^n_{\omega}v\|
	=&\lambda_i\qquad\hbox{for } v\in\mathcal O^i_{\omega}(T)\setminus\{0\}.\notag
	\end{align}
Let $\Pi(\omega)$ denote the projection onto the space $ \bigoplus_{i=1}^{k} \mathcal O^i_{\omega}(T)$ throughout the space $ F^{\infty}_{\omega}(T)$. Since $\|\Pi(\omega)\|$ is a random variable, there exists $M>0$ such that
\begin{equation}\label{Eq1}
\mP\big(U\big)>0,\qquad\hbox{where } U:=\big\{\omega: \|\Pi(\omega)\|\leq M\big\}.
\end{equation}
Let $d:=\dim \bigoplus_{i=1}^k\mathcal O_i(\omega)$. Obviously, $d\geq k$ and by Measurable Selection Theorem there exists a random orthogonal basis $(e_i(\omega))_{i=1}^d$ of the space $\bigoplus_{i=1}^k\mathcal O_i(\omega)$, see e.g. \cite{Arnold}. By equivariance of $\bigoplus_{i=1}^k\mathcal O_i(\omega)$, there exists a random matrix $A(\omega)=(a_{ij}(\omega))_{i,j=1}^d$ satisfying that
\[
T(\omega) e_i(\omega)=\sum_{j=1}^{d} a_{ij}(\omega) e_j(\theta\omega)\quad\hbox{for } i=1,\dots,d.
\]
By the boundedness of $T$, the random matrix $A$ is also bounded. By the density of simple Lyapunov spectrum of linear random dynamical systems in a finite-dimensional space, see \cite{Knill92,Arnold_Cong_1999} and Theorem \ref{Density}, there exists $B\in\mathcal M_{\infty}(d)$ satisfying that
\begin{equation}\label{Eq2}
\mbox{ess}\sup_{\omega\in\Omega}\|B(\omega)-A(\omega)\|\leq \frac{\eps}{M},\quad B(\omega)=A(\omega)\hbox{ for all } \omega\in\Omega\setminus U
\end{equation}
and the linear random dynamical system $B$ has a simple Lyapunov spectrum, i.e. $\lambda_1(B)>\lambda_2(B)>\dots>\lambda_d(B)>-\infty$. Define the random operator $\widetilde T:\Omega\rightarrow \mathcal K(\mathcal H)$ by
\begin{equation}\label{Eq3_Part1}
\widetilde T(\omega) e_i(\omega)=\sum_{j=1}^{d} b_{ij}(\omega) e_j(\theta\omega)\quad\hbox{ for } i=1\dots d
\end{equation}
and 
\begin{equation}\label{Eq3_Part2}
\widetilde T(\omega)v=T(\omega)v\quad\hbox{ for all } v\in F_{\omega}^{\infty}(T).
\end{equation}
By \eqref{Eq3_Part2} and \eqref{Minusinfinity}, we have $	\lim_{n\to\infty}\frac{1}{n}\log\|\widetilde T^n_{\omega}|_{F_{\omega}^{\infty}(T)}\|=
		-\infty$.This together with \eqref{Eq3_Part1} implies that the Lyapunov exponent of $\widetilde T$ is $\lambda_1(B),\dots,\lambda_d(B)$. Furthermore, by the simplicity of the Lyapunov spectrum of $B$, the Lyapunov spectrum of $\widetilde T$ is also simple. To conclude the proof, it is sufficient to show that $\|T-\widetilde T\|_{\infty}\leq \eps$. Since $B$ and $A$ coincide on the set $\Omega\setminus U$ it follows that $T(\omega)=\widetilde T(\omega)$ for all $\omega\in \Omega\setminus U$. Now, let $\omega\in U$,  $v\in \mathcal H, \|v\|=1$ be arbitrary. Then, by \eqref{Eq3_Part2} we have 
		\begin{align*}
		\|T(\omega)v-\widetilde T(\omega)v\|=&\|T(\omega)\Pi(\omega)v-\widetilde T(\omega)\Pi(\omega)v\|\\[1ex]
		\leq& \|B(\omega)-A(\omega)\|\|\Pi(\omega)v\|,
		\end{align*}
where we used \eqref{Eq3_Part1} to obtain the last inequality.  This together with \eqref{Eq1} and \eqref{Eq2} indicates that $\|T(\omega)v-\widetilde T(\omega)v\|\leq \eps$. Thus,  $\|T-\widetilde T\|_{\infty}\leq \eps$ and the proof is complete.
\end{proof}
\subsection{Density of bounded random compact operators having an infinite and countable number of Lyapunov exponents
}
\begin{proof}[Proof of Theorem A (ii)]
Let $T\in \mathcal K_{\infty}(\Omega;\mathcal H)$ and $\eps>0$ be arbitrary. To conclude the proof, we construct $\widetilde T\in  \mathcal K_{\infty}(\Omega;\mathcal H)$ satisfying that $\|T-\widetilde T\|_{\infty}\leq \eps$ and having an infinitely countable number of Lyapunov exponents. Since $T(\omega)$ is a compact operator, there exists a measurable subspace $F(\omega)$ of infinite dimension such that \begin{equation}\label{TheoremB_Eq1}
\|T(\omega)|_{F(\omega)}\|\leq \frac{\eps}{3}\quad \hbox{ for } \mP\hbox{-a.e. } \omega\in\Omega.
\end{equation}
Let $(e_i(\omega))_{i\in\N}$ be an orthonormal basic of $F(\omega)$. Define $\widetilde T(\omega):\cH\rightarrow \cH$ as $\widetilde T(\omega)v=T(\omega)v$ for all $v\in F(\omega)^{\perp}$ and 
\begin{equation}\label{Componentwisedefinition}
\widetilde T(\omega)e_i(\omega)=\left(\frac{\eps}{3}\right)^i e_i(\theta\omega)\qquad\hbox{ for } i=1,2,\dots.
\end{equation}
Note that for $\mP$-a.e. $\omega\in\Omega$ any unit vector $u\in F(\omega)$  can be uniquely expressed as $u=\sum_{i=1}^{\infty}\beta_i e_i(\omega)$, where $\sum_{i=1}^{\infty}\beta_i^2=1$. Thus, by definition of $\widetilde T(\omega)$ we have 
\[
\|\widetilde T(\omega)u\|^2
=
\sum_{i=1}^{\infty}\left(\frac{\eps}{3}\right)^{2i}\beta_i^2\leq \left(\frac{\eps}{3}\right)^2,
\]
which implies that $\|\widetilde T(\omega)|_{F(\omega)}\|\leq \frac{\eps}{3}$. This together with \eqref{TheoremB_Eq1} implies that $\|T(\omega)-\widetilde T(\omega)|_{F(\omega)}\|\leq \frac{2\eps}{3}$. Note that $T(\omega)-\widetilde T(\omega)|_{F(\omega)^{\perp}}=0$, hence $\|T-\widetilde T\|_{\infty}\leq \frac{2\eps}{3}$. Obviously, by \eqref{Componentwisedefinition}, $\log\left(\frac{\eps}{3}\right), 2\log\left(\frac{\eps}{3}\right),\dots$ are Lyapunov exponents of $\widetilde T$, which means that $\widetilde T$ has an infinite and countable number of  Lyapunov exponents. Hence, it remains to show that $\widetilde T(\omega)$ is compact. To see that, by \eqref{Componentwisedefinition} the restriction linear operator $\widetilde T(\omega)|_{F(\omega)}$ can be approximated by a sequence of finite-rank linear operators $\widehat T_k(\omega)|_{F(\omega)}$, where $k=1,2,\dots$, defined by 
\[
\widehat T_k(\omega)e_i(\omega)
=\left\{
\begin{array}{ll}
\widetilde T(\omega)e_i(\omega)     & \hbox{ for } i=1,\dots,k, \\[1ex]
0     & \hbox{ for } i=k+1,k+2,\dots. 
\end{array}
\right.
\]
On the other hand, since $T(\omega)$ is compact and $T(\omega)|_{F(\omega)^{\perp}}$, $\widetilde T(\omega)|_{F(\omega)^{\perp}}$ coincide, the restriction linear operator $\widetilde T(\omega)|_{F(\omega)^{\perp}}$ can be also approximated by a sequence of finite-rank linear operators $\left(\widehat T_k(\omega)|_{F(\omega)^{\perp}}\right)_{k\in\N}$. Combining two sequences of finite-rank operators $\left(\widehat T_k(\omega)|_{F(\omega)}\right)_{k\in\N}$  and $\left(\widehat T_k(\omega)|_{F(\omega)^{\perp}}\right)_{k\in\N}$, we obtain a sequence of finite-rank linear  operators $(\widehat T_k(\omega))_{k\in \N}$ on $\mathcal H$ defined as 
\[
\widehat T_k(\omega)v
=
\left\{
\begin{array}{ll}
\widehat T_k(\omega)|_{F(\omega)} v     & \hbox{ for } v\in F(\omega),\\[1.5ex]
\widehat T_k(\omega)|_{F(\omega)^{\perp}} v     & \hbox{ for } v\in F(\omega)^{\perp},
\end{array}
\right.
\]
converging to $\widetilde T(\omega)$. Thus, $\widetilde T(\omega)$ is compact and the proof is complete.
\end{proof}
\subsection{Negligibility for the convergence of random compact operators to the null operator}
The structure of the proof of part (iii) of Theorem A  is that we start with a random  compact operator converging to the null operator. Then, we first perturb this random compact operator to the new one having an invariant direction with a good growth rate, i.e., this new random compact operator satisfies all properties of Proposition \ref{TechnicalProposition_C}. Finally, we apply Proposition \ref{TechnicalProposition_C} to show the existence of an open set of random compact operators nearby to the original one such that any of them has a non-trivial Lyapunov spectrum. 

Concerning a perturbation of a random compact linear operator converging to the null operator to the new one having an invariant direction with good growth rate, we follow the idea of the Millionschchikov rotation technique developed in \cite[Theorem 3.6]{Cong05}, see also \cite{Avila09}. The content of this result is stated as follows.

\begin{proposition}\label{Millionschikov}
Let $T\in\mathcal N$ be arbitrary and let $V,\theta V,\dots, \theta^N V$ be pairwise disjoint sets satisfying assertions of Proposition \ref{TechnicalProposition_1}. Then, for any $\eps\in (0,\frac{1}{6})$ and $k\in \N$ there exist $S\in\cK_{\infty}(\Omega,\cH)$ and an invariant unit random vector $e:\Omega\rightarrow \cH$ such that $\|S-T\|_{\infty}\leq 3\eps (1+ \|T\|_{\infty})$ and
\begin{align}
\|S(\omega)e(\omega)\|
& \geq \frac{\eps^{N+1}}{(1+2\|T\|_{\infty})^{N-1}}\qquad\hbox{ for } \omega\in\Omega,\label{Lowerestimate}\\[1ex]
 \|S^{kN}_{\omega}e(\omega)\|
 & \geq  \frac{\eps^{k+4N}}{(1+2\|T\|_{\infty})^{4N}} \|S_{\omega}^{kN}\|\qquad\hbox{ for } \omega\in\Omega.\label{Goodgrowthrate}
\end{align}
\end{proposition}

To prove the above proposition, we show the following preparatory lemma about a perturbation result for a random compact operator along a finite disjoint sets.
\begin{lemma}\label{Perturbationalongashortperiod}
Let $T\in\cK_{\infty}(\Omega,\cH)$ be arbitrary. Let $V\subset \Omega$ be a measurable set such that $V,\theta V,\dots, \theta^{N-1} V$ are disjoint, where $N\in \mathbb N$.  Let $g: V\rightarrow \cH$ be a unit random vector. Then, for any $\eps\in (0,\frac{1}{6})$ there exists a strongly measurable map $S:\bigcup_{i=0}^{N-1} \theta^i V\rightarrow \mathcal K(\mathcal H)$ satisfying the following requirements:
\begin{itemize}
\item [(r1)] For $\omega\in \bigcup_{i=0}^{N-1} \theta^i V$, we have $\|S(\omega)-T(\omega)\|\leq 3\eps + 3\eps \|T\|_{\infty}$, 

\item [(r2)] For $\omega\in V$, we have 
 
\begin{equation}\label{Smallgrowth}
 \|S_{\omega}^N g(\omega)\|\geq \eps \|S_{\omega}^N\|\quad\hbox{and}  \quad \|S_{\omega}^N\|\geq \eps^{N+1}.
\end{equation}
\end{itemize}
\end{lemma}
\begin{proof}
By Proposition \ref{Completeness} (i), there exists a unit random vector $e:V\rightarrow \cH$ such that $\|T_{\omega}^N e(\omega)\|= \|T_{\omega}^N\|$ for $\omega\in V$. In view of  Proposition \ref{Completeness} (iii), by perturbing $T(\omega)$ on $\bigcup_{i=0}^{N-1}\theta^i V$, if necessary, we can assume that $T^N_{\omega}e(\omega)\not=0$ for $\omega\in V$. We first perturb $T(\omega)$ on $\bigcup_{i=0}^{N-1}\theta^i V$ to the new one $R(\omega)$ such that $\|R^{N}_{\omega}\|$ is uniformly separated from zero for $\omega\in V$. Precisely, for $\omega\in V$ and $i=0,\dots,N-1$ let  $e_{i}(\omega)=\frac{T^{i-1}_{\omega} e(\omega)}{\|T^{i-1}_{\omega} e(\omega)\|} $. A random compact operator $R(\theta^i\omega)$ is defined by 
\[
R(\theta^i\omega) e_i(\omega)
=
\left\{
\begin{array}{ll}
    T(\theta^i\omega)  e_i(\omega)& \hbox{ if } \|T(\theta^i\omega)  e_i(\omega)\| \geq \eps,  \\[1ex]
     \frac{\eps}{\|T(\theta^i\omega)e_i(\omega)\|}T(\theta^i\omega) e_i(\omega)     & \hbox{ if } \|T(\theta^i\omega)  e_i(\omega)\| < \eps,
\end{array}
\right.
\]
and 
$R(\theta^i\omega)v= T(\theta^i\omega) v$ for $v\in e_i(\omega)^{\perp}$. Then, it is easily seen that $\|R(\omega)-T(\omega)\|\leq 2\eps$ for $\omega\in \bigcup_{i=0}^{N-1}\theta^i V$ and 
\begin{equation}\label{NormSeparatedzero}
\|R^N_{\omega}e(\omega)\|=\prod_{i=0}^{N-1}\left\| R(\theta^i\omega)e_i(\omega)\right\|\geq \eps^N\quad\hbox{ for all } \omega\in V.     
\end{equation}
Now, we will perturb $R(\omega)$ to obtain the desired random compact operator $S(\omega)$. For this purpose, using Proposition \ref{Completeness} (i), there exists a unit random vector $f:V\rightarrow \cH$ such that $\|R_{\omega}^N f(\omega)\|= \|R_{\omega}^N\|$ for $\omega\in V$.
We write $g(\omega)=\alpha(\omega) f(\omega)+ h(\omega)$, where 
\[
\alpha(\omega)=\langle g(\omega), f(\omega)\rangle,\quad
h(\omega)=g(\omega)-\alpha(\omega) f(\omega)\in f(\omega)^{\perp}.
\]
For any $\omega\in\bigcup_{i=1}^{N-1}\theta^i V$, we let $S(\omega)=R(\omega)$. For any $\omega\in V$, we construct a compact operator $S(\omega)$ as follows $S(\omega)v=R(\omega)v$ for all $v\in g(\omega)^{\perp}$ and $S(\omega)g(\omega)$ is defined by 
\begin{equation}\label{ConstructionS}
S(\omega) g(\omega)
=
\left\{
\begin{array}{ll}
R(\omega) g(\omega)     & \hbox{if } |\alpha(\omega)|\geq \eps, \\[1ex]
    (\alpha(\omega)+ 3\eps) R(\omega)  f(\omega)+ R(\omega)h(\omega))     & \hbox{if } |\alpha(\omega)|< \eps.
\end{array}
\right.
\end{equation}
By the above construction, we have $\|S(\omega)-R(\omega)\|\leq 3\eps \|R(\omega)\|$ for $\omega\in\bigcup_{i=0}^{N-1}\theta^i V$. This together with the fact that $\|R(\omega)-T(\omega)\|\leq 2\eps$  for $\omega\in \bigcup_{i=0}^{N-1}\theta^i V$ implies that 
\[
\|S(\omega)-T(\omega)\|
\leq 2\eps +3\eps (\|T(\omega)+2\eps\|)
\leq 3\eps +3\eps \|T(\omega)\|\quad \hbox {for } \omega\in \bigcup_{i=0}^{N-1}\theta^i V.
\]
Thus, (r1) is fulfilled, and to complete the proof, it remains to verify \eqref{Smallgrowth} for $\omega\in V$. To do this, we choose and fix $\omega\in V$ and  consider two separated cases: 

\emph{Case 1}: $|\alpha(\omega)|\geq \eps$. By \eqref{ConstructionS}, $S(\omega)=R(\omega)$ and therefore from \eqref{NormSeparatedzero} we derive that   $\|S_{\omega}^N\|=\|R_{\omega}^N\|\geq \eps^N$. Furthermore,  
\[
\|S_{\omega}^N g(\omega)\|= \|\alpha(\omega) R_{\omega}^Nf(\omega)+  R_{\omega}^N h(\omega)\|.
\]
Since $h(\omega)\in f(\omega)^{\perp}$ and then by Proposition \ref{Completeness}(i), we have $\langle R_{\omega}^N f(\omega), R_{\omega}^N h(\omega)\rangle=0$. Thus, 
\[
\|S_{\omega}^N g(\omega)\|^2 =\alpha(\omega)^2 \|R_{\omega}^Nf(\omega)\|^2+ \|R_{\omega}^Nh(\omega)\|^2
\geq \eps^2 \|R_{\omega}^N\|^2,
\]
which yields that $\|S_{\omega}^Ng(\omega)\|\geq \eps \|S_{\omega}^N\|$. This, together with \eqref{NormSeparatedzero}, verifies  \eqref{Smallgrowth} in this case.

\emph{Case 2}:  $|\alpha(\omega)|< \eps$. By \eqref{ConstructionS}, we have 
\[
\|S(\omega)g(\omega)-R(\omega)g(\omega)\|
=
\|3\eps R(\omega)g(\omega)\|\leq 3\eps \|R\|_{\infty},
\]
which together with the fact that $S(\omega)v=R(\omega)v$ for all $v\in g(\omega)^{\perp}$ implies that $\|S(\omega)-R(\omega)\|\leq 3\eps \|R\|_{\infty}$ for all $\omega\in V$. Concerning the estimate of $\|S_{\omega}^N g(\omega)\|$, by using a similar argument as in Case 1, we have 
\begin{equation}\label{Firstestimate}
\|S_{\omega}^N g(\omega)\| \geq (\alpha(\omega)+3\eps) \|R_{\omega}^N\|\geq 2\eps \|R_{\omega}^N\|.
\end{equation}
Thus, on the one hand, by \eqref{NormSeparatedzero} we have 
\[
\|S_{\omega}^N\|\geq \|S_{\omega}^Ng(\omega)\|\geq \eps^{N+1}.
\]
On the other hand, by \eqref{ConstructionS} and (r1) we have $S^N_{\omega}v=R^N_{\omega}v$ for all $v\in g(\omega)^{\perp}$ and 
\[
\|S^N_{\omega}g(\omega)-R^N_{\omega}g(\omega)\|=
\|3\eps R_{\omega}^N f(\omega)\|=3\eps \|R_{\omega}^N\|,
\]
which implies that $\|S_{\omega}^N\|\leq (1+3\eps ) \|R_{\omega}^N\|$. This, together with \eqref{Firstestimate}, deduces that 
\[
\|S_{\omega}^Ng(\omega)\|\geq \frac{2\eps}{1+3\eps }\|S_{\omega}^N\|\geq \eps\|S_{\omega}^N\|.
\]
Thus, \eqref{Smallgrowth}  holds in this case. The proof is complete. 
\end{proof}

\begin{proof}[Proof of Proposition \ref{Millionschikov}]
By virtue of Halmos-Rokhlin's lemma (see Theorem \ref{HR_Lemma} in the Appendix), there exists a measurable subset $U$ of $V$ such that  $\mP(U)>0$ and the sets $U,\theta U,\dots,\theta^{(k+1)N}U$ are disjoint. By the Poincar\'{e} recurrence theorem, the set $\theta^NU$ is decomposed as the union of the following disjoint subsets $\bigcup_{i={kN+1}}^{\infty} E_i$ defined as 
\begin{equation*}\label{DecomposeofU}
E_i:=\Big\{\omega\in\theta^N U: \theta^j\omega\not\in U \hbox{ for } j=1,\dots,i-1 \hbox{ and } \theta^i\omega\in U \Big\}.    
\end{equation*}
Therefore, all the subsets $E_i,\theta E_i,\dots,\theta^i E_i$, where $i=kN+1,\dots$, are pairwise disjoint and the whose space $\Omega$ is represented as the union of the following disjoint subsets 
\begin{align}
\Omega
=&  \bigcup_{j=1}^{N-1}\theta^j U\;\cup\; \bigcup_{i=kN+1}^{\infty}  \Big(E_i\cup \theta E_i\dots\cup\theta^i E_i\Big)\label{PresentationOmega_I}\\
=&
 \bigcup_{j=0}^{N-1}\theta^j U\;\cup\; \bigcup_{i=kN+1}^{\infty}  \Big(E_i\cup \theta E_i\dots\cup\theta^{i-1} E_i\Big).\label{PresentationOmega_II}
\end{align}
The proof is divided into the construction of 
$S(\omega)$ and $e(\omega)$ and the estimate $\|S(\omega)e(\omega)\|, \|S^{kN}_{\omega}e(\omega)\|$:

Part 1: Constructions of $S(\omega)$ and $e(\omega)$. These constructions are divided into two steps:

\emph{Step 1}: Construction $S(\omega)$ on $\bigcup_{j=0}^{i-1} \theta^j E_i$ and $e(\omega)$ on $\bigcup_{j=0}^i \theta^j E_i$ for $i=kN+1,\dots$. Let $i\geq kN+1$ and $\omega\in E_i$ be arbitrary.  We write $i=\ell N+r$, where $\ell\in\N$ with $\ell\geq k$ and $0\leq r\leq N-1$. Then, $\bigcup_{j=0}^i \theta^j E_i$ is decomposed as 
\[
\bigcup_{j=0}^i \theta^j E_i
=
\underbrace{\bigcup_{j=0}^{N-1} \theta^j E_i}_{\mathrm{Set\; I}}
\cup 
\underbrace{\bigcup_{j=N}^{\ell N-1} \theta^j E_i}_{\mathrm{Set\; II}}
\cup 
\underbrace{\bigcup_{j=\ell N}^{\ell N+r} \theta^j E_i}_{\mathrm{Set\; III}}.
\]

The constructions of $S(\cdot)$ and $e(\cdot)$ on Set I, Set II, and Set III are given in the following initial step, induction step, and remainder step, respectively.

\noindent 
\underline{Initial step}: Choose and fix an arbitrary random unit map $e:E_i\rightarrow \cH$. Using Lemma \ref{Perturbationalongashortperiod} for a perturbation of $T$ on the disjoint sets $E_i,\theta E_i,\dots,\theta^{N-1}E_i$, there exists $S:\bigcup_{j=0}^{N-1} \theta^j E_i\rightarrow \mathcal K(\cH)$ such that $\|S(\omega)-T(\omega)\|\leq 3\eps (1+\|T\|_{\infty})$ and 
\begin{equation}\label{Construction_Step1a}
\|S_{\omega}^N\|\geq \eps^{N+1}\quad\hbox{and}\quad  \|S_{\omega}^Ne(\omega)\|\geq \eps \|S_{\omega}^N\|\quad \hbox{ for } \omega \in E_i.
\end{equation}
Now, we define $e(\omega)$ on $\theta E_i,\dots,\theta^N E_i$ as 
\begin{equation}\label{Construction_Step1b}
e(\theta^{j}\omega)=\frac{S^{j}_{\omega}e(\omega)}{\|S^j_{\omega}e(\omega)\|}\quad \hbox{ for } \omega \in E_i \hbox{ and } j=1,\dots, N.
\end{equation}

\noindent 
\underline{Induction step}: After the step above, $e(\omega)$ was determined for $\omega\in \theta^N E_i$. Then, applying the construction above to perturb $T$ on the disjoint sets $\theta^{N}E_i,\theta^{N+1}E_i,\dots,\theta^{2N-1}E_i$, $S(\omega)$ is constructed on the sets $\bigcup_{j=N}^{2N-1}\theta^j E_i$ and $e(\omega)$ is constructed on the sets   $\bigcup_{j=N+1}^{2N}\theta^j E_i$ such that \eqref{Construction_Step1a} and \eqref{Construction_Step1b} also hold for $\omega\in \theta^N E_i$. We do this progress inductively and obtain $S$ on $\bigcup_{j=0}^{\ell N-1} \theta^j E_i$ and $e(\omega)$ on $\bigcup_{j=0}^{\ell N} \theta^j E_i$ satisfying that $\|S(\omega)-T(\omega)\|\leq 3\eps(1+ \|T\|_{\infty})$ for  $\omega\in \bigcup_{j=0}^{\ell N-1}\theta^j V_i$ and 
\begin{equation}
\|S_{\omega}^N\|\geq \eps^{N+1}, \|S_{\omega}^Ne(\omega)\|
\geq  \eps \|S_{\omega}^N\|\quad \hbox{for all } \omega \in E_i\cup \theta^N E_i\cup\dots\cup\theta^{\ell N} E_i\label{Nstepnorm},
\end{equation}
and
\begin{equation}\label{Invariance_InductionStep}
e(\theta^j\omega)
=
\frac{S^{j}_{\omega}e(\omega)}{\|S^j_{\omega}e(\omega)\|}\quad \hbox{ for all } \omega \in E_i \hbox{ and } j=1,\dots, \ell N.
\end{equation} 

\noindent 
\underline{Remainder step}: So far, $S(\theta^j\omega)$ has been constructed for $\omega\in E_i$ and $j=0,1,\dots,\ell N-1$ and $e(\theta^j\omega)$ has been constructed for $\omega\in E_i$ and $j=0,1,\dots,\ell N$. For $\omega\in E_i, j=\ell N$, if $\|T(\theta^{j}\omega) e(\theta^{j}\omega)\|\geq \eps$ we let 
\[
S(\theta^{j}\omega)=T(\theta^{j}\omega),\quad   e(\theta^{j+1}\omega)=\frac{S(\theta^j\omega) e(\theta^j\omega)}{\|S(\theta^j\omega) e(\theta^j\omega)\|}
\]
and if $\|T(\theta^{j}\omega) e(\theta^{j}\omega)\|< \eps$ we let
\[
S(\theta^j\omega) v
=
\left\{
\begin{array}{ll}
  (3\eps+T(\theta^j\omega)) v   & \hbox{if } v=e(\theta^j\omega),\\[1ex]
T(\theta^j\omega) v   &  \hbox{if } v\in e(\theta^j\omega)^{\perp},
\end{array}
\right.\quad e(\theta^{j+1}\omega)=\frac{S(\theta^j\omega) e(\theta^j\omega)}{\|S(\theta^j\omega) e(\theta^j\omega)\|}.
\]
Next, we repeat this construction step by step for  $\omega\in E_i, j=\ell N+1,\dots, \ell N+r-1$. Then, $S(\theta^j\omega)$ has been constructed for $\omega \in E_i, j=\ell N,\dots,\ell N+r-1$ and $e(\theta^j\omega)$ has been constructed for $\omega\in E_i, j=\ell N+1,\dots,\ell N+r$. Furthermore, for $j=\ell N,\dots,\ell N+r-1$
\begin{equation}\label{Remainder}
S(\theta^j\omega)e(\theta^j\omega) \hbox{ is colinear to } e(\theta^{j+1}\omega)
\quad \&\quad \|S(\theta^j\omega)e(\theta^j\omega)\|\geq \eps.
\end{equation}
\emph{Step 2}: According to \eqref{PresentationOmega_I} and \eqref{PresentationOmega_II}, it remains to construct $S(\omega)$ on $\bigcup_{j=0}^{N-1}\theta^j U$ and to construct unit random vectors $e(\omega)$ on $\bigcup_{j=1}^{N-1} \theta^j U$. It is noted that $e(\omega)$ was already determined for $\omega\in U\cup\theta^N U$. Thus, by virtue of Proposition \ref{TechnicalProposition_1}(i) we can construct $S(\omega)$ on $\bigcup_{j=0}^{N-1}\theta^j U$ and unit random vectors $e(\omega)$ on  $\bigcup_{j=1}^{N-1} \theta^j U$ such that $\|S(\omega)-T(\omega)\|\leq 3\eps(1+\|T\|_{\infty})$ and 
\begin{equation}\label{Construction_Step2}
S(\theta^i\omega)e(\theta^i\omega)=\beta(\theta^i\omega)e(\theta^{i+1}\omega)\quad \hbox{ for } \omega\in U, i=0,1,\dots,N-1,
\end{equation}
where $\beta: \bigcup_{j=0}^{N-1}\theta^j U\rightarrow \R$ satisfies that $|\beta(\omega)|\geq \eps$ for all $\omega\in \bigcup_{j=0}^{N-1}\theta^j U$.
From \eqref{Construction_Step1b}, \eqref{Invariance_InductionStep}, \eqref{Remainder} and \eqref{Construction_Step2} of  the above construction,  $e(\omega)$ is invariant in the sense that $S(\omega)e(\omega)$ is colinear to $e(\theta\omega)$ for all $\omega\in\Omega$ and $\|S-T\|_{\infty}\leq 3\eps(1+\|T\|_{\infty})$. 

\noindent 
Part 2: Estimation on  $\|S(\omega)e(\omega)\|, \|S^{kN}_{\omega}e(\omega)\|$. Since $\eps\in (0,\frac{1}{6})$ it follows that 
\begin{equation}\label{BoundonS}
\|S\|_{\infty}\leq \|T\|_{\infty}+3\eps(1+\|T\|_{\infty})\leq 1 +2\|T\|_{\infty}.
\end{equation}
We now estimate $\|S(\omega)e(\omega)\|$. For any $\omega\in\bigcup_{j=0}^{N-1}\theta^jU$ we use \eqref{Construction_Step2} to gain that 
\begin{equation}\label{Norm_1}
\|S(\omega)e(\omega)\|
\geq \eps \geq \frac{\eps^{N+1}}{(1+2\|T\|_{\infty})^{N-1}}.
\end{equation}
It remains to estimate $\|S(\omega)e(\omega)\|$ for $\omega\in \bigcup_{i=kN+1}^{\infty}  \Big(E_i\cup \theta E_i\dots\cup\theta^{i-1} E_i\Big)$. Let $\omega=\theta^j \widehat\omega$ for some $\widehat\omega \in E_i$, where $i=\ell N+r \geq kN+1$ and $0\leq j\leq i-1$. Write $j=sN+t$. In the case that $s<\ell$,  by the construction in the Initial Step and Induction Step of Part 1, we have  
\begin{align*}
\|S(\omega)e(\omega)\|
&=\|S(\theta^{sN+t}\widehat\omega)e(\theta^{sN+t}\widehat\omega)\|\notag\\
&=\frac{\|S^N_{\theta^{sN}\widehat\omega}e(\theta^{sN}\widehat\omega)\|}{\|S^t_{\theta^{sN}\widehat\omega}e(\theta^{sN}\widehat\omega)\|\|S^{N-t-1}_{\theta^{sN+t+1}\widehat\omega}e(\theta^{sN+t+1}\widehat\omega)\|}\notag\\
&\geq  
\frac{\eps^{N+1}}{\|S\|_{\infty}^{N-1}},
\end{align*}
which together with \eqref{Norm_1} implies that
\begin{equation}\label{Norm_2}
\|S(\omega)e(\omega)\| 
\geq \frac{\eps^{N+1}}{(1+2\|T\|_{\infty})^{N-1}}.
\end{equation}
In the case that $s=\ell$, by the construction in the Remainder Step, we have 
\begin{equation}\label{Norm_3}
\|S(\omega)e(\omega)\|
\geq \eps \geq \frac{\eps^{N+1}}{(1+2\|T\|_{\infty})^{N-1}}.
\end{equation}
Combining \eqref{Norm_1}, \eqref{Norm_2} and \eqref{Norm_3}, the estimate on $\|S(\omega)e(\omega)\|$ as in  \eqref{Lowerestimate} was verified. To conclude the proof, we show that $\|S^{kN}_{\omega}e(\omega)\|$ fulfills  \eqref{Goodgrowthrate}. Thanks to \eqref{PresentationOmega_II}, we can verify \eqref{Goodgrowthrate} by considering  two separated cases $\omega\in  \bigcup_{j=0}^{N-1}\theta^j U$ (Case 1) or $\omega\in \bigcup_{i=kN+1}^{\infty}  \Big(E_i\cup \theta E_i\dots\cup\theta^{i-1} E_i\Big)$ (Case 2).

\noindent
\emph{Case 1}: $\omega\in \bigcup_{j=0}^{N-1}\theta^j U$. Let $\omega=\theta^{\tau}\widehat \omega$ for some $\widehat\omega\in U$ and $\tau\in\{0,\dots,N-1\}$. Then, by \eqref{Construction_Step2} we have 
\begin{equation}\label{Case1_Eq1}
\|S_{\omega}^{kN} e(\omega)\|=\|S_{\theta^{\tau}\widehat\omega}^{kN}e(\theta^{\tau}\widehat\omega)\|\geq \eps^{N-\tau} \|S_{\theta^N\widehat\omega}^{kN-(N-\tau)}e(\theta^N\widehat\omega)\|.
\end{equation}
Note that $\theta^N\widehat\omega\in \theta^N U$ and from the invariance of $e(\omega)$, \eqref{Nstepnorm} and \eqref{Nstepnorm}  we derive that 
\begin{align}
&\|S_{\theta^N\widehat\omega}^{(k-1)N+\tau}e(\theta^N\widehat\omega)\|\notag\\[1ex]
=&\|S_{\theta^{kN}\widehat\omega}^{\tau}e(\theta^{kN}\widehat\omega)\| \|S^{N}_{\theta^{(k-1)N}\widehat\omega}e(\theta^{(k-1)N}\widehat\omega)\|\dots\|S^N_{\theta^N\widehat\omega}e(\theta^N\widehat\omega)\|\notag\\[1ex]
\geq& 
\eps^{k-1}\|S_{\theta^{kN}\widehat\omega}^{\tau}e(\theta^{kN}\widehat\omega)\| \|S^{N}_{\theta^{(k-1)N}\widehat\omega}\|\dots\|S^N_{\theta^{2N}\widehat\omega}\|\|S^N_{\theta^N\widehat\omega}\|.\label{FirstEstimate}
\end{align}
By \eqref{Nstepnorm} we have 
\[
\|S_{\theta^{kN}\widehat\omega}^{\tau}e(\theta^{kN}\widehat\omega)\|\geq \frac{\eps^{N+1}}{\|S_{\theta^{kN+\tau}\widehat\omega}^{N-\tau}\|}
\geq 
\frac{\eps^{N+1}}{\|S\|_{\infty}^{N-\tau}},
\]
which together with the fact that $\|S_{\theta^{kN}\widehat\omega}^{\tau}\|\leq \|S\|_{\infty}^{\tau}$ implies that 
\[
\|S_{\theta^{kN}\widehat\omega}^{\tau}e(\theta^{kN}\widehat\omega)\|
\geq
\frac{\eps^{N+1}}{\|S\|_{\infty}^{N}} \||S_{\theta^{kN}\widehat\omega}^{\tau}\|.
\]
Combining the above inequality with \eqref{Case1_Eq1} and \eqref{FirstEstimate} yields that 
\begin{align*}
\|S_{\omega}^{kN}e(\omega)\|
\geq &
\eps^{N-\tau}\eps^{k-1}
\frac{\eps^{N+1}}{\|S\|_{\infty}^N} \||S_{\theta^{kN}\widehat\omega}^{\tau}\|\|S^{N}_{\theta^{(k-1)N}\widehat\omega}\|\dots\|S^N_{\theta^{2N}\widehat\omega}\|\|S^N_{\theta^N\widehat\omega}\|\\[1ex]
\geq &
\frac{\eps^{2N+k-\tau}}{\|S\|_{\infty}^{N-\tau}}\|S_{\theta^N\widehat\omega}^{(k-1)N+\tau}\|.
\end{align*}
Note that $\|S_{\omega}^{kN}\|=\|S_{\theta^N\widehat\omega}^{(k-1)N+\tau}S_{\theta^{\tau}\widehat\omega}^{N-\tau}\|
\leq \|S_{\theta^N\widehat\omega}^{(k-1)N+\tau}\|\|S\|_{\infty}^{N-\tau}$. Thus, we arrive at 
\[
\|S_{\omega}^{kN}e(\omega)\|
\geq 
\frac{\eps^{2N+k-\tau}}{\|S\|_{\infty}^{2(N-\tau)}}\|S_{\omega}^{kN}\|
\geq 
\frac{\eps^{2N+k}}{( 1+2\|T\|_{\infty})^{2N}}\|S_{\omega}^{kN}\|.
\]
It means that \eqref{Goodgrowthrate} holds for all $\omega\in \bigcup_{j=0}^{N-1}\theta^j U$.

\noindent
\emph{Case 2}: $\omega\in \bigcup_{i=kN+1}^{\infty}  \Big(E_i\cup \theta E_i\dots\cup\theta^{i-1} E_i\Big)$. Let $i\geq kN+1$ and $\tau\in \{0,\dots,i-1\}$ such that $\omega\in\theta^{\tau }E_i$. We write $\omega=\theta^{\tau}\widehat\omega$ for some $\widehat\omega\in E_i$ and let  $i=\ell N+r, \tau=sN+t$, where $s,\ell\in\N$ and $r,t\in \{0,\dots,N-1\}$. To estimate $\|S_{\omega}^{kN}e(\omega)\|=
\|S^{kN}_{\theta^{sN+t}\widehat\omega}e(\theta^{sN+t}\widehat\omega)\|$, we consider three subcases $\tau+kN\leq i-1$ (Subcase 2a), $i\leq \tau+kN\leq i+N-1$ (Subcase 2b) and $\tau+kN\geq i+N$ (Subcase 2c).

\emph{Subcase 2a}: $\tau+kN\leq i-1$. Then, by invariance of $e(\omega)$ we have 
\begin{align*}
&\|S^{kN}_{\theta^{sN+t}\widehat\omega}e(\theta^{sN+t}\widehat\omega)\|\\[1ex]
=&\|S^{N-t}_{\theta^{sN+t}\widehat\omega}e(\theta^{sN+t}\widehat\omega)\|\|S^{(k-1)N}_{\theta^{(s+1)N}\widehat\omega}e(\theta^{(s+1)N}\widehat\omega)\|
\|S^{t}_{\theta^{(k+s)N}\widehat\omega}e(\theta^{(k+s)N}\widehat\omega)\|,
\end{align*}
which together with \eqref{Construction_Step1a} and \eqref{Nstepnorm} implies that 
\begin{align}
&\|S^{kN}_{\theta^{sN+t}\widehat\omega}e(\theta^{sN+t}\widehat\omega)\|\notag
\\[1ex]
\geq& 
\eps^{k-1} \|S^{N-t}_{\theta^{sN+t}\widehat\omega}e(\theta^{sN+t}\widehat\omega)\|\|S^{(k-1)N}_{\theta^{(s+1)N}\widehat\omega}\|
\|S^{t}_{\theta^{(k+s)N}\widehat\omega}e(\theta^{(k+s)N}\widehat\omega)\|.
\label{Case2a_Eq1}
\end{align}
Using again  \eqref{Construction_Step1a} and \eqref{Nstepnorm}, we obtain that 
\begin{equation}\label{Subcase2a:Eq2}
\|S^{N-t}_{\theta^{sN+t}\widehat\omega}e(\theta^{sN+t}\widehat\omega)\|=\frac{\|S^{N}_{\theta^{sN}\widehat\omega}e(\theta^{sN}\widehat\omega)\|}{\|S^{t}_{\theta^{sN}\widehat\omega}e(\theta^{sN}\widehat\omega)\|}
\geq 
\frac{\eps^{N+1}}{\|S\|_{\infty}^t}.
\end{equation}
It remains to estimate $\|S^{t}_{\theta^{(k+s)N}\widehat\omega}e(\theta^{(k+s)N}\widehat\omega)\|$. Since $\tau+k\ell\leq i-1$ it follows that  either $k+s<\ell$ or $k+s=\ell$. Concerning the first case, i.e. $k+s<\ell$, from  \eqref{Construction_Step1a} and \eqref{Nstepnorm} we derive that 
\[
\|S^{t}_{\theta^{(k+s)N}\widehat\omega}e(\theta^{(k+s)N}\widehat\omega)\|
\geq 
\frac{\|S^{N}_{\theta^{(k+s)N}\widehat\omega}e(\theta^{(k+s)N}\widehat\omega)\|}{\|S^{N-t}_{\theta^{(k+s)N}\widehat\omega}e(\theta^{(k+s)N}\widehat\omega)\|}
\geq\frac{\eps^{N+1}}{\|S\|_{\infty}^{N-t}},
\]
which together with \eqref{Case2a_Eq1} and \eqref{Subcase2a:Eq2} implies that 
\begin{equation}\label{Firstcase}
\|S^{kN}_{\theta^{sN+t}\widehat\omega}e(\theta^{sN+t}\widehat\omega)\|
\geq \frac{\eps^{2N+k+1}}{\|S\|_{\infty}^N}\|S^{kN}_{\omega}\|
\geq 
\frac{\eps^{2N+k+1}}{(1+2\|T\|_{\infty})^N}\|S^{kN}_{\omega}\|,
\end{equation}
where we used \eqref{BoundonS} to obtain the last inequality. Meanwhile, in the other case, i.e. $k+s=\ell$, by using \eqref{Remainder} we obtain that 
\[
\|S^{t}_{\theta^{(k+s)N}\widehat\omega}e(\theta^{(k+s)N}\widehat\omega)\|
\geq 
\eps^t
\geq 
\eps^t\frac{\|S^{t}_{\theta^{(k+s)N}\widehat\omega}\|}{\|S\|_{\infty}^{t}},
\]
which together with \eqref{Case2a_Eq1} and \eqref{Subcase2a:Eq2} implies that 
\begin{equation}\label{Secondcase}
\|S^{kN}_{\theta^{sN+t}\widehat\omega}e(\theta^{sN+t}\widehat\omega)\|
\geq \frac{\eps^{N+k+t}}{\|S\|_{\infty}^N}\|S^{kN}_{\omega}\|
\geq 
\frac{\eps^{2N+k}}{(1+2\|T\|_{\infty})^N}\|S^{kN}_{\omega}\|,
\end{equation}
where we used \eqref{BoundonS} to obtain the last inequality. Combining \eqref{Firstcase} and \eqref{Secondcase}, we obtain that in this subcase 
\[
\|S^{kN}_{\theta^{sN+t}\widehat\omega}e(\theta^{sN+t}\widehat\omega)\|
\geq 
\frac{\eps^{2N+k+1}}{(1+2\|T\|_{\infty})^N}\|S^{kN}_{\omega}\|,
\]
which means that  \eqref{Goodgrowthrate} holds. 

\emph{Subcase 2b}: $i\leq \tau+kN\leq i+N-1$. Hence, from $i=\ell N+r, \tau=sN+t$ with $r,t\in \{0,\dots,N-1\}$ we have 
\begin{equation}\label{estimateindex}
0\leq kN+\tau-i\leq N-1,\quad (\ell-s)\leq k    
\end{equation}
By invariance of $e(\omega)$ we have 
\begin{align*}
\|S^{kN}_{\theta^{sN+t}\widehat\omega}e(\theta^{sN+t}\widehat\omega)\|
&=\|S^{N-t}_{\theta^{sN+t}\widehat\omega}e(\theta^{sN+t}\widehat\omega)\|\|S^{(\ell-s-1)N}_{\theta^{(s+1)N}\widehat\omega}e(\theta^{(s+1)N}\widehat\omega)\|\\[1ex]
&\times
\|S^{r}_{\theta^{\ell N}\widehat\omega}e(\theta^{\ell N}\widehat\omega)\|\|S^{kN+\tau-i}_{\theta^i\widehat\omega}e(\theta^i\widehat\omega)\|.
\end{align*}
Using \eqref{Construction_Step1a} and  \eqref{Nstepnorm}, we obtain that 
\begin{align}
\|S^{N-t}_{\theta^{sN+t}\widehat\omega}e(\theta^{sN+t}\widehat\omega)\|
&= \frac{\|S^{N}_{\theta^{sN}\widehat\omega}e(\theta^{sN}\widehat\omega)\|}{\|S^{t}_{\theta^{sN}\widehat\omega}e(\theta^{sN}\widehat\omega)\|}\notag\\
&\geq \frac{\eps^{N+1}}{\|S\|_{\infty}^N}\|S^{N-t}_{\theta^{sN+t}\widehat\omega}\|,\label{Subcase2b_Eq1}\\[1ex]
\|S^{(\ell-s-1)N}_{\theta^{(s+1)N}\widehat\omega}e(\theta^{(s+1)N}\widehat\omega)\|
&=\prod_{j=s+1}^{\ell-1}\|S_{\theta^{jN}\widehat\omega}e(\theta^{jN}\widehat\omega)\|\notag\\
&\geq 
\eps^{\ell-s-1} \|S^{(\ell-s-1)N}_{\theta^{(s+1)N}\widehat\omega}\|\label{Subcase2b_Eq2}.
\end{align}
Meanwhile, by \eqref{Remainder} we have 
\begin{equation}\label{Subcase2b_Eq3}
\|S^{r}_{\theta^{\ell N}\widehat\omega}e(\theta^{\ell N}\widehat\omega)\|
\geq \eps^r\geq \eps^r \frac{\|S^{r}_{\theta^{\ell N}\widehat\omega}\|}{\|S\|_{\infty}^r}.
\end{equation}
Finally, note that $\theta^i\widehat\omega\in U$ and thus by using \eqref{Construction_Step2} we gain that 
\begin{equation}\label{Subcase2b_Eq4}
\|S^{kN+\tau-i}_{\theta^i\widehat\omega}e(\theta^i\widehat\omega)\|\geq \eps^{kN+\tau-i}
\geq \eps^{kN+\tau-i}\frac{\|S^{kN+\tau-i}_{\theta^i\widehat\omega}\|}{\|S\|_{\infty}^{kN+\tau-i}}.
\end{equation}
Combining \eqref{Subcase2b_Eq1}, \eqref{Subcase2b_Eq2}, \eqref{Subcase2b_Eq3} and \eqref{Subcase2b_Eq4} yields that \[
\|S^{kN}_{\theta^{sN+t}\widehat\omega}e(\theta^{sN+t}\widehat\omega)\|
\geq 
\frac{\eps^{N+(\ell-s)+r+(kN+\tau-i)}}{\|S\|_{\infty}^{N+r+(kN+\tau-i)}}\|S^{kN}_{\theta^{sN+t}\widehat\omega}\|,
\]
which together with \eqref{estimateindex} and \eqref{BoundonS} implies that 
\[
\|S^{kN}_{\theta^{sN+t}\widehat\omega}e(\theta^{sN+t}\widehat\omega)\|
\geq 
\frac{\eps^{3N+k}}{(1+2\|T\|_{\infty})^{3N}}\|S^{kN}_{\theta^{sN+t}\widehat\omega}\|.
\]
Thus, \eqref{Goodgrowthrate} holds in this subcase. 

\emph{Subcase 2c}: $ \tau+kN\geq i+N$.  Hence, from $i=\ell N+r, \tau=sN+t$ with $r,t\in \{0,\dots,N-1\}$ we have 
\begin{equation}\label{estimateindex_casec}
0\leq (\ell-s)\leq k-1.    
\end{equation}
By invariance of $e(\omega)$ we have 
\begin{equation}\label{Twoparts}
\|S^{kN}_{\theta^{sN+t}\widehat\omega}e(\theta^{sN+t}\widehat\omega)\|=  \|S^{i-(sN+t)}_{\theta^{sN+t}\widehat\omega}e(\theta^{sN+t}\widehat\omega)\|  \|S^{kN+(sN+t)-i}_{\theta^{i}\widehat\omega}e(\theta^{i}\widehat\omega)\|.
\end{equation}
On the one hand, analogous to the estimates in \eqref{Subcase2b_Eq1}, \eqref{Subcase2b_Eq2} and \eqref{Subcase2b_Eq3} we have 
\begin{equation}\label{Estimate_Part1}
  \|S^{i-(sN+t)}_{\theta^{sN+t}\widehat\omega}e(\theta^{sN+t}\widehat\omega)\| 
  \geq \frac{\eps^{N+r+(\ell-s)}}{\|S\|_{\infty}^{N+r}}\|S^{i-(sN+t)}_{\theta^{sN+t}\widehat\omega}\|.
\end{equation}
On the other hand,  by denoting $\widetilde\omega:=\theta^i\widehat\omega$ we have $\widetilde \omega\in U$ and from  invariance of $e(\omega)$ we derive that 
\begin{align}
\|S^{kN+(sN+t)-i}_{\theta^{i}\widehat\omega}e(\theta^{i}\widehat\omega)\|
&=\|S_{\widetilde\omega}^{(k+s-\ell)N+(t-r)}e(\widetilde\omega)\|\notag\\[1ex]
&=
\|S_{\theta^N\widetilde\omega}^{(k+s-\ell-1)N+(t-r)}e(\theta^N\widetilde\omega)\|
\|S_{\widetilde\omega}^{N}e(\widetilde\omega)\|.\label{New_Estimate1}
\end{align}
Thanks to the construction of $S(\omega)$ for $\omega\in\bigcup_{j=0}^{N-1}\theta^j U$ as in Step 2, we have \begin{equation}\label{New_Estimate2}
 \|S_{\widetilde\omega}^{N}e(\widetilde\omega)\|
 \geq 
 \eps^N\geq \frac{\eps^N}{\|S\|_{\infty}^N}\|S^N_{\widetilde\omega}\|.
\end{equation}
Thus, to estimate $\|S^{kN+(sN+t)-i}_{\theta^{i}\widehat\omega}e(\theta^{i}\widehat\omega)\|$ it remains to estimate $\|S_{\theta^N\widetilde\omega}^{(k+s-\ell-1)N+(t-r)}e(\theta^N\widetilde\omega)\|$. In fact, we show that 
\begin{equation}\label{Smallaim}
  \|S_{\theta^N\widetilde\omega}^{(k+s-\ell-1)N+(t-r)}e(\theta^N\widetilde\omega)\| 
  \geq \frac{\eps^{k+s-\ell+N}}{\|S\|_{\infty}^N}       \|S_{\theta^N\widetilde\omega}^{(k+s-\ell-1)N+(t-r)}\|.
\end{equation}
For this purpose, we consider two separated cases $(t-r)\geq 0$ or $(t-r)<0$:
\begin{itemize}
    \item If $(t-r)\geq 0$ then
    \begin{align*}
       &\|S_{\theta^N\widetilde\omega}^{(k+s-\ell-1)N+(t-r)}e(\theta^N\widetilde\omega)\|\\[1ex]
       &= \|S_{\theta^N\widetilde\omega}^{(k+s-\ell-1)N}e(\theta^N\widetilde\omega)\|
       \frac{\|S^{N}_{\theta^{(k+s-\ell)N}\widetilde\omega}e(\theta^{(k+s-\ell)N}\widetilde\omega)\|}{\|S^{N-(t-r)}_{\theta^{(k+s-\ell)N+(t-r)}\widetilde\omega}e(\theta^{(k+s-\ell)N+(t-r)}\widetilde\omega)\|},
    \end{align*}
which together with \eqref{Construction_Step1a} and \eqref{Nstepnorm} implies that  \begin{align*}
       \|S_{\theta^N\widetilde\omega}^{(k+s-\ell-1)N+(t-r)}e(\theta^N\widetilde\omega)\|
       &\geq \eps^{k+s-\ell-1}\|S_{\theta^N\widetilde\omega}^{(k+s-\ell-1)N}\|\frac{\eps^{N+1}}{\|S\|_{\infty}^{N-(t-r)}},\\[1ex]
       &\geq 
       \frac{\eps^{k+s-\ell+N}}{\|S\|_{\infty}^N}       \|S_{\theta^N\widetilde\omega}^{(k+s-\ell-1)N+(t-r)}\|.
    \end{align*}
Thus, \eqref{Smallaim} is proved in this case.
\item If $(t-r)<0$ then 
 \begin{align*}
       &\|S_{\theta^N\widetilde\omega}^{(k+s-\ell-1)N+(t-r)}e(\theta^N\widetilde\omega)\|\\[1ex]
       &= \|S_{\theta^N\widetilde\omega}^{(k+s-\ell-2)N}e(\theta^N\widetilde\omega)\|
       \frac{\|S^{N}_{\theta^{(k+s-\ell-1)N}\widetilde\omega}e(\theta^{(k+s-\ell)N}\widetilde\omega)\|}{\|S^{r-t}_{\theta^{(k+s-\ell)N+(t-r)}\widetilde\omega}e(\theta^{(k+s-\ell)N+(t-r)}\widetilde\omega)\|},
    \end{align*}
which together with \eqref{Construction_Step1a} and \eqref{Nstepnorm} implies that  \begin{align*}
       \|S_{\theta^N\widetilde\omega}^{(k+s-\ell-1)N+(t-r)}e(\theta^N\widetilde\omega)\|
       &\geq \eps^{k+s-\ell-2}\|S_{\theta^N\widetilde\omega}^{(k+s-\ell-2)N}\|\frac{\eps^{N+1}}{\|S\|_{\infty}^{r-t}},\\[1ex]
       &\geq 
       \frac{\eps^{k+s-\ell+N-1}}{\|S\|_{\infty}^N}       \|S_{\theta^N\widetilde\omega}^{(k+s-\ell-1)N+N-(r-t)}\|\\
       &\geq \frac{\eps^{k+s-\ell+N}}{\|S\|_{\infty}^N}       \|S_{\theta^N\widetilde\omega}^{(k+s-\ell-1)N+(t-r)}\|
    \end{align*}
Thus, \eqref{Smallaim} is proved.
\end{itemize}
Combining \eqref{Smallaim} with \eqref{New_Estimate1} and \eqref{New_Estimate2}, we obtain that 
\[
\|S^{kN+(sN+t)-i}_{\theta^{i}\widehat\omega}e(\theta^{i}\widehat\omega)\|
\geq 
       \frac{\eps^{k+s-\ell+2N}}{\|S\|_{\infty}^{2N}} \|S^{kN+(sN+t)-i}_{\theta^{i}\widehat\omega}\|,
\]
which together with \eqref{Twoparts} and \eqref{Estimate_Part1} leads to
\begin{align*}
\|S^{kN}_{\theta^{sN+t}\widehat\omega}e(\theta^{sN+t}\widehat\omega)\|
& \geq \frac{\eps^{k+3N+r}}{\|S\|_{\infty}^{3N+r}}\|S^{kN}_{\theta^{sN+t}\widehat\omega}\|\\[1ex]
&\geq 
\frac{\eps^{k+4N}}{(1+2\|T\|_{\infty})^{4N}} \|S^{kN}_{\theta^{sN+t}\widehat\omega}\|,
\end{align*}
where we used \eqref{BoundonS} and the fact that $\eps\in (0,\frac{1}{6})$ and $r<N$ to obtain the last inequality. It means that \eqref{Goodgrowthrate} also holds in this case, and the proof is complete.
\end{proof}

\begin{proof}[Proof of Theorem A (iii)]
Let $T\in\mathcal N$ and $\eta>0$ be arbitrary.  To complete the proof, it is sufficient to construct $\widetilde T\in  \mathcal K_{\infty}(\Omega;\mathcal H)$ and $\delta>0$ such that 
\begin{equation}\label{Requirement}
B_{\delta}(\widetilde T) \subset B_{\eta}(T)\quad \hbox{ and } \quad B_{\delta}(\widetilde T) \cap \mathcal N=\emptyset.
\end{equation}
For this purpose, choose and fix $\eps, \alpha>0$ such that 
\begin{equation}\label{Constant_EpsilonAlpha}
\eps\leq\min\left(\frac{1}{6}, \frac{\eta}{9(1+\|T\|_{\infty})}\right),\qquad (e^{2\alpha}-1)(1+\|T\|_{\infty})\leq \frac{\eta}{3}.   
\end{equation}
Let $N^*(T,\eps)$ be a natural number satisfying the conclusion of Proposition \ref{TechnicalProposition_1}. Thus, for all $N\geq N^*(T,\eps)$ the conclusion of Proposition \ref{Millionschikov} holds. Choose and fix $N\geq N^*(T,\eps)$ such that $ e^{-\frac{\alpha}{2} N}\leq \eps$. Since $e^{-\alpha N}<e^{-\frac{\alpha}{2} N}\leq \eps$ it follows that there exists $k\in\N$ such that
\begin{equation}\label{Constant_k}
e^{-\alpha kN}
\leq 
\frac{\eps^{k+4N}}{(1+2\|T\|_{\infty})^{4N}}.
\end{equation}
Using Proposition \ref{Millionschikov}, there exist $S\in\cK_{\infty}(\Omega,\cH)$ and an invariant unit random vector $e:\Omega\rightarrow \cH$ of $S$ such that $ \|S-T\|_{\infty}\leq 3\eps (1+ \|T\|_{\infty})$ and 
\begin{align}
\|S(\omega)e(\omega)\|
& \geq \frac{\eps^{N+1}}{(1+2\|T\|_{\infty})^{N-1}}\quad\hbox{ for } \omega\in\Omega,\label{EstimateRho}\\[1ex]
\|S^{kN}_{\omega}e(\omega)\|
&\geq  \frac{\eps^{k+4N}}{(1+2\|T\|_{\infty})^{4N}} \|S_{\omega}^{kN}\|\quad\hbox{ for } \omega\in\Omega.\notag
\end{align}
This together with \eqref{Constant_EpsilonAlpha} and \eqref{Constant_k} implies that $ \|S-T\|_{\infty}\leq \frac{\eta}{3}$ and 
\begin{equation}\label{BetterGrowthRate}
\|S^{kN}_{\omega}e(\omega)\|\geq  e^{-\alpha kN} \|S_{\omega}^{kN}\|\quad\hbox{ for } \omega\in\Omega.
\end{equation}
In light of \eqref{EstimateRho} and \eqref{BetterGrowthRate}, the random compactor operator $S$ satisfies all assumptions of Proposition \ref{TechnicalProposition_C}. Thus, by virtue of Proposition \ref{TechnicalProposition_C} there exists $\widetilde T\in \mathcal K_{\infty}(\Omega;\mathcal H)$ such that the top Lyapunov exponent function $\kappa(\cdot)$ is continuous at $\widetilde T$ and 
\[
\|\widetilde T-S\|_{\infty}\leq (e^{2\alpha}-1)\|S\|_{\infty}.
\]
Using $\|S-T\|_{\infty}\leq \frac{\eta}{3}$, $\eta\in (0,1)$ and \eqref{Constant_EpsilonAlpha}, we obtain that 
\[
\|\widetilde T-S\|_{\infty}\leq (e^{2\alpha}-1)(\frac{\eta}{3}+\|T\|_{\infty})\leq \frac{\eta}{3}.
\]
Consequently, $\|\widetilde T-T\|_{\infty}\leq \frac{2\eta}{3}$. This, together with the fact that the top Lyapunov exponent function $\kappa(\cdot)$ is continuous at $\widetilde T$, implies that there exists $\delta>0$ satisfying \eqref{Requirement}. The proof is complete.
\end{proof}
\section{Proof of Theorem B}\label{Section_ProofB}
\subsection{Dominated splitting}
\subsubsection{Dominated splitting and integral separation for linear random dynamical systems on a finite-dimensional space}This subsection is devoted to recalling the notion and the result on the genericity of dominated splitting and integral separation for linear random dynamical systems in a finite-dimensional space $\R^d$, see \cite{Bochi2002,Cong05,Cong_Doan16} for more details.

For each $d\in \N$, let $\mathcal L^{\infty}(\Omega,\R^{d\times d})$ denote the set of bounded measurable maps $A:\Omega\rightarrow \R^{d\times d}$. Each random  map $A\in \mathcal L^{\infty}(\Omega,\R^{d\times d})$  gives rise to an one-sided linear random dynamical system $\Phi_A:\N\times \Omega \rightarrow \R^{d\times d}$ via
\[
\Phi_A(n,\omega)
:=
\left\{
\begin{array}{ll}
A(\theta^{n-1}\omega)\dots A(\omega),
& \quad \hbox{ if }  n>0,\\[1.5ex]
\id, & \quad \hbox{ if }  n=0,
\end{array}
\right.
\]
where $\id$ denotes the identity matrix in $\R^{d\times d}$. By virtue of the Multiplicative Ergodic Theorem for one-sided linear random dynamical systems defined on an invertible metric dynamical system, see e.g. \cite{Froyland}, there exist finite numbers of Lyapunov exponents $\lambda_1(A)>\dots >\lambda_{p(A)}(A)$, where $p(A)\in \{1,\dots, d\}$, and a forward invariant decomposition
\[
\R^d=\mathcal O^1_{\omega}(A)\oplus \mathcal O^2_{\omega}(A)\oplus\dots\oplus \mathcal O^{p(A)}_{\omega}(A)
\]
with the property that
\[
\lim_{n\to \infty} \frac{1}{n}\log\|\Phi_A(n,x)v\|=\lambda_i\quad\hbox{iff}\quad v\in \mathcal O^i_A(\omega)\setminus\{0\}.
\]
When $p(A)=d$, we call the Lyapunov spectrum of $A$ \emph{simple}.  We now recall a result on the density of simple Lyapunov spectra of linear random dynamical systems in a finite-dimensional space in \cite{Arnold_Cong_1999}.
\begin{theorem}\label{Density}
 For $d\in\N$, let   $A\in \cL^{\infty}(\Omega,\R^{d\times d})$ and $\eps>0$ be arbitrary. Then, for any  measurable set $U\subset \Omega$ of positive probability, there exists $B\in \cL^{\infty}(\Omega,\R^{d\times d})$ satisfying the following properties:
\begin{itemize}
\item [(i)] $\mbox{ess}\sup_{\omega\in\Omega}\|B(\omega)-A(\omega)\|\leq \eps$ and $
B(\omega)=A(\omega)$ for all  $\omega\in\Omega\setminus U$,	
\item [(ii)] The Lyapunov spectrum of $B$ is simple. 	
\end{itemize}
\end{theorem}
\begin{definition}[Dominated splitting]\label{IntegralSeparation}
Let  $A\in \cL^{\infty}(\Omega,\R^{d\times d})$. An invariant measurable decomposition $\R^d=E(\omega)\oplus F(\omega)$ of $A$ is said to be a \emph{dominated splitting} if there exist $K,\alpha,\delta>0$ such that for $\mP- \hbox{a.s. } \omega\in\Omega$ and $n\in\N$ the following statements hold
\begin{align}
\|A(\omega)v\|
& \geq
\delta \|v\|\quad \hbox{for }  
  v\in E(\omega)\setminus\{0\},\label{EquivalenceForm_Eq1a}\\[1ex]
\frac{\|\Phi_A(n,\omega)v\|}{\|v\|}
& \geq
K e^{\alpha n}
 \frac{\|\Phi_A(n,\omega)u\|}{\|u\|}\quad \hbox{for }  
 v\in E(\omega)\setminus\{0\}, u\in F(\omega)\setminus\{0\}.\label{EquivalenceForm_Eq1b}
\end{align}
\end{definition}
\begin{remark}\label{Equivalence_DominatedSplitting}
(i) The condition \eqref{EquivalenceForm_Eq1b} can be replaced by an equivalent statement that 
there exists $N\in\N$ and $\rho>1$ such that for $\mP- \hbox{a.s. } \omega\in\Omega$
\begin{equation}\label{EquivalenceForm_Eq2}
\frac{\|\Phi_A(N,\omega)v\|}{\|v\|}
\geq
\rho 
 \frac{\|\Phi_A(N,\omega)u\|}{\|u\|} \quad \hbox{for }  
 v\in E(\omega)\setminus\{0\}, u\in F(\omega)\setminus\{0\}.    
 \end{equation}

\begin{proof}
($\Rightarrow$) Suppose that \eqref{EquivalenceForm_Eq1b}. Then, choose and fix an arbitrary $N\in\N$ such that $Ke^{\alpha N}<1$. Therefore, \eqref{EquivalenceForm_Eq2} holds for $\rho:=K e^{\alpha N}$.

\noindent 
($\Leftarrow$) In this direction, we assume that  \eqref{EquivalenceForm_Eq2} holds. Let  $n\in\N$ and  $u\in E(\omega)\setminus\{0\}, v\in F(\omega)\setminus\{0\}$ be arbitrary. We write $n=sN+r$, where $r\in\{0,\dots,N-1\}$. Then, by cocycle property $\Phi_A(n+m,\omega)=\Phi_A(n,\theta^m\omega)\Phi_A(m,\omega)$ and \eqref{EquivalenceForm_Eq1a}, \eqref{EquivalenceForm_Eq2} we have  for $\mP- \hbox{a.s. } \omega\in\Omega$ and  $v\in E(\omega)\setminus\{0\}, u\in F(\omega)\setminus\{0\}$
\begin{align}
\frac{\|\Phi_A(n,\omega)v\|}{\|v\|}
& =\frac{\|\Phi_A(r,\theta^{sN}\omega)\Phi_A(sN,\omega)v\|}{\|\Phi_A(sN,\omega)v\|}\frac{\|\Phi_A(sN,\omega)v\|}{\|v\|}\notag\\
& \geq \delta^r \rho^s  \frac{\|\Phi_A(sN,\omega)u\|}{\|u\|}\notag\\[1ex]
& \geq \frac{\delta^r}{\|A\|_{\infty}^r}\rho^{\frac{n}{N}-1} \frac{\|\Phi_A(n,\omega)u\|}{\|u\|}.
\end{align}
It means that \eqref{EquivalenceForm_Eq1b} holds for $K:=\frac{\delta^r}{\rho \|A\|_{\infty}^r}$ and $\alpha:=\frac{\log\rho}{N}$.
\end{proof}

(ii) There is a characterization of dominated splitting in terms of exponential separation of singular values, see  \cite{Bochi09,Blumenthal19,Quas_19}.
 \end{remark}
\begin{definition}[Integral separation]
 $A$ is said to be \emph{integrally separated} if there exist  an invariant measurable decomposition $\R^d=\bigoplus_{i=1}^d E_i(\omega)$, where $E_i(\omega)$ is a linear subspace of dimension $1$, and positive constants $K,\alpha,\delta$ such that for $\mP- \hbox{a.s. } \omega\in\Omega$ and $n\in\N$ the following statements hold
 \begin{align}
\|A(\omega)v\|
& \geq \delta\|v\|\quad \hbox{ for } v\in \R^d\setminus\{0\},
\label{Integralseparaction_C1}\\[1ex]     
\frac{\|\Phi_A(n,\omega)v\|}{\|v\|}
& \geq
K e^{\alpha n}
 \frac{\|\Phi_A(n,\omega)u\|}{\|u\|}\hbox{ for }  v\in \bigoplus_{j=1}^i E_j(\omega)\setminus\{0\}, u\in \bigoplus_{j=i+1}^dE_{j}(\omega)\setminus\{0\},
 \label{Integralseparaction_C2}
 \end{align}
where $i=1,\dots,d-1$.
\end{definition}
The following remark is devoted to presenting a way to check integral separation inductively. 
\begin{remark}\label{Checkingintegralseparation}
Condition \eqref{Integralseparaction_C2} can be replaced by that the restriction of $A$ on $\bigoplus_{i=1}^{d-1} E_i(\omega)$ is integrally separated, i.e. there exist $K,\alpha>0$ such that for $i\in\{1,\dots,d-2\}$
\begin{equation}\label{IntegralSeparation_Equivalence1}
\frac{\|\Phi_A(n,\omega)v\|}{\|v\|}
 \geq
K e^{\alpha n}
 \frac{\|\Phi_A(n,\omega)u\|}{\|u\|}\hbox{ for }  v\in \bigoplus_{j=1}^i E_j(\omega)\setminus\{0\}, u\in \bigoplus_{j=i+1}^{d-1}E_{j}(\omega)\setminus\{0\}    
\end{equation}
and the decomposition $\R^d= \left(\bigoplus_{i=1}^{d-1} E_i(\omega)\right)\oplus E_d(\omega)$ is dominated with respect to $A$, i.e. there exist $L,\beta>0$ such that 
\begin{equation}\label{IntegralSeparation_Equivalence2}
\frac{\|\Phi_A(n,\omega)v\|}{\|v\|}
 \geq
L e^{\beta n}
 \frac{\|\Phi_A(n,\omega)u\|}{\|u\|}\hbox{ for }  v\in \bigoplus_{j=1}^{d-1} E_j(\omega)\setminus\{0\}, u\in E_{d}(\omega)\setminus\{0\}.    
\end{equation}
\end{remark}
\begin{proof}
Suppose that \eqref{IntegralSeparation_Equivalence1} and \eqref{IntegralSeparation_Equivalence2} hold. Then, \eqref{Integralseparaction_C2} holds for $i=d-1$. To conclude the proof, we choose and fix an arbitrary $i\in\{1,\dots,d-2\}$. Let $v\in \bigoplus_{j=1}^i E_j(\omega), u\in \bigoplus_{j=i+1}^{d}E_{j}(\omega)$ with $\|v\|=\|u\|=1$. We write $u=u_1+u_2$ with $u_1\in \bigoplus_{j=i+1}^{d-1}E_{j}(\omega), u_2\in E_{d}(\omega)$. It is noted that the angle between $\bigoplus_{j=i+1}^{d-1}E_{j}(\omega)$ and $E_{d}(\omega)$ is uniformly separated from zero, see \cite{Cong05}. Then, there exists $M>0$ (being independently with $u$) such that $\|u_1\|,\|u_2\|\leq M$. Hence, using \eqref{IntegralSeparation_Equivalence1} and \eqref{IntegralSeparation_Equivalence2}, we obtain that 
\[
\|\Phi_A(n,\omega)v\|\geq \frac{K}{M} e^{\alpha n} \|\Phi_A(n,\omega)u_1\|,\quad  \|\Phi_A(n,\omega)v\|\geq \frac{L}{M} e^{\beta n} \|\Phi_A(n,\omega)u_2\|,
\]
which together with $\|\Phi_A(n,\omega) u_1\|+\|\Phi_A(n,\omega) u_2\|\geq \|\Phi_A(n,\omega)u\|$ implies that 
\[
\|\Phi_A(n,\omega)v\|\geq
\widetilde{K} e^{{\widetilde \alpha}n}\|\Phi_A(n,\omega)v\|,\quad \hbox{ where } \widetilde K:=\min\{\frac{K}{2M},\frac{L}{2M}\}, \widetilde \alpha:=\min\{\alpha,\beta\}.
\]
Thus, \eqref{Integralseparaction_C2} holds for  $i\in\{1,\dots,d-2\}$.
\end{proof}
We recall below a result in \cite{Cong05} on the density of integral separation of random dynamical systems in a finite-dimensional space.
\begin{theorem}\label{Density_IntegralSeparation}
The set of integrally separated random dynamical systems is dense in the space of bounded linear random dynamical systems $\cL^{\infty}(\Omega,\R^{d\times d})$. 
\end{theorem}
%
%

%
%
%
%

%
\subsubsection{Dominated splitting of random compact operators on an infinite-dimensional Hilbert space}
In what follows, we introduce the notion of dominated splitting of Oseledets-Ruelle decomposition for random compact operators on an infinite-dimensional Hilbert space, see \cite{Bessa2008}. Let $T\in \mathcal K_{\infty}(\Omega;\mathcal H)$ with a nontrivial Oseledets-Ruelle decomposition 
	\begin{equation*}
	\mathcal H= F^{\infty}_{\omega}(T)\oplus \bigoplus_{i=1}^{s(T)} \mathcal O^i_{\omega}(T).
	\end{equation*}
For each $k\in\{1,\dots,s(T)\}$, the decomposition corresponding to the first $k$ Lyapunov exponents of $T$ is given by 
\begin{equation}
 	\mathcal H= F^{k}_{\omega}(T)\oplus \bigoplus_{i=1}^{k} \mathcal O^i_{\omega}(T),\quad\hbox{where }  F^{k}_{\omega}(T):= F^{\infty}_{\omega}(T)\oplus \bigoplus_{i=k+1}^{s(T)} \mathcal O^i_{\omega}(T). 
\end{equation}
\begin{definition}[Dominated splitting of random compact operators on an infinite-dimensional Hilbert space]\label{Dominatedsplliting}
The Oseledets-Ruelle decomposition corresponding to the first $k$ Lyapunov exponents of $T\in  \mathcal K_{\infty}(\Omega;\mathcal H)$ is said to be a \emph{dominated splitting} if there exist $\delta>0$ and $K,\alpha>0$ such that 
\begin{itemize}
    \item [(d1)] $\|T(\omega)v\|\geq \delta \|v\|$ for all $v\in \bigoplus_{i=1}^{k} \mathcal O^i_{\omega}(T)\setminus\{0\}$,
    \item [(d2)] $\frac{\|T_{\omega}^n v\|}{\|v\|}\geq K e^{\alpha n}
    \frac{\|T_{\omega}^n u\|}{\|u\|}$ for all $n\in\N$, $v\in \bigoplus_{i=1}^{k} \mathcal O^i_{\omega}(T)\setminus\{0\}$ and $u\in F^{k}_{\omega}(T)\setminus\{0\}$.
\end{itemize}
The Oseledets-Ruelle decomposition of $T$ is said to be a \emph{dominated splitting} if there exists $k\in\{1,\dots,s(T)\}$ such that the Oseledets-Ruelle decomposition corresponding to the first $k$ Lyapunov exponents of $T$ is dominated.
\end{definition}
\begin{remark}\label{Angleseparation_Remark} 
Suppose that the Oseledets-Ruelle decomposition corresponding to the first $k$ Lyapunov exponents of $T\in  \mathcal K_{\infty}(\Omega;\mathcal H)$ is dominated with positive constants $\delta,K,\alpha$ as in the above definition. Let $N$ be a natural number such that $Ke^{\alpha N}\geq 2$. Let
\begin{equation}\label{Constanteta}
\eta:=
\min
\left\{1, \frac{\delta^{2N}}{2 K^2e^{2\alpha N} \|T\|_{\infty}^{2N}}\right\}.
\end{equation}
Then, the following statements hold:

\noindent 
\emph{Claim (i)}: For all  $v\in \bigoplus_{i=1}^{k} \mathcal O^i_{\omega}(T)\setminus\{0\}$ and $u\in F^{k}_{\omega}(T)\setminus\{0\}$ we have 
\begin{equation}\label{Angleseparation_EqNew}
\left|1- \frac{\langle v, u \rangle}{\|v\|\|u\|}\right|
\geq  
\eta.
\end{equation}
We now verify the above claim. Since $\frac{\langle v, u \rangle}{\|v\|\|u\|}$ remains the same when $v$, $u$ are replaced by $\frac{v}{\|v\|}, \frac{u}{\|u\|}$, respectively.  Thus, we can assume additionally that $\|v\|=\|u\|=1$. We write $v=\gamma u+\zeta$, where $\gamma\in \R, \zeta\in u^{\perp}$. Since $\|v\|=\|u\|=1$ it follows that $\gamma^2+\|\zeta\|^2=1$. Thus, 
\begin{equation}\label{Angleseparation_EqNew_01}
\left|1- \frac{\langle v, u \rangle}{\|v\|\|u\|}\right|
=
|1-\gamma|
\geq \frac{\|\zeta\|^2}{2}.
\end{equation}
Using (D2) and the fact that $|\gamma|\leq 1$, we obtain 
\[
\|T^N_{\omega} v\| \geq K e^{\alpha N} \|T^N_{\omega}u\|
\geq K e^{\alpha N} \left(\|T^N_{\omega} v\|-\|T^N_{\omega} \zeta\|\right),
\]
which together with the fact that $Ke^{\alpha N}\geq 2$ and (D1) implies that 
\[
\|\zeta\|
\geq \frac{\|T^N_{\omega}v\|}{Ke^{\alpha N} \|T\|_{\infty}^N}
\geq \frac{\delta^N}{Ke^{\alpha N} \|T\|_{\infty}^N}.
\]
Combining the above inequality and \eqref{Angleseparation_EqNew_01}, inequality \eqref{Angleseparation_EqNew} is verified.  

\noindent
\emph{Claim (ii)}: Let $\xi=v+u$ with $v\in \bigoplus_{i=1}^{k} \mathcal O^i_{\omega}(T)\setminus\{0\}, u\in F^{k}_{\omega}(T)\setminus\{0\}$ and $\|\xi\|=1$. Then, 
\begin{equation}\label{Consequence_upperestimate}
\|v\|,\|u\|\leq \frac{1}{\eta}.    
\end{equation}
To see the above inequality, using $\|v+u\|=1$ and \eqref{Angleseparation_EqNew} we obtain that 
\begin{align*}
1=\|v\|^2+2 \langle v, u \rangle +\|u\|^2
& \geq  
\|v\|^2+2 (-1+\eta)\|v\|\|u\|  +\|u\|^2,\\[1ex]
& = (2\eta-\eta^2) \|v\|^2+ (\|u\|-(1-\eta)\|v\|)^2,
\end{align*}
which implies that $\|v\| \leq \frac{1}{\eta}$. Similarly, we also have $\|h\| \leq \frac{1}{\eta}$. Thus, \eqref{Consequence_upperestimate} is proved.
\end{remark}
Based on the notion of dominated splitting above, we introduce a notion of integral separation of the restriction of random compactor operators on the Oseledets-Ruelle subspaces corresponding to their first $k$ Lyapunov exponents.
\begin{definition}[Integral  separation of random compact operators on an infinite-dimensional Hilbert space]\label{IntegralSeparation_Hilbertspace}
For each $k\in \N$, the restriction of a random compact operator $T\in\cK_{\infty}(\Omega,\cH)$ on the Oseledets-Ruelle subspaces corresponding to its first $k$ Lypaunov exponent  is said to be integrally separated if following properties hold
\begin{itemize}
\item [(i1)] The Oseledets-Ruelle subspaces corresponding to the first $k$ Lyapunov exponent of $T$ are simple, i.e.
\[
\dim \mathcal O^1_{\omega}(T)=\dots=\dim \mathcal O^k_{\omega}(T)=1,
\]
\item [(i2)] The Oseledets-Ruelle decomposition corresponding to the first $k$ Lyapunov exponents of $T$ is dominated,
\item [(i3)] The restriction of $T$ on the invariant subspace $\bigoplus_{i=1}^k \cO^i_{\omega}(T)$ is integrally separated. 
\end{itemize}
\end{definition}
We show below the density of random compact linear operators having dominated splitting and the restriction of these random compact linear operators on the dominated random subspaces are integrally separated.  This result provides a natural extension of the work in \cite{Cong05} from random dynamical systems in finite-dimensional spaces to infinite-dimensional spaces. 
\begin{theorem}[Density of integrally separated random compact linear operators]\label{DensityIntegralSeparation_Hilbertspace}
Let $\mathcal D_k$ be the set of random compact operators $T\in\cK_{\infty}(\Omega,\cH)$ such that the restriction of them on the Oseledets-Ruelle subspaces corresponding to their first $k$ Lyapunov exponents are integrally separated. Then, the set $\mathcal D_k$ is dense in $\cK_{\infty}(\Omega,\cH)$. 
\end{theorem}
The main ingredient in the proof of the above theorem is from \cite{Bessa2008} on the genericity of either the products of random compact operators converge to the null operator or the Oseledets-Ruelle decomposition is dominated. 
\begin{theorem}\label{Genericity_Dichotomy} There exists a residual set $\mathcal R\subset \cK_{\infty}(\Omega;\mathcal H)$ such that for any $T\in\mathcal R$ either the limit $\lim_{n\to\infty}\left(\left(T_{\omega}^n\right)^{\rT} T_{\omega}^n\right)^{\frac{1}{2n}}$ is the null operator or the Oseledets-Ruelle decomposition of $T$ is dominated.
\end{theorem}
Combining the above theorem and Theorem A (iii), we have the following corollary. 
\begin{corollary}\label{Corollary_Density}
 There exists a dense set $\mathcal D\subset \cK_{\infty}(\Omega;\mathcal H)$ such that for any $T\in\mathcal D$ the Oseledets-Ruelle decomposition of $T$ is dominated.
\end{corollary}
We improve the above Corollary by replacing the part "the Oseledets-Ruelle decomposition of $T$ is dominated" by "the Oseledets-Ruelle decomposition corresponding to the first Lyapunov exponent of $T$ is dominated".
\begin{lemma}\label{Density_FirstLemma}
There exists a dense set $\mathcal D_1\subset \cK_{\infty}(\Omega;\mathcal H)$ such that for any $T\in\mathcal D_1$  the Oseledets-Ruelle decomposition corresponding to the first Lyapunov exponent of $T$ is dominated.
\end{lemma}
\begin{proof}
By Corollary  \ref{Corollary_Density}, it is sufficient to start with a $T\in \cK_{\infty}(\Omega;\mathcal H)$ such that the Oseledets-Ruelle decomposition of $T$ is dominated and we need to verify that for any $\eps>0$ there exists $\widetilde T\in \cK_{\infty}(\Omega;\mathcal H)$ such that $\|T-\widetilde T\|_{\infty}<\eps$ and the Oseledets-Ruelle decomposition corresponding to the first Lyapunov exponent of $\widetilde T$ is dominated. By Definition \ref{Dominatedsplliting} there exist $k\in\{1,\dots,s(T)\}$, $K,\alpha>0$ and $\delta\in (0,1)$ such that 
\begin{align}
\frac{\|T(\omega)v\|}{\|v\|}& \geq \delta 
\quad \hbox{ for } v\in \bigoplus_{i=1}^{k} \mathcal O^i_{\omega}(T)\setminus\{0\},\label{Angleseparation_Eq1}\\[1ex]
\frac{\|T_{\omega}^n v\|}{\|v\|}
& \geq K e^{\alpha n}
    \frac{\|T_{\omega}^n u\|}{\|u\|}
\quad \hbox{ for } n\in\N, v\in \bigoplus_{i=1}^{k} \mathcal O^i_{\omega}(T)\setminus\{0\}, u\in F^{k}_{\omega}(T)\setminus\{0\}.\label{Speedseparation}
\end{align}
Choose and fix an arbitrary $\eps>0$. Let $N\in\N$ such that $Ke^{\alpha N}\geq 2$. Choose $\eta$ as in \eqref{Constanteta} in Remark \ref{Angleseparation_Remark}. By invariance of the random subspace $\bigoplus_{i=1}^k \cO^i_{\omega}(T)$ (see Theorem \ref{MET_HilbertSpace}) and Theorem \ref{Density_IntegralSeparation} about density of integrally separated linear random dynamical systems, there exists a linear random map $B(\omega): \bigoplus_{i=1}^k \cO^i_{\omega}(T)\rightarrow \bigoplus_{i=1}^k \cO^i_{\theta \omega}(T)$ such that 
\begin{equation}\label{Estimate_Restriction}
\|B(\omega)v-T(\omega)v\|
\leq \frac{\eps \eta}{2}\|v\| \qquad \hbox{ for all } v\in \bigoplus_{i=1}^k \cO^i_{\omega}(T)
\end{equation}
and 
\begin{equation}\label{Estimate_Restriction_II}
\frac{\|B^N_{\omega}v\|}{\|v\|}
\geq \frac{2}{3}\frac{\|T^N_{\omega}v\|}{\|v\|} \qquad \hbox{ for all } v\in \bigoplus_{i=1}^k \cO^i_{\omega}(T) \setminus\{0\}
\end{equation}
and the subspace $\bigoplus_{i=1}^k \cO^i_{\omega}(T)$ can be decomposed as the direct sum of one dimensional random subspaces $\bigoplus_{j=1}^{\ell} E_j(\omega)$ such that for all $n\in\N$ $\omega\in \Omega$ and $v\in \bigoplus_{j=1}^i E_j(\omega)\setminus\{0\}, u\in \bigoplus_{j=i+1}^dE_{j}(\omega)\setminus\{0\}$, where $i=1,\dots,d-1$, we have 
\begin{equation}\label{Integralseparation_Density_Apply}
\frac{\|B^n_{\omega}v\|}{\|v\|}
\geq
L e^{\beta n}
\frac{\|B^n_{\omega}u\|}{\|u\|},
\end{equation}
where $L,\beta>0$ are independent with $n,\omega,v,u$. Combining \eqref{Speedseparation} and \eqref{Estimate_Restriction_II}, we obtain that 
\[
\frac{\|B^N_{\omega}v\|}{\|v\|}
\geq \frac{2}{3}\frac{\|T^N_{\omega}v\|}{\|v\|}
\geq 
\frac{4}{3}\frac{\|T^N_{\omega}u\|}{\|u\|}
\qquad \hbox{ for all } v\in \bigoplus_{i=1}^k \cO^i_{\omega}(T)\setminus\{0\}, u\in F^k_{\omega}(T)\setminus\{0\}.
\]
Hence, by invariance of $\bigoplus_{i=1}^k \cO^i_{\omega}(T)$ and $F^k_{\omega}(T)$ under $B(\omega)$ and $T(\omega)$ and by using the cocycle properties, i.e. $B^{n+m}_{\omega}=B^n_{\theta^m\omega}\; B^m_{\omega}$ and $T^{n+m}_{\omega}=T^n_{\theta^m\omega}\; T^m_{\omega}$, we have for any $n=sN+r$ with $0\leq r\leq N-1$ and $v\in \bigoplus_{i=1}^k \cO^i_{\omega}(T)\setminus\{0\}, u\in F^k_{\omega}(T)\setminus\{0\}$  
\begin{align}
    \frac{\|B^n_{\omega}v\|}{\|v\|}
& =
\frac{\|B^r_{\theta^{sN}\omega}B^{sN}_{\omega}v\|}{\|B^{sN}_{\omega}v\|}
\frac{\|B^{sN}_{\omega}v\|}{\|v\|}\notag  \geq 
\left(\frac{2}{3}\delta\right)^r \left(\frac{4}{3}\right)^s \frac{\|T^{sN}_{\omega}u\|}{\|u\|} \\[1ex]
& \geq 
\left(\frac{2}{3}\delta\right)^N \left(\frac{4}{3}\right)^{\frac{n}{N}}\frac{1}{\|T\|_{\infty}^r} \frac{\|T^{n}_{\omega}u\|}{\|u\|}\notag\\[1ex]
& \geq K_1 e^{\alpha_1 n} \frac{\|T^{n}_{\omega}u\|}{\|u\|},\label{Separation_Eq1}
\end{align}
where $K_1:=\left(\frac{2}{3}\delta\right)^N \frac{1}{(1+\|T\|_{\infty})^N}$, $\alpha_1:=\frac{\log\frac{4}{3}}{N}$. We now define the desired random compact operator $\widetilde T:\Omega\rightarrow \cK(\cH)$ as 
\begin{equation}\label{Desiredrandomoperator}
\widetilde T(\omega)v=
\left\{
\begin{array}{ll}
B(\omega) v    &  \hbox{ for } v\in \bigoplus_{i=1}^k \cO^i_{\omega}(T),\\[1ex]
T(\omega) v     &  \hbox{ for } v\in F^k_{\omega}. 
\end{array}
\right.
\end{equation}
To conclude the proof, we first verify that $\|T-\widetilde T\|_{\infty}<\eps$. For this purpose, let $\xi=v+u$ with $v\in \bigoplus_{i=1}^{k} \mathcal O^i_{\omega}(T)\setminus\{0\}, u\in F^{k}_{\omega}(T)\setminus\{0\}$ and $\|\xi\|=1$. By virtue of Claim (ii) of Remark \ref{Angleseparation_Remark}, we have $\|v\|\leq \frac{1}{\eta}$. This together with  \eqref{Estimate_Restriction} and \eqref{Desiredrandomoperator} implies that 
\begin{equation*}
\|T(\omega)\xi-\widetilde T(\omega)\xi\|
= 
\|T(\omega) v- B(\omega) v\|
\leq
\frac{\eps\eta}{2}\|v\|<\eps.
\end{equation*}
 Finally, we verify that the Oseledets-Ruelle decomposition corresponding to the first Lyapunov exponent of  $\widetilde T$ is dominated. For this purpose, let $v\in E_1(\omega)$ and $u \in \bigoplus_{i=2}^dE_i(\omega)\oplus F^k_{\omega}(T)$ with $\|v\|=\|u\|=1$. We write $u=u_1+u_2$ where $u_1\in \bigoplus_{i=2}^dE_i(\omega)$  and $u_2\in F^k_{\omega}(T)$. By Claim (ii) of  Remark \ref{Angleseparation_Remark}, we have
 $\|u_1\|,\|u_2\|\leq \frac{1}{\eta}$. Consequently, on one hand by \eqref{Integralseparation_Density_Apply} and \eqref{Desiredrandomoperator} we have \[
\|\widetilde T^n_{\omega}v\|
\geq Le^{\beta n} \frac{\|\widetilde T^n_{\omega}u_1\|}{\|u_1\|}
\geq \eta L e^{\beta n} \|\widetilde T^n_{\omega}u_1\|,
\]
and on the other hand by \eqref{Separation_Eq1}, \eqref{Desiredrandomoperator} we have 
\[
\|\widetilde T^n_{\omega}v\|
=
\|B^n_{\omega}v\|\geq K_1 e^{\alpha_1 n }\frac{\|T^{n}_{\omega}u_2\|}{\|u_2\|}
\geq 
\eta K_1 e^{\alpha_1 n} \|{\widetilde T}^{n}_{\omega}u_2\|.
\]
As a consequence, we derive that 
\begin{align}
\|\widetilde T^n_{\omega}v\|
& \geq \frac{1}{2}\left(
\eta L e^{\beta n} \|\widetilde T^n_{\omega}u_1\|
+
\eta K_1 e^{\alpha_1 n }\|{\widetilde T}^{n}_{\omega}u_2\|
\right)\notag\\[1ex]
& \geq 
\frac{\eta L + \eta K_1}{2} e^{\min(\beta,\alpha_1)n} \|{\widetilde T}^n_{\omega}u\|,\notag
\end{align}
which implies that the Oseledets-Ruelle decomposition corresponding to the first Lyapunov exponent of  $\widetilde T$ is dominated. The proof is complete.
\end{proof}
\begin{proof}[Proof of Theorem \ref{DensityIntegralSeparation_Hilbertspace}]
By Lemma \ref{Density_FirstLemma}, the assertion is proved for $k=1$. Suppose that the assertion holds for some $k\in \N$.  Let $T\in \mathcal D_k$ and $\eps>0$ be arbitrary. To conclude the proof, it is sufficient to find $\widetilde T\in \mathcal D_{k+1}$ such that $\|\widetilde T-T\|_{\infty}<\eps$.  By properties (i1), (i2), (i3) and  Definition \ref{Dominatedsplliting}, Definition \ref{IntegralSeparation} there exist  $K,\alpha>0$ and $\delta\in (0,1)$ such that for any $n\in\N$, $j\in\{1,\dots,i-1\}$
\begin{align}
\|T(\omega)v\|& \geq \delta \|v\|
\quad \hbox{ for } v\in \bigoplus_{i=1}^{k} \mathcal O^i_{\omega}(T)\setminus\{0\},\label{Angleseparation_Eq1_New}\\[1ex]
\frac{\|T_{\omega}^n v\|}{\|v\|}
& \geq K e^{\alpha n}
    \frac{\|T_{\omega}^n u\|}{\|u\|}
\quad \hbox{ for } v\in \bigoplus_{i=1}^{j} \mathcal O^i_{\omega}(T)\setminus\{0\}, u\in \bigoplus_{i=j+1}^{k} \mathcal O^i_{\omega}(T)\setminus\{0\},\label{Integralseparation_New}\\[1ex]
\frac{\|T_{\omega}^n v\|}{\|v\|}
& \geq K e^{\alpha n}
    \frac{\|T_{\omega}^n u\|}{\|u\|}
\quad \hbox{ for } v\in \bigoplus_{i=1}^{k} \mathcal O^i_{\omega}(T)\setminus\{0\}, u\in F^{k}_{\omega}(T)\setminus\{0\}.\label{Speedseparation_New}
\end{align}
Let $N\in\N$ such that $Ke^{\alpha N}\geq 2$. Choose $\eta$ as in \eqref{Constanteta} in Remark \ref{Angleseparation_Remark}. Thus, by Claim (i) of Remark \ref{Angleseparation_Remark} we have 
 \begin{equation}\label{Angleseparation_Eq_Apply1}
\left|1- \frac{\langle v, u \rangle}{\|v\|\|u\|}\right|
\geq  
\eta\qquad \hbox{for } v\in \bigoplus_{i=1}^{k} \mathcal O^i_{\omega}(T)\setminus\{0\}, u\in F^{k}_{\omega}(T)\setminus\{0\}. 
\end{equation}
Considering the restriction of $T(\omega)$ on $F^k_{\omega}(T)$ and in light of Lemma \ref{Density_FirstLemma}, there exists $S(\omega): F^k_{\omega}(T)\rightarrow F^k_{\theta\omega}(T)$ with an invariant decomposition $F^k_{\omega}(T)=\mathcal O^1_{\omega}(S)\oplus F_{\omega}^1(S)$, i.e. $S(\omega)\mathcal O^1_{\omega}(S)\subset \mathcal O^1_{\theta\omega}(S), S(\omega) F^1_{\omega}(S)\subset F^1_{\theta\omega}(S)$ such that the following properties hold:
\begin{itemize}
    \item [(p1)] $\|T(\omega)u-S(\omega)u\|< \eps \eta\|u\|$ for all $u\in F_{\omega}^k(T)$,
    \item [(p2)] $\|S^N_{\omega}u\|\leq \frac{3}{2}\|T^N_{\omega}u\|$ for all $u\in F_{\omega}^k(T)$,
    \item [(p3)]$\dim \mathcal O^1_{\omega}(S)=1$ and there exist $L,\beta>0$ such that 
\begin{equation}\label{PropertyP2}
\frac{\|S^n_{\omega}v\|}{\|v\|} \geq L e^{\beta n}   \frac{\|S^n_{\omega}u\|}{\|u\|} \quad\hbox{ for } v\in \mathcal O^1_{\omega}(S)\setminus\{0\}, u\in  F^1_{\omega}(S)\setminus\{0\}.
\end{equation}
\end{itemize}
Define $\widetilde T(\omega):\cH\rightarrow \cH$ as 
\begin{equation}\label{Newopearator}
\widetilde T(\omega)u
=
\left\{
\begin{array}{ll}
    T(\omega) u & \hbox{ for } u\in \bigoplus_{i=1}^k\mathcal O^i_{\omega}(T),  \\[1ex]
     S(\omega) u &  \hbox{ for } u\in F^{k}_{\omega}(T).
\end{array}
\right.
\end{equation}
To conclude the proof, we first show that $\|T-\widetilde T\|_{\infty}<\eps$ and then we show $\widetilde T\in \mathcal D_{k+1}$ by verifying (d1), (d2) and (d3) for $k+1$:\\

\noindent 
\emph{Verification of $\|T-\widetilde T\|_{\infty}<\eps$}: Let $\xi\in \cH$ be an arbitrary unit vector. We write $\xi=v+u$, where  $v\in \bigoplus_{i=1}^k\mathcal O^i_{\omega}(T)$ and $u\in F^{k}_{\omega}(T)$. By virtue of Claim (ii) of Remark \ref{Angleseparation_Remark}, we have $\|u\|, \|v\|\leq \frac{1}{\eta}$. Hence, by (p1) and \eqref{Newopearator} we have 
\[
\|\widetilde T(\omega) \xi- T(\omega)\xi\|
=
\|T(\omega)u-S(\omega)u\|
\leq \frac{\eps_1}{\eta}<\eps.  
\]

\noindent 
\emph{Verification of (i1) \& (i2)}: 
We first show that the decomposition 
\begin{equation}\label{DominatedDecomposition_I}
\cH= \left(\bigoplus_{i=1}^k\mathcal O^i_{\omega}(T) \oplus \mathcal O_{\omega}^1(S)\right) \oplus F^1_{\omega}(S)
\end{equation}
is dominated with respect to $\widetilde T$. For this purpose, let $v\in \bigoplus_{i=1}^k\mathcal O^i_{\omega}(T) \oplus \mathcal O_{\omega}^1(S)$ and $u\in F^1_{\omega}(S)$ with $\|v\|=\|u\|=1$. We write $v=v_1+v_2$, where $v_1\in  \bigoplus_{i=1}^k\mathcal O^i_{\omega}(T)$ and $v_2\in F^1_{\omega}(S)$. By \eqref{Newopearator} we have ${\widetilde T}^n_{\omega}v_1\in \bigoplus_{i=1}^k\mathcal O^i_{\theta^n\omega}(T)$ and ${\widetilde T}^n_{\omega}v_2\in F^k_{\theta^n\omega}(T)$. This together with \eqref{Angleseparation_Eq_Apply1} implies that 
\[
\left|1- \frac{\langle  {\widetilde T}^n_{\omega}v_1, {\widetilde T}^n_{\omega}v_2 \rangle}{\|{\widetilde T}^n_{\omega}v_1\|\|{\widetilde T}^n_{\omega}v_2\|}\right|
\geq  
\eta
\]
Being  analogous to the proof of Claim (ii) of Remark \ref{Angleseparation_Remark}, we have 
\[
\|{\widetilde T}^n_{\omega}v\|^2
=
\|{\widetilde T}^n_{\omega}v_1+{\widetilde T}^n_{\omega}v_2\|^2
\geq \eta \max\{\|{\widetilde T}^n_{\omega}v_1\|^2,\|{\widetilde T}^n_{\omega}v_2\|^2\},
\]
which together with \eqref{Speedseparation_New} and \eqref{PropertyP2} implies that
\[
\|{\widetilde T}^n_{\omega}v\|^2
\geq 
\eta\max\left\{
K^2 e^{2\alpha n} \|v_1\|^2\|{\widetilde T}^n_{\omega}u\|^2,  
K^2 e^{2\alpha n} \|v_2\|^2\|{\widetilde T}^n_{\omega}u\|^2
\right\}
\]
Since $\|v_1+v_2\|=1$ it follows that $\max\{\|v_1\|,\|v_2\|\}\geq \frac{1}{2}$. Thus, 
\[
\|{\widetilde T}^n_{\omega}v\|
\geq \frac{\sqrt{\eta} K e^{\alpha n}}{2} \|{\widetilde T}^n_{\omega}u\|,
\]
which yields that the decomposition \eqref{DominatedDecomposition_I} is dominated with respect to $\widetilde T$. Consequently, the Lyapunov exponent of $\widetilde T$ starting from any non-zero vector $v\in \left(\bigoplus_{i=1}^k\mathcal O^i_{\omega}(T) \oplus \mathcal O_{\omega}^1(S)\right)$ is strictly greater than the Lyapunov exponent of $\widetilde T$ starting from any non-zero vector $u\in F^1_{\omega}(S)$. Note that all subspaces $\mathcal O^i_{\omega}(T), i=1,\dots,k$ and $\mathcal O^1_{\omega}(S)$ are of dimension $1$ almost surely. By \eqref{Newopearator}, for $i=1,\dots,k$ we have 
\[
\lim_{n\to\infty}\frac{1}{n}\log\|{\widetilde T}^n_{\omega}v\|=\lambda_i(T) \qquad \hbox{ for } v\in \mathcal O^i(T)\setminus\{0\}.
\]
Thus, to prove the simplicity of the first $k+1$ Lyapunov exponent of $T$ it is sufficient to show that the decomposition $ \left(\bigoplus_{i=1}^k\mathcal O^i_{\omega}(T)\right) \oplus \mathcal O_{\omega}^1(S)$
is also dominated with respect to $\widetilde T$. For this purpose, let $v\in \bigoplus_{i=1}^k\mathcal O^i_{\omega}(T)$  and $u\in \mathcal O_{\omega}^1(S)$ with $\|v\|=\|u\|=1$. By \eqref{Newopearator}, we have 
\[
\|{\widetilde T}^N_{\omega}v\|=\|T^N_{\omega}v\|,\quad \|{\widetilde T}^N_{\omega}u\|=\|S^N_{\omega}u\|, 
\]
which together with  \eqref{Speedseparation_New}, the fact that $Ke^{\alpha N}\geq 2$ and (p2) implies that 
$
\|{\widetilde T}^N_{\omega}v\|
\geq \frac{4}{3} \|{\widetilde T}^N_{\omega}u\|$. Thus, by Remark \ref{Equivalence_DominatedSplitting} the decomposition $ \left(\bigoplus_{i=1}^k\mathcal O^i_{\omega}(T)\right) \oplus \mathcal O_{\omega}^1(S)$
is dominated with respect to $\widetilde T$. Consequently, (d1) and (d2) are verified.\\ 

\noindent 
\emph{Verification of (i3)}: By \eqref{Newopearator}, the restriction of $\widetilde T$ on $\bigoplus_{i=1}^k\mathcal O^i_{\omega}(T)$ is integrally separated. Furthermore,  the decomposition $ \left(\bigoplus_{i=1}^k\mathcal O^i_{\omega}(T)\right) \oplus \mathcal O_{\omega}^1(S)$
is dominated with respect to $\widetilde T$. Hence, by Remark \ref{Checkingintegralseparation} the restriction of $\widetilde T$ on  $ \left(\bigoplus_{i=1}^k\mathcal O^i_{\omega}(T)\right) \oplus \mathcal O_{\omega}^1(S)$ is integrally separated. The proof is complete.  
\end{proof}
\begin{theorem}\label{Integralseparation_Exteriorpower}
For any $T\in \mathcal D_k$, the Oseledets-Ruelle decomposition corresponding to the first Lyapunov exponent of $\Lambda^k T$ is dominated. 
\end{theorem}
\begin{proof}
Let $T\in\mathcal D_k$ be arbitrary.  By (d1) and (d3) of Definition \ref{IntegralSeparation_Hilbertspace}, there exist $K,\alpha,\delta>0$ such that 
\begin{align}
\|T(\omega)v\|& \geq \delta \|v\|
\quad \hbox{ for } v\in \bigoplus_{i=1}^{k} \mathcal O^i_{\omega}(T)\setminus\{0\},\label{Angleseparation_Eq1_New_Part2}\\[1ex]
\frac{\|T_{\omega}^n v\|}{\|v\|}
& \geq K e^{\alpha n}
    \frac{\|T_{\omega}^n u\|}{\|u\|}
\quad \hbox{ for } v\in \bigoplus_{i=1}^{j} \mathcal O^i_{\omega}(T)\setminus\{0\}, u\in \bigoplus_{i=j+1}^{k} \mathcal O^i_{\omega}(T)\setminus\{0\},\label{Integralseparation_New_Part2}\\[1ex]
\frac{\|T_{\omega}^n v\|}{\|v\|}
& \geq K e^{\alpha n}
    \frac{\|T_{\omega}^n u\|}{\|u\|}
\quad \hbox{ for } v\in \bigoplus_{i=1}^{k} \mathcal O^i_{\omega}(T)\setminus\{0\}, u\in F^{k}_{\omega}(T)\setminus\{0\}.\label{Speedseparation_New_Part2}
\end{align}
We now transform $T$ to a new random compact operator which still belongs to $\mathcal D_k$ and satisfies additionally that the Oseledets-Ruelle subspaces corresponding to the first $k$ Lyapunov exponents are orthogonal and also are orthogonal to the Oseledets-Ruelle subspaces corresponding to the remaining Lyapunov exponents. Let $e_{1}(\omega),\dots,e_{k}(\omega) $ be an orthonormal basis of $\bigoplus_{i=1}^k{\mathcal O}^i_{\omega}(T)$ and $\left(e_j(\omega)\right)_{j=k+1}^{\infty}$ an orthonomal basis of $\left(\bigoplus_{i=1}^k{\mathcal O}^i_{\omega}(T)\right)^{\perp}$. Thus, $\left(e_j(\omega)\right)_{j=1}^{\infty}$ is an orthonormal basis of $\cH$.
By (d2) of Definition \ref{IntegralSeparation_Hilbertspace}, $\dim \cO^1_{\omega}(T)=\dots=\cO^k_{\omega}(T)=1$.  For $i=1,\dots,k$, let $o_i(\omega)\in {\mathcal O}^i_{\omega}(T)$ be a random unit vector. Let $\left(f_j(\omega)\right)_{j=k+1}^{\infty}$ be an orthonormal basis of $F_{\omega}^k(T)$. Define a linear operator $P(\omega):\cH\rightarrow \cH$ as 
\begin{equation}\label{Transformation_Bounded}
P(\omega)e_i(\omega)
=
\left\{
\begin{array}{ll}
o_i(\omega) & \hbox{for } i=1,\dots,k,\\ [1ex]
f_i(\omega) & \hbox{for } i=k+1,\dots.
\end{array}
\right.
\end{equation}
By Remark \ref{Angleseparation_Remark},  the angle between subsapces $\bigoplus_{i=1}^k\mathcal O^i_{\omega}(T), F^k_{\omega}(T)$ and the angle between subspaces $\bigoplus_{i=1}^j\mathcal O^i_{\omega}(T), \bigoplus_{i=j+1}^k\mathcal O^i_{\omega}(T)$, where $j=1,\dots,k-1$, are uniformly separated from zero. Thus, $P(\omega)$ and its inverse $P(\omega)^{-1}$ are essentially bounded, i.e. there exists $M>0$ such that $\|P(\omega)\|, \|P(\omega)^{-1}\|\leq M$. Define $S:\Omega\rightarrow \cK(\cH)$ as 
\begin{equation}\label{Transformcompactoperator}
S(\omega)=P(\theta\omega)^{-1}T(\omega)P(\omega).
\end{equation}
Since $\|P(\omega)\|, \|P(\omega)^{-1}\|\leq M$ it follows that $S\in \mathcal D_k$. More concretely, the corresponding properties in \eqref{Angleseparation_Eq1_New_Part2}, \eqref{Integralseparation_New_Part2} and \eqref{Speedseparation_New_Part2} for $S$ are given, respectively, by 
\begin{align}
\|S(\omega)e_{i}(\omega)\|& \geq \frac{\delta}{M^2} 
\quad \hbox{ for } i=1,\dots,k,\label{Angleseparation_Normalform}\\[1ex]
\|S_{\omega}^n e_j(\omega)\|
& \geq \frac{K}{M^2} e^{\alpha n}
    \|S_{\omega}^n e_{j+1}(\omega)\|
\quad \hbox{ for } j=1,\dots,k-1,\label{Integralseparation_Normalform}\\[1ex]
\frac{\|S_{\omega}^n v\|}{\|v\|}
& \geq \frac{K}{M^2} e^{\alpha n}
    \frac{\|S_{\omega}^n u\|}{\|u\|}
\quad \hbox{ for } v\in  \mbox{Span}\{e_1(\omega),\dots,e_k(\omega)\}\setminus\{0\}, u\in F^{k}_{\omega}(T)\setminus\{0\}.\label{Speedseparation_Normalform}
\end{align}
Define 
\begin{align}
 \cE(\omega) & :=\mbox{Span}\{e_1(\omega)\wedge\dots\wedge e_k(\omega)\},\label{Stablespace}\\   
 \cF(\omega) & : =\mbox{Span}\big\{e_{i_1}(\omega)\wedge\dots\wedge e_{i_k}(\omega): (i_1,\dots, i_k)\in \mathcal I\},\label{Unstablespace}
\end{align}
where $\mathcal I:=\big\{\{i_1,\dots,i_k\}: i_1<\dots<i_k,\{i_1,\dots,i_k\}\not=\{1,\dots,k\}\big\}$. Then, as is mentioned in Subsection \ref{ExteriorPower}, $\Lambda^k\cH=\cE(\omega) \oplus  \cF(\omega)$  is a decomposition of $\Lambda^k\cH$. Furthermore,  by \eqref{Transformation_Bounded}, \eqref{Transformcompactoperator} this decomposition is invariant under $\Lambda^k S$. By \eqref{Transformcompactoperator} we have
\[
\Lambda^kS(\omega)=\Lambda^kP(\theta\omega)^{-1}\Lambda^kT(\omega)\Lambda^kP(\omega).
\]
Hence, to conclude the proof, it is sufficient to show that the decomposition $\Lambda^k\cH=\cE(\omega) \oplus  \cF(\omega)$  is dominated with respect to $\Lambda^k S$. Firstly, by \eqref{Angleseparation_Normalform} we have 
\[
\|\Lambda^kS(\omega) e_1(\omega)\wedge\dots\wedge e_k(\omega)\|\geq \frac{\delta^k}{M^{2k}},
\]
which verifies the property (d1) of Definition \ref{Dominatedsplliting} for the decomposition $\Lambda^k\cH=\cE(\omega) \oplus  \cF(\omega)$. Concerning the property (d2) of Definition \ref{Dominatedsplliting}, let $u\in \cF(\omega)$ be an arbitrary unit vector. We write $u=\sum_{\{i_1,\dots,i_k\}\in\mathcal I} \alpha_{i_1,\dots,i_k}e_{i_1}(\omega)\wedge\dots\wedge e_{i_k}(\omega)$. From $\|u\|=1$, we derive that $\sum_{\{i_1,\dots,i_k\}\in\mathcal I} \alpha_{i_1,\dots,i_k}^2=1$. By invariance of $e_i(\omega)$ under $S(\omega)$, i.e. $S(\omega)e_i(\omega)$ is parallel to $e_i(\theta\omega)$,  we have 
\begin{align}
\|\Lambda^k S^n_{\omega} u\|^2
& =
\sum_{\{i_1,\dots,i_k\}\in\mathcal I} \alpha_{i_1,\dots,i_k}^2\|S^n_{\omega} e_{i_1}(\omega)\wedge\dots S^n_{\omega} e_{i_k}(\omega)\|^2\notag\\
& = 
\sum_{\{i_1,\dots,i_k\}\in\mathcal I} \alpha_{i_1,\dots,i_k}^2\|S^n_{\omega} e_{i_1}(\omega)\|^2 \dots \|S^n_{\omega} e_{i_k}(\omega)\|^2\label{Exterior
_Eq1}.
\end{align}
By definition of the set of indices $\mathcal I$, we have $\{i_1,\dots,i_k\}\not=\{1,\dots,k\}$. Hence, the cardinality of the set $\{i_1,\dots,i_k\}\setminus\{1,\dots,k\}$ is at least $1$. Thus, by \eqref{Integralseparation_Normalform} and \eqref{Speedseparation_Normalform}
\begin{align}
\|S^n_{\omega} e_{i_1}(\omega)\| \dots \|S^n_{\omega} e_{i_k}(\omega)\|
& \leq \left(\frac{K}{M^2} e^{\alpha n}\right)^{\mbox{Card} \{i_1,\dots,i_k\}\setminus\{1,\dots,k\}} \|S^n_{\omega}e_1(\omega)\|\dots \|S^n_{\omega}e_k(\omega)\|\notag\\[1ex]
& \leq \widetilde K e^{\alpha n}\|S^n_{\omega}e_1(\omega)\|\dots \|S^n_{\omega}e_k(\omega)\|\notag, 
\end{align}
where $\widetilde K:=\min\left\{\left(\frac{K}{M^2}\right)^i: i=1,\dots,k\right\}$. This, together with \eqref{Exterior
_Eq1}, implies that 
\begin{align}
\|\Lambda^k S^n_{\omega} u\|^2
& \leq 
\sum_{\{i_1,\dots,i_k\}\in\mathcal I} \alpha_{i_1,\dots,i_k}^2 {\widetilde K}^2 e^{2\alpha n} \|S^n_{\omega}e_1(\omega)\|^2\dots \|S^n_{\omega}e_k(\omega)\|^2\notag\\[1ex]
& \leq 
{\widetilde K}^2 e^{2\alpha n} \|\Lambda^kS_{\omega}^n e_1(\omega)\wedge \dots \wedge e_k(\omega)\|^2,\notag
\end{align}
which verifies the property (d2) of Definition \ref{Dominatedsplliting} for the decomposition $\Lambda^k\cH=\cE(\omega) \oplus  \cF(\omega)$. The proof is complete.
\end{proof}
\subsection{Random compact operators having dominated splitting also processes closed proper convex cones}
The main ingredient in the proof of Theorem B is the following preparatory result on a connection between random compact operators processing closed proper convex cones and dominated splitting. More precisely, in the first direction we extend the work in \cite{Doan17} by showing that random compact operators whose the Oseledets-Ruelle decomposition corresponding to the first Lyapunov exponent are dominated splitting also possess closed proper convex cones. The same spirit of these results can be found in \cite{Lian}, when the author also studied a connection between dominated splitting and the existence of an invariant cone of a random compact linear operator, but in the non-uniformity sense. 

\begin{proposition}\label{firstproposition}
Let $T\in \cK_{\infty}(\Omega;\mathcal H)$ be arbitrary. Suppose that the Oseledets-Ruelle of the first Lyapunov exponent of $T$ is simple, i.e., $\dim \cO^1_{\omega}=1$. Suppose further that the decomposition $\cH= \mathcal O^1_{\omega}(T)\oplus F_{\omega}^1(T)$ corresponding to the first Lyapunov exponent of $T$ is dominated. Let $e(\omega)$ be a random unit vector of the subspace $\cO^1(\omega)$.    For each $\omega\in\Omega$, let \begin{equation}\label{Invariantcone_FirstLE}
    \cC_{\omega}:=\left\{ u+\gamma e(\omega): u\in F_{\omega}^1(T)\quad  \hbox{ with }\quad  \|u\|\leq \frac{\gamma}{2} \right\}.
    \end{equation}
Then, $(\mathcal C_{\omega})_{\omega\in\Omega}$ is a family of closed proper convex cones  $(\mathcal C_{\omega})_{\omega\in \Omega}$  satisfying condition (C) with $e(\omega)$ as an interior point. Furthermore, there exists a prime number $N$ such that $T^N_{\omega}\mathcal C_{\omega}\subset \mathcal C_{\theta^N \omega}\cup (-\mathcal C_{\theta^N\omega})$ and there exists $R<\infty$ such that for all $v\in  \mathcal C_{\omega}\setminus\{0\}$ with $\|v\|=1$
\begin{equation}\label{Interior_Aim}
\beta_{\mathcal C_{\theta^N\omega}}(T^N_{\omega}v, e(\theta^N\omega))\leq R\qquad\hbox{for } \omega\in \Omega.
\end{equation}

\end{proposition}
\begin{proof}
By Definition \ref{Dominatedsplliting}, there exist $K,\alpha,\delta>0$ such that 
\begin{equation}\label{Separationzero}
\|T(\omega)e(\omega)\|\geq \delta 
\end{equation}
and 
\begin{equation}\label{DominatedDirection_Eq1}
\|T^n_{\omega}e(\omega)\|\geq   Ke^{\alpha n}  \frac{\|T_{\omega}^n u\|}{\|u\|} \quad \hbox{ for all } n\in\N, u\in F^1_{\omega}(T).
\end{equation}
Thanks to Example \eqref{ConeExample} (iii), the family of closed proper convex cones $(\cC_{\omega})_{\omega\in\Omega}$ defined as in \eqref{Invariantcone_FirstLE} satisfies the condition (C). Choose and fix a prime number $N$ such that $Ke^{\alpha N}\geq 2 $. We now verify that $T^N_{\omega}\mathcal C_{\omega}\subset \mathcal C_{\theta^N \omega}\cup (-\mathcal C_{\theta^N\omega})$. To see this,  let $v\in \cC_{\omega}$ with $\|v\|=1$. We write $v=u+\gamma e(\omega)\in\cC_{\omega}$ with $u\in F^1_{\omega}(T)$  and $\gamma\in\R$. By invariance of the random subspace $\mathcal O^1_{\omega}(T)$, we can present 
\begin{equation}\label{Invariance_Application_01}
T^n_{\omega}(u+\gamma e(\omega))
=T^n_{\omega} u + \gamma \beta(n,\omega) e(\theta^n\omega),
\end{equation}
where $\beta(n,\omega)$ satisfies $|\beta(n,\omega)|=\|T^n_{\omega}e(\omega)\|$. By \eqref{DominatedDirection_Eq1} we have 
\begin{equation}\label{EstimateCone_01}
\|T^N_{\omega} u\| \leq \frac{\|u\|}{Ke^{\alpha N}}\|T^N_{\omega}e(\omega)\|
\leq 
\frac{\gamma}{2Ke^{\alpha N}}|\beta(n,\omega)|\leq \frac{|\gamma \beta(n,\omega)|}{4},
\end{equation}
which together with \eqref{Invariantcone_FirstLE} implies that $T^N_{\omega}\cC_{\omega} \subset \mathcal C_{\theta^N \omega}\cup (-\mathcal C_{\theta^N\omega})$. To conclude the proof, let $R:=\delta^N$, where $\delta$ is given to satisfy \eqref{Separationzero},  and we verify \eqref{Interior_Aim}. To see this, since  $\|v\|=1$ and $\|u\|\leq \frac{\gamma}{2}$ it follows that $\gamma\geq \frac{2}{3}$. This, together with \eqref{Separationzero}, implies that 
$|\gamma \beta(N,\omega)|=|\gamma|\|T^N_{\omega}e(\omega)\|\geq \frac{2}{3}\delta^N$. Furthermore, by \eqref{Invariance_Application_01} we have 
\[
R T^N_{\omega}(u+\gamma e(\omega))-e(\theta^N\omega)
= -R T^N_{\omega}u+ (R\gamma\beta(N,\omega)-1)e(\theta^N\omega).
\]
Then,  we use \eqref{EstimateCone_01} to gain the following estimate
\begin{align}
\frac{|R\gamma\beta(N,\omega)-1|}{2} -\|R T^N_{\omega}u\|
& \geq 
\frac{|R\gamma\beta(N,\omega)-1|}{2} -\frac{|R\gamma\beta(N,\omega)|}{4}\notag\\[1ex]
& \geq 
\frac{|R\gamma\beta(N,\omega)|}{4}-\frac{1}{2}.\notag
\end{align}
Thus, $R T^N_{\omega}(u+\gamma e(\omega))-e(\theta^N\omega)
\in  \mathcal C_{\theta^N \omega}\cup (-\mathcal C_{\theta^N\omega})$ and therefore $\beta_{\mathcal C_{\theta^N\omega}}(T^N_{\omega}v, e(\theta^N\omega))\leq R$ and \eqref{Interior_Aim} is proved. The proof is complete.
\end{proof}
\subsection{Proof of Theorem B}
\begin{proof}[Proof of Theorem B] Let $T\in \mathcal D_k$ be arbitrary. Thus, the Oseledets-Ruelle subspaces corresponding to the first $k$ Lyapunov exponent of $T$ are simple, i.e.
\[
\dim \mathcal O^1_{\omega}(T)=\dots=\dim \mathcal O^k_{\omega}(T)=1.
\]
Using Theorem \ref{Connection}, we obtain that for $i=1,\dots,k+1$
\begin{equation}\label{Connection_Eq1}
\lambda_1(T)+\dots+\lambda_i(T)=\ell_i(T)=\kappa(\Lambda^iT)=\lim_{n\to\infty}\frac{1}{n}\log\|\Lambda^iT^n_{\omega}\|.
\end{equation}
By virtue of Theorem \ref{Integralseparation_Exteriorpower}, the Oseledets-Ruelle decompositions corresponding to the first Lyapunov exponents of $T,\Lambda T,\dots,\Lambda^k T$, respectively, are dominated. Hence, in light of Proposition \ref{firstproposition}, $T,\Lambda T,\dots,\Lambda^k T$ have closed proper convex cones satisfying condition (C). Using  Proposition \ref{GeneralizedKrein-Rutman}, the top Lyapunov exponent function $\kappa(\cdot)$ is analytic at $T, \Lambda T,\dots,\Lambda^k T$. Thus, there exists $\eps_1(T)>0$ such that for any  $S\in \mathcal K_{\infty}(\Omega;\mathcal H)$ with $\|S-T\|_{\infty}\leq \eps_1(T)$ we have 
\[
\kappa(\Lambda^{i+1}S)-\kappa(\Lambda^{i}S)<\kappa(\Lambda^{i}S)-\kappa(\Lambda^{i-1}S)\quad \hbox{ for } i=1,\dots,k.
\]
This together with Theorem \ref{Connection} implies that for any  $S\in \mathcal K_{\infty}(\Omega;\mathcal H)$ with $\|S-T\|_{\infty}\leq \eps_1(T)$ 
the Oseledets-Ruelle subspaces corresponding to the first $k$ Lyapunov exponent of $S$ are simple and 
\begin{equation}\label{Connection_Eq2}
\kappa(\Lambda^i S)=\lambda_1(S)+\dots+\lambda_i(S)\quad \hbox{ for } i=1,\dots,k.    
\end{equation}
Using  Proposition \ref{GeneralizedKrein-Rutman}, there exists $\eps(T)\in (0,\eps_1(T))$ such that for any  $S\in \mathcal K_{\infty}(\Omega;\mathcal H)$ with $\|S-T\|_{\infty}\leq \eps(T)$ the top Lyapunov exponent function $\kappa(\cdot)$ is analytic at $S, \Lambda S,\dots,\Lambda^k S$. This together with \eqref{Connection_Eq2} implies that for any  $S\in \mathcal K_{\infty}(\Omega;\mathcal H)$ with $\|S-T\|_{\infty}\leq \eps(T)$  the map $\lambda_i(\cdot)$ is analytic at $S$, where $i=1,\dots,k$. Hence, each element of  $\mathcal G_k:=\bigcup_{T\in\mathcal D_k} \{S\in \mathcal K_{\infty}(\Omega;\mathcal H): \|S-T\|<\eps(T)\}$ satisfies both properties (i) and (ii). Obviously, $\mathcal G_k$ is open and by density of $\mathcal D_k$ (see Theorem \ref{DensityIntegralSeparation_Hilbertspace}) the set $\mathcal G_k$ is also dense in $K_{\infty}(\Omega;\mathcal H)$. The proof is complete.
\end{proof}

\section{Appendix}\label{appendixA}
\subsection{Exterior power}\label{ExteriorPower}
For any $k\in\N$, let $\Lambda^k\cH$ denote the $k$-th exterior power of $\cH$. The space $\Lambda^k\cH$ can be identified with the set of formal expressions $\sum_{i=1}^m c_i (u_1^{(i)}\wedge\dots \wedge u_k^{(i)})$ with $m\in\N$, $c_i\in\R$ and $u_j^{(i)}\in\cH$ if we do computations with the following conventions:
\begin{itemize}
\item [1.] {\em Addition}:
\[
u_1\wedge\dots \wedge (u_j+\widehat u_j)\wedge\dots\wedge u_k =   u_1\wedge\dots \wedge u_j\wedge\dots\wedge u_k+u_1\wedge\dots \wedge \widehat u_j\wedge\dots\wedge u_k,
\]
\item [2.] {\em Scalar multiplication}: $u_1\wedge\dots \wedge c u_j\wedge\dots \wedge u_k=c u_1\wedge\dots\wedge u_j\wedge\dots\wedge u_k$,

\item [3.] for any permutation $\pi$ of $\{1,\dots,k\}$
\[
u_{\pi(1)}\wedge\dots\wedge u_{\pi (k)}
=\mbox{sign}(\pi) u_1\wedge\dots\wedge u_k.
\]
\end{itemize}
The inner product on $\cH$ induces an inner product $\langle\cdot,\cdot\rangle$ on $\Lambda^k{\cH}$ via
\[
\langle u_1\wedge\dots\wedge u_k, v_1\wedge \dots\wedge v_k
\rangle
:=\det (\langle u_i,v_j\rangle )_{k\times k}.
\]
Let $(e_n)_{n=1}^{\infty}$ be an orthonormal basic of $\mathcal H$. Then,  the set $(e_{i_1}\wedge\dots\wedge e_{i_k})_{i_1<\dots<i_k}$ is an orthonormal basic of $\Lambda^k\mathcal H$, see e.g. \cite{Temam}.

A random compact operator map $T:\Omega\rightarrow \cK(\cH)$ generates a random compact operator  $\Lambda^k T:\Omega\rightarrow \cK(\Lambda^k\cH)$ defined by
\[
\Lambda^kT(\omega)(u_1\wedge\dots\wedge u_k):= T(\omega)u_1\wedge\dots\wedge T(\omega)u_k.
\]
The following result provides a relation between the top Lyapunov exponents of random compact operators $(\Lambda^kT)_{k=1}^{\infty}$ and the Lyapunov exponents of the random compact operator $T$, see \cite[Theorem 1.3]{Blumenthal16}.
\begin{theorem}\label{Connection}
 Let $T\in \mathcal K_{\infty}(\Omega;\mathcal H)$. Then, there exists a measurable set $
	\widehat \Omega$ being of full probability and invariant under $\theta$ such that for any $q\in \N$ the limit 
	\[
	\ell_q(T):=\lim_{n\to\infty}\frac{1}{n}\log\|\Lambda^qT^n_{\omega}\|
	\]
exists and is constant over $\omega\in\widehat\Omega$. Define the sequence $(K_q(T))_{q\geq 1}$ by $K_1(T)=\ell_1(T)$ and $K_q(T)=\ell_q(T)-\ell_{q-1}(T)$ for $q>1$. Then, $K_1(T)\geq K_2(T)\geq \dots$ and the Lyapunov exponents $\lambda_i(T)$ are distinct values of the decreasing  sequence $(K_q(T))_{q\geq 1}$ with finite multiplicities $m_i(T)$. Furthermore, for any $i$ the multiplicity $m_i(T)$ is the dimension of the Oseledets-Ruelle subspace $\mathcal O_i(T)$ corresponding to the Lyapunov exponent $\lambda_i(T)$.
\end{theorem}
\subsection{Measurable selection of random subspaces}
We collect in this subsection some technical results concerning random compact operators. For this purpose, we recall the following measurable selection theorem, see \cite[Theorem 4.2]{Wagner1977}.
\begin{theorem}[Measurable selection theorem]\label{Measurableselectiontheorem}
Let $B(\cH)$ denote the set of closed subsets of $\cH$. Consider a map $A:\Omega\rightarrow B(\cH)$. Then, the following statements are equivalent:
\begin{itemize}
    \item [(i)] $A$ is a random closed set, i.e. for all $x\in\cH$ the map $\omega\mapsto d(x,A(\omega))$\footnote{The distance $d(x,E):=\inf_{u\in E}\|x-u\|$ for a closed subset $E$.} is measurable.
    \item [(ii)] For all open sets $\cO\in\cH$, the set $\{\omega\in\Omega: A(\omega)\cap \cO\not=\emptyset\}$ is measurable.
    \item [(iii)] There exists a sequence $(a_n)_{n\in \N}$ if measurable maps $a_n:\Omega\rightarrow \cH$ such that $A(\omega)=\mbox{Closure}\{a_n(\omega): n\in\N\}$ for all $\omega\in\Omega$.
\end{itemize}

\end{theorem}

\begin{proposition}\label{Completeness}
Let $T\in\cK_{\infty}(\Omega,\cH)$. Then the following statements hold:
\begin{itemize}
    \item [(i)] There exists a unit random vector $f:\Omega\rightarrow \cH$ such that $\|T(\omega)f(\omega)\|=\|T(\omega)\|$ and  $\langle T(\omega) f(\omega), T(\omega) u \rangle=0$ for $u\in f(\omega)^{\perp}$.
    \item [(ii)] For any $\eps>0$ there exists $S \in \mathcal K_{\infty}(\Omega;\mathcal H)$ such that $\|T-S\|_{\infty}\leq \eps$ and $S(\omega)$ is a finite-rank linear operator for $\omega\in\Omega $.
    \item [(iii)] For any $\eps>0$ there exists $R\in \mathcal K_{\infty}(\Omega;\mathcal H)$ such that $\|T-R\|_{\infty}\leq \eps$ and $R(\omega)$ is bijective for $\omega\in\Omega$
\end{itemize}
\end{proposition}
\begin{proof}
(i) For each $\omega\in\Omega$, let $L(\omega)=T(\omega)^* T(\omega)$ where $T(\omega)^*$ denotes the adjoint operator of $T(\omega)$. Then, $L(\omega)$ is a self-adjoint compact operator. Let
\begin{equation}\label{Firsteigensubspaces}
U(\omega):=\Big\{v\in\cH:L(\omega)v  =\|L(\omega)\|v\Big\}.
\end{equation}
Thus, $U(\omega)$ is a non-trivial linear subspace of $\cH$ of finite dimension. From $\|L(\omega)\|=\|T(\omega)\|^2$ we derive that $\|T(\omega)v\|^2= \langle L(\omega)v,v\rangle\|=\|T(\omega)\|^2$ for all $v\in U(\omega)$ with $\|v\|=1$. Furthermore, for any $v\in U(\omega)$ and $u\in v^{\perp}$ we have 
\[
\langle T(\omega) v, T(\omega) u\rangle
=
\langle L(\omega) v,  u\rangle=0.
\]
Hence, to conclude the proof of this part it is sufficient to show that the set-valued map $\omega\mapsto V(\omega)$, where  $V(\omega):=U(\omega)\cap \{v\in\cH:\|v\|=1\}$, satisfies (i) of Theorem \ref{Measurableselectiontheorem}, i.e. 
\begin{equation}\label{Aim1}
\omega\mapsto d(x,V(\omega))\quad \hbox{ is measurable for any } x\in\cH.   
\end{equation}
To do this, choose and fix $x\in\cH$. Using the Shur Theorem for self-adjoint compact operator $L(\omega)$, see e.g. \cite[Theorem VII.4]{Retherford}, we have 
\begin{equation}\label{Measurability_Eq01}
V(\omega)=\Big\{v\in\cH: \|v\|=1 \hbox{ and } \|L(\omega) v\|=\|L(\omega)\|\Big\}.    
\end{equation}
For each $n\in\N$ with $n\geq 2$ we define 
\begin{equation}\label{Approximatesubspaces}
V_n(\omega):=\left\{v\in\cH: \|v\|=1 \hbox{ and } \|L(\omega)v\|>\left(1-\frac{1}{n}\right) \|L(\omega)\|\right\}.
\end{equation}
We now show that 
\begin{equation}\label{Approximatedsequence}
d(x,V(\omega))=\lim_{n\to\infty} d(x,V_n(\omega))\qquad \hbox{ for } \omega\in\Omega.
\end{equation}
Pick an arbitrary $\omega\in\Omega$. Obviously, $(V_n(\omega))_{n\geq 2}$ is a decreasing sequence of sets and $\bigcap_{n\geq 2} V_n(\omega)=V(\omega)$, then $(d(x,V_n(\omega)))_{n\geq 2}$ is an increasing sequence and $d(x,V_n(\omega))\leq d(x,V(\omega))$ for all $n\geq 2$. Consequently, $\lim_{n\to\infty} d(x,V_n(\omega))\leq V(x,V(\omega))$. For the inverse inequality, by definition of $U(\omega)$ as in \eqref{Firsteigensubspaces}
\begin{equation}\label{Gamma_Constant}
\gamma:=\|L(\omega)|_{U(\omega)^{\perp}}\|<\|L(\omega)\|.
\end{equation}
For any $n\geq 2$ and $v\in V_n(\omega)$, we can write $v=u+u^{\perp}$ for $u\in U(\omega)$ and $u^{\perp}\in U(\omega)^{\perp}$. Then,
\begin{align}
\|x-v\|
& \geq \|x-\frac{u}{\|u\|}\|-\|u-\frac{u}{\|u\|}\| -\|u^{\perp}\|
\notag\\
& \geq \mbox{dist}(x,U(\omega))-(1-\|u\|+\|u^{\perp}\|,\notag
\end{align}
which together with $\|u\|+\|u^{\perp}\|\geq \|u+u^{\perp}\|\geq 1$ implies that 
\begin{equation}\label{Firstdirection}
 \|x-v\|
 \geq 
 \mbox{dist}(x,U(\omega))-2\|u^{\perp}\|.
\end{equation}
To estimate $\|u^{\perp}\|$, from $\|v\|=1$ we derive that $\|u\|^2+ \|u^{\perp}\|^2=1$. Since $v\in V_n(\omega)$ and $u\in U(\omega)$ it follows that 
\begin{align}
\left(1-\frac{1}{n}\right)\|L(\omega)\| 
& \leq 
\|L(\omega)v\|\notag\\
& \leq \|L(\omega)\| \|u\| +\|L(\omega) u^{\perp}\|\notag\\
&\leq 
\|L(\omega)\| \|u\| +\gamma \|u^{\perp}\|,\notag
\end{align}
which implies that $\|u\|+\frac{\gamma}{\|L(\omega)\|}\|u^{\perp}\|\geq 1-\frac{1}{n}$. Hence,
\[
\|u\|^2+ \frac{\gamma^2}{\|L(\omega)\|^2}\|u^{\perp}\|^2\geq \left(1-\frac{1}{n}\right)^2\geq 1-\frac{2}{n}.
\]
Then, from $\|u\|^2+ \|u^{\perp}\|^2=1$ we derive that  $\|u^{\perp}\|^2\leq \frac{2}{n}\frac{\|L(\omega)\|^2}{\|L(\omega)\|^2-\gamma^2}$. Thus, by \eqref{Firstdirection} we arrive at 
\[
\|x-v\|\geq \mbox{dist}(x,V(\omega))
-\frac{2}{n}\frac{\|L(\omega)\|}{\|L(\omega)\|-\gamma}
\quad \hbox{ for all } v\in V_n(\omega).
\]
Consequently, $\mbox{dist}(x,V_n(\omega))\geq \mbox{dist}(x,V(\omega))
-\frac{2}{n}\frac{\|L(\omega)\|}{\|L(\omega)\|-\gamma}$ and therefore  $\mbox{dist}(x,V(\omega))=\lim_{n\to\infty}\mbox{dist}(x,V_n(\omega))$. To prove \eqref{Aim1}, it is sufficient to show that for an arbitrary but fixed $n\in\N$ the map $\omega\mapsto \mbox{dist}(x,V_n(\omega))$ is measurable. By using Theorem \ref{Measurableselectiontheorem}, we prove this fact by verifying that for an arbitrary open set 
$\cO$ of $\cH$, the set $V_{\cO;n}:=\{\omega\in\Omega: V_n(\omega)\cap \cO\not=\emptyset\}$ is measurable. Let $(f_i)_{i\geq 1}$ be a countable dense set of $\cO$. We can assume that $f_i\not=0$ for all $i$. By \eqref{Approximatesubspaces} and density of $(f_i) $ in $\mathcal O$ we have 
\begin{equation*}
V_{\cO;n}=\bigcap_{i\geq 1}\left\{
\omega\in\Omega: \frac{\|L(\omega)f_i\|}{\|f_i\|}>\left(1-\frac{1}{n}\right)\|L(\omega)\|
\right\},
\end{equation*}
which proves that $V_{\cO;n}$ is measurable, and the proof of this part is complete.

(ii) By virtue of (i), we can inductively construct a sequence of linear subspaces $(\cH_{j}(\omega))_{j=0}^{\infty}$ and a sequence of  random unit vectors $(f_i(\omega))_{i=1}^{\infty}$ such that for $i=1\dots$ we have 
\begin{align*}
\cH_{0}(\omega)&:=\cH \hbox{ and } \cH_i(\omega):=\{f_1(\omega),\dots, f_i(\omega)\}^{\perp},\\[1ex]
L(\omega)f_i(\omega)
&=
\|L(\omega)|_{\cH_{i-1}(\omega)}\|f_i(\omega).
\end{align*}
Furthermore, both linear subspaces $\mbox{Span}\{f_1(\omega),\dots, f_i(\omega)\}$ and $\cH_i(\omega)$ are invariant under $L(\omega)$. Let $\lambda_i(\omega):=\|L(\omega)|_{\cH_{i-1}(\omega)}\|$ for $i=1,\dots$. By measurability of  $(f_i(\omega))_{i=1}^{\infty}$ and separability of $\cH$ the map $\lambda_i$ is measurable for any $i$. Since $S(\omega)$ is a compact linear operator it follows that $\lim_{i\to\infty}\lambda_i(\omega)=0$. Then, for any $\eps>0$ there exists a random maps $k(\omega)\in \mathbb N$ such that $\|L(\omega)|_{\cH_{k(\omega)}(\omega)}\|<\eps^2$. Now, we define the linear operator $S(\omega)$ as 
\begin{equation}\label{FiniteRank}
S(\omega)v
:=
\left\{
\begin{array}{ll}
    T(\omega)v & \hbox{ for } v\in \mbox{Span}\{f_1(\omega),\dots, f_{k(\omega)}(\omega)\}, \\[1ex]
     0 & \hbox{ for } v\in \cH_{k(\omega)}(\omega).
\end{array}
\right.
\end{equation}
Then, $S(\omega)$ is a finite-rank linear operator. By \eqref{FiniteRank}, we have  $S(\omega)v=T(\omega)v$ for  $v\in \mbox{Span}\{f_1(\omega),\dots, f_{k(\omega)}(\omega)\}$. Thus,
\[
\|S(\omega)-T(\omega)\|^2=\sup_{v\in \cH_{k(\omega)}(\omega)}\frac{\|T(\omega)v\|^2}{\|v\|^2}=
\sup_{v\in \cH_{k(\omega)}(\omega)}\frac{\left\langle L(\omega)v,v\right\rangle}{\|v\|^2}\leq \eps^2,
\]
which completes the proof of this part.

(iii) By (ii), a random compact linear operator can be approximated by a sequence of random finite-rank operators. Then, it is sufficient to start with a random finite-rank operator $T\in\mathcal K(\mathcal H)$ and an arbitrary but fixed $\eps>0$. For each $\omega$, using the Schmidt Decomposition Theorem, see e.g. \cite[Theorem VIII. 8]{Retherford}, there are orthonormal basics $(x_n(\omega))_{n=1}^{\infty}$ and $(y_n(\omega))_{n=1}^{\infty}$ such that 
\[
T (\omega) f= \sum_{n=1}^{k}\alpha_n \langle f, x_n(\omega)\rangle y_n(\omega),\qquad f\in\mathcal H,
\] 
where $k\in\N$ and $\alpha_1,\dots,\alpha_k$ are non-zero real numbers. We now choose arbitrary non-zero real numbers $\alpha_{k+1},\alpha_{k+2},\dots$ such that $\sum_{j=k+1}^{\infty}\alpha_j^2<\eps^2$. Then, the operator defined by 
\[
R(\omega) f=\sum_{n=1}^{\infty}\alpha_n \langle f, x_n(\omega)\rangle y_n(\omega),\qquad f\in\mathcal H,
\] 
satisfies that $R$ is a bijective compact linear operator and $\|T(\omega)-R(\omega)\|<\eps$. The proof is complete.
\end{proof}
\subsection{Halmos-Rokhlin's lemma}
We recall the following tool from ergodic theory, which is derived as an application of Halmos-Rokhlin's lemma, see e.g. \cite{CFS_1982}.
\begin{theorem}\label{HR_Lemma}
Assume that $\theta$ is an ergodic invertible transformation of the non-atomic Lebesgue space $(\Omega,\mathcal F,\mathbb P)$. Then for any $U\in\cF$ with $\mP(U)>0$ and any $n > 0$ there exists a measurable set $V\subset U$ such that $\mP(V)>0$ and the sets $\theta^{-n}V,\dots,\theta^{-1}V,V,\theta V,\dots,\theta^n V$ are pairwise disjoint.
\end{theorem}
%

\end{document}